\documentclass[letterpaper,11pt,reqno,tbtags]{amsart} 
\usepackage[portrait,margin=3cm]{geometry}  
\usepackage{mathrsfs}
\usepackage{amsmath}
\usepackage[
    colorlinks=true,%
    breaklinks,
    hyperindex,%
    plainpages=false,%
    bookmarksopen=false,%
    linkcolor=blue,%
    anchorcolor=blue,%
    citecolor=blue,%
    filecolor=black,%
    menucolor=black,%
    pagecolor=black,%
    urlcolor=blue,%
    pdfview=FitH,
    pdfstartview=FitH,
    hyperfigures,
    bookmarksnumbered%
  ]{hyperref} 
\usepackage{amsmath,amssymb,amsthm,amsfonts,amsbsy,latexsym,dsfont,tikz,multirow}
\usepackage[skip=5pt]{caption}

\usepackage{algorithm2e}
\usepackage{enumerate}
\newenvironment{enumeratei}{\begin{enumerate}[\upshape (i)]}{\end{enumerate}}

\newenvironment{enumeraten}{\begin{enumerate}[\upshape 1.]}{\end{enumerate}}

\newcommand{\matnorm}[1]{{\left\vert\kern-0.25ex\left\vert\kern-0.25ex\left\vert #1 
    \right\vert\kern-0.25ex\right\vert\kern-0.25ex\right\vert}}

\newcommand{\Var}[0]{\text{Var}}

\newcommand{\diag}[0]{\text{diag}}

\newcommand{\Id}[0]{\text{Id}}

\newcommand{\Prob}[0]{\mathds{P}}

\theoremstyle{remark}
\newtheorem{thm}{Theorem}[section]
\newtheorem{lem}[thm]{Lemma}
\newtheorem{cor}[thm]{Corollary}

\newtheorem{defn}[thm]{Definition}
\newtheorem{rem}[thm]{Remark}

\renewcommand{\leq}{\leqslant} 
\renewcommand{\geq}{\geqslant}

\def\qed{ \hfill $\blacksquare$}  
\newcommand{\rd}{\mathrm{d}}

\newcommand{\mc}[1]{\mathcal{#1}}
\newcommand{\cA}{\mathcal{A}}
\newcommand{\cE}{\mathcal{E}}
\newcommand{\cG}{\mathcal{G}}

\newcommand{\cM}{\mathcal{M}}
\newcommand{\cP}{\mathcal{P}}

\newcommand{\cW}{\mathcal{W}}
  
\newcommand{\vone}{\mathbf{1}}

\newcommand{\vG}{\mathbf{G}}

\newcommand{\fh}{\mathfrak{h}}

\newcommand{\bb}[1]{\mathbb{#1}}

\newcommand{\bH}{\mathbb{H}}

\newcommand{\bN}{\mathbb{N}}
\newcommand{\bR}{\mathbb{R}}
\newcommand{\bS}{\mathbb{S}}

\newcommand{\sC}{\mathscr{C}}
\newcommand{\sD}{\mathscr{D}}

\DeclareMathOperator{\E}{\mathds{E}}

\DeclareMathOperator{\argmax}{argmax}

\DeclareMathOperator{\tr}{tr} 
  
\begin{document}

\title[Diffusion $K$-means and SDP relaxations]{Diffusion $K$-means clustering on manifolds: provable exact recovery via semidefinite relaxations}

\author[Xiaohui Chen]{Xiaohui Chen}
\address{\newline Department of Statistics\newline
University of Illinois at Urbana-Champaign\newline
725 S. Wright Street, Champaign, IL 61820\newline
{\it E-mail}: \href{mailto:xhchen@illinois.edu}{\tt xhchen@illinois.edu}\newline 
{\it URL}: \href{http://publish.illinois.edu/xiaohuichen/}{\tt http://publish.illinois.edu/xiaohuichen/}
}
\author[Yun Yang]{Yun Yang}
\address{\newline Department of Statistics\newline
University of Illinois at Urbana-Champaign\newline
725 S. Wright Street, Champaign, IL 61820\newline
{\it E-mail}: \href{mailto:yy84@illinois.edu}{\tt yy84@illinois.edu}\newline 
{\it URL}: \href{https://sites.google.com/site/yunyangstat/}{\tt https://sites.google.com/site/yunyangstat/}
}

\date{First arXiv version: March 11, 2019. This version: \today}
\thanks{X. Chen's research is supported in part by NSF DMS-1404891, NSF CAREER Award DMS-1752614, UIUC Research Board Awards (RB17092, RB18099), and a Simons Fellowship (Award Number: 663673). Y. Yang's research is supported in part by NSF DMS-1810831. This work is completed in part with the high-performance computing resource provided by the Illinois Campus Cluster Program at UIUC. Authors are listed in alphabetical order}

\begin{abstract}
We introduce the {\it diffusion $K$-means} clustering method on Riemannian submanifolds, which maximizes the within-cluster connectedness based on the diffusion distance. The diffusion $K$-means constructs a random walk on the similarity graph with vertices as data points randomly sampled on the manifolds and edges as similarities given by a kernel that captures the local geometry of manifolds. The diffusion $K$-means is a multi-scale clustering tool that is suitable for data with non-linear and non-Euclidean geometric features in mixed dimensions. Given the number of clusters, we propose a polynomial-time convex relaxation algorithm via the semidefinite programming (SDP) to solve the diffusion $K$-means. In addition, we also propose a nuclear norm regularized SDP that is adaptive to the number of clusters. In both cases, we show that exact recovery of the SDPs for diffusion $K$-means can be achieved under suitable between-cluster separability and within-cluster connectedness of the submanifolds, which together quantify the hardness of the manifold clustering problem. We further propose the {\it localized diffusion $K$-means} by using the local adaptive bandwidth estimated from the nearest neighbors. We show that exact recovery of the localized diffusion $K$-means is fully adaptive to the local probability density and geometric structures of the underlying submanifolds. 
\end{abstract}
\maketitle


\section{Introduction}
\label{sec:introduction}

This article studies the clustering problem of partitioning $n$ data points to $K$ disjoint (smooth) Riemannian submanifolds with $1 \leq K \leq n$. 

\subsection{Problem formulation}

Let $\cM_{k},k=1,\dots,K$ be compact and connected Riemannian manifolds of dimension $q_{k}$. Suppose that $\cM_{k}$ can be embedded as a {\it submanifold} of an ambient Euclidean space $\bR^{p}$ equipped with the Euclidean metric $\|\cdot\|$ (i.e., there is an immersion $\varphi_{k}: \cM_{k} \to \bR^{p}$ such that the differential $\rd \varphi_{k}$ is injective and $\varphi_{k}$ is a homeomorphism onto $\varphi_{k}(\cM_{k}) \subset \bR^{p}$; cf. \cite{doCarmo1992_RG}). In our clustering setting, we work with {\it disjoint} submanifolds $\cM_{1},\dots,\cM_{K}$ in $\bR^{p}$ and denote $S = \bigsqcup_{k=1}^{K} \cM_{k}$ as their disjoint union. Each smooth submanifold $\cM_{k}$ is endowed with the Riemannian metric $\rho_{k}$ induced from $\|\cdot\|$, and we denote $\sD_{k}$ as the Borel $\sigma$-algebra on $\cM_{k}$ (i.e., the $\sigma$-algebra generated by the open balls in $\cM_{k}$ with respect to $\rho_{k}$). Let $X_{1}^{n} := \{X_{1},\dots,X_{n}\}$ be a sequence of independent random variables taking values in $S$. Suppose that there exists a clustering structure $G_{1}^{*},\dots,G_{K}^{*}$ (i.e., a partition on $[n] := \{1,\dots,n\}$ satisfying $\bigsqcup_{k=1}^{K} G_{k}^{*} = [n]$) such that each of the $n$ data points belongs to one of the $K$ clusters: if $i \in G_{k}^{*}$, then $X_{i} \sim \mu_{k}$ for some probability distribution $\mu_{k}$ supported on $\cM_{k}$. Given the observations $X_{1}^{n}$, the task of this paper is to develop computationally tractable algorithms with strong theoretical guarantees for recovering the true clustering structure $G_{1}^{*},\dots,G_{K}^{*}$. 

\begin{figure}[t!] 
    \centering
    \includegraphics[scale=0.4]{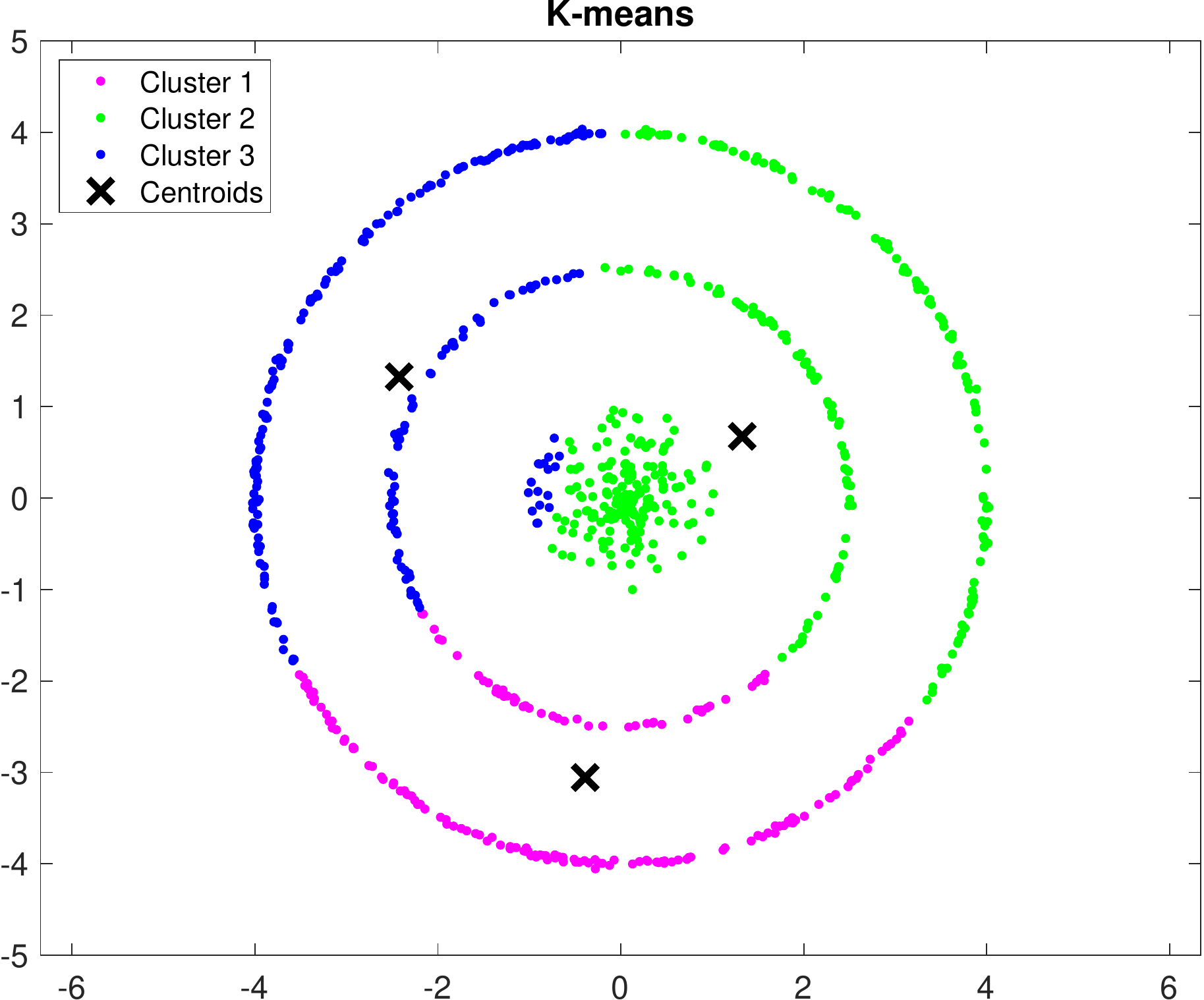}\\
    \includegraphics[scale=0.4]{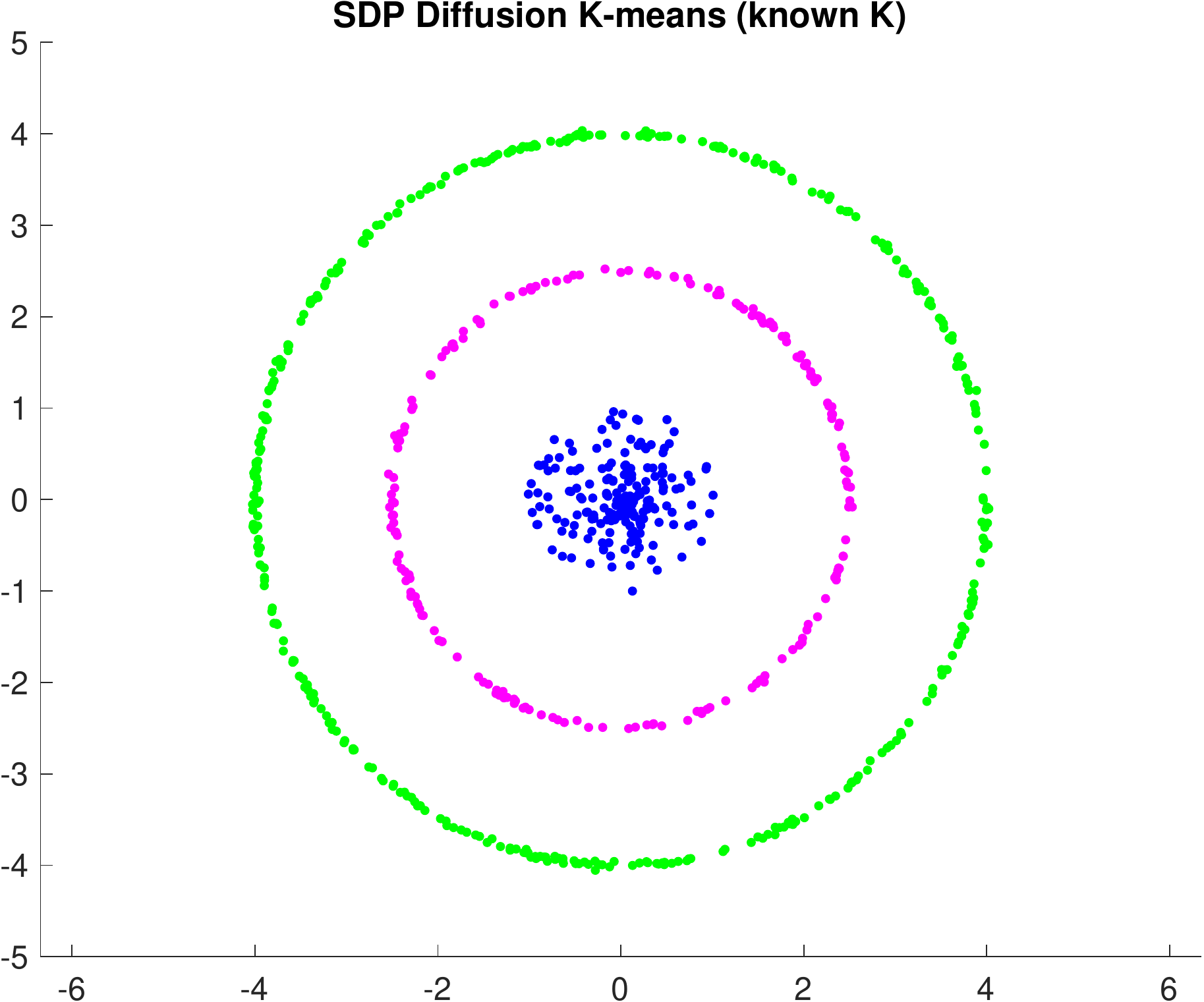}
    \caption{Comparison of the $K$-means and the SDP relaxed diffusion $K$-means clustering methods on a synthetic data sampled from three clusters with one disk and two annuli.}
    \label{fig:kmeans_demo}
\end{figure}

Classical clustering methods such as $K$-means \cite{MacQueen1967_kmeans} and mixture models \cite{FraleyRaftery2002_JASA} assume that data points from each cluster are sampled in the neighborhood (with the same dimension) of a {\it centroid}, where $\cM_{k}$ contains only one point in $\bR^{p}$. Such methods are effective for partitioning data with ellipsoidal contours, which implicitly implies that the similarity (or affinity) criteria of centroid-based clustering methods target on some notions of ``compactness". In modern applications such as image processing and computer vision \cite{ShiMalik2000_IEEEPAMI,SouvenirPless2005_ICCV,ElhmifarVidal2011_NIPS}, structured data with geometric features are commonly seen as clusters without necessarily being close together and having the same dimension. Figure~\ref{fig:kmeans_demo} is an illustration for such observation on a synthetic data sampled from a noisy version of three clusters with one disk and two annuli. In this example, it is visually clear to distinguish the three clusters, however the $K$-means method fails to correctly cluster the data points. There are two main reasons for the failure of $K$-means. First, the north pole and south pole in the outer annulus have the largest Euclidean distance among all data points, even though they belong to the same cluster. Second, the annuli and the disk live in different dimensions. In particular, the annulus is a one-dimensional circle in $\bR^{2}$ that is locally isometric to the real line and the disk has dimension two. Thus these geometric concerns motivate us to seek a more natural and flexible notion of closeness for clustering analysis. In this paper, we shall focus on the clustering criterion based on the {\it connectedness}, which is suitable for simultaneously addressing the two issues. First, connectedness is a graph property that does not rely on the physical distance: two vertices are connected if there is a path joining them. This extends the closeness from the local neighborhood to the global sense. Second, connectivity is a viable notion for clustering components of mixed dimensions, as long as all clusters live in the same ambient space where the graph connectivity weights can be computed. 

In the population version, a clustering component can be viewed as a smooth submanifold, embedded in $\bR^{p}$. In Riemannian geometry, a Riemannian submanifold $\cM$ in $\bR^{p}$ is said to be {\it connected} if for any $x, y \in \cM$, there is a parameterized regular curve joining $x$ and $y$. Thus an appealing notion of $\cM$ for being a cluster is that $\cM$ is a compact and connected component in $\bR^{p}$. In our setting, a clustering model is the union of $K$ disjoint submanifolds $\cM_{1},\dots,\cM_{K}$, and each $\cM_{k}$ is equipped with a probability distribution $\mu_{k}$. Thus for the underlying true clustering model, there are $K$ connected graphs that do not overlap. In the sample version, data points are randomly generated from $\{(\cM_{k}, \mu_{k})\}_{i=1}^{k}$ with a clustering structure $\{G_{k}^{*}\}_{k=1}^{K}$. Typically, weighted graphs computed from the observed data points are fully connected (e.g., based on the Gaussian kernel). Thus a fundamental challenge of clustering analysis is to recover the true clustering structure from a noisy and fully connected weighted graph on the data. 

\subsection{Our contributions}

In this paper, we propose a new clustering method, termed as the {\it diffusion $K$-means}, for manifold clustering. The diffusion $K$-means contains two key ingredients. First, it constructs a random walk (i.e., a Markov chain) on the weighted random graph with data points $X_{1}^{n}$ as vertices and edge weights computed from a kernel representing the similarity of data in a local neighborhood. By running the Markov chain forward in time, the local geometry (specified by the kernel bandwidth) will be integrated at multiple time scales to reveal global (topological) structures such as the connectedness properties of the graph. In the limiting case as the sample size tends to infinity and the bandwidth tends to zero, the random walk becomes a diffusion process over the manifold.
By looking at the spectral decomposition of this limiting diffusion process, the evaluations of the eigenfunctions at vertices $X_{1}^{n}$ can be viewed as a continuous embedding of the data, called {\it diffusion map}, into a higher-dimensional Euclidean space. Second, once the diffusion map is obtained, we can compute the diffusion distance \cite{coifman2006diffusion} and the $K$-means algorithm (with the Euclidean metric) can be naturally extended with the diffusion affinity as the similarity measure. Since the diffusion distance/affinity captures the connectedness among vertices on the weighted random graph, the diffusion $K$-means aims to maximize the within-cluster connectedness, which can be recast as an {\it assignment} problem via a 0-1 integer program. 

Because 0-1 integer programming problems with a non-linear objective function is generally $\mathsf{NP}$-hard, solving the diffusion $K$-means is computationally intractable, i.e., polynomial-time algorithms with exact solutions only exists in special cases. This motivates us to consider semidefinite programming (SDP) relaxations. We propose two versions of SDP relaxations of the diffusion $K$-means. The first one requires the knowledge of the number of clusters, and it can be viewed as an extension from Peng and Wei's SDP relaxation \cite{PengWei2007_SIAMJOPTIM} for the $K$-means (as well as Chen and Yang's SDP relaxation \cite{ChenYang2018} for the generalized $K$-means for non-Euclidean data in an inner product space) to the manifold clustering setting with diffusion distances. Figure~\ref{fig:kmeans_demo} (on the right) shows that the SDP relaxed diffusion $K$-means can correctly identify the three clusters in the previous example. The second SDP relaxation does not require the number of clusters as an input. The idea is to drop the constraint on the trace of the clustering membership matrix (which involves number of clusters $K$), and to add a penalization term on the diffusion $K$-means objective function. Thus it can be seen as a nuclear norm {\it regularized} version of the SDP for diffusion $K$-means that is adaptive to the number of clusters. For both SDP relaxations of the diffusion $K$-means, we show that exact recovery can be achieved when the underlying submanifolds are well separated and subsamples within each submanifold are well connected. 

Since the diffusion $K$-means and its regularized version have only one (non-adaptive) bandwidth parameter to control the local geometry, they may fail for clustering problems with unbalanced sizes, mixed dimensions, and different densities. In such situations, a random walk on the vertices sampled from regions of low density mixes slower than that from regions of high density. This motivates us to consider a variant of diffusion $K$-means, termed as the {\it localized diffusion $K$-means}, by using data-dependent local bandwidth. We adopt the self-tuning procedure from \cite{zelnik2005self} where local adaptive bandwidth is estimated from the nearest neighbors and we show that the localized diffusion $K$-means is adaptive to the local geometry and the local sampling density for the purpose of exact recovery of the true clustering structure. 

\smallskip

To summarize, our contributions are listed as below. 

\begin{enumerate}
\item We introduce the diffusion $K$-means clustering method for manifold clustering, which integrates the nonlinear embedding via the diffusion maps and the $K$-means clustering. 

\item We propose two versions of the SDP relaxations of the diffusion $K$-means: one requires to know the number of clusters, and the other one does not require such knowledge as an input (and thus it is adaptive to the unknown number of clusters). 

\item We derive the exact recovery property of the SDP relaxed diffusion $K$-means in terms of two hardness parameters of the clustering problem: one reflects the separation of the submanifolds, and the other one quantifies the degree of connectedness of the submanifolds. 

\item We combine the local scaling procedures with the diffusion $K$-means and its regularized version, and derive their adaptivity when the clustering problems have unbalanced sizes, mixed dimensions, and different densities. 
\end{enumerate}

\subsection{Related work}

There is a large collection of clustering methods and algorithms in literature, which can be broadly classified into two categories: hierarchical clustering and partition-based clustering. Hierarchical clustering recursively divides data points into groups in either a top-down or bottom-up way. Such algorithms are greedy and they often get stuck into local optimal solutions. 

Partition-based clustering methods such as $K$-means clustering \cite{MacQueen1967_kmeans} and spectral clustering \cite{vanLuxburg2007_spectralclustering} directly assign each data point with a group membership. Perhaps one of the most widely used clustering methods is the $K$-means method, due to the existence of algorithms with linear sample complexity (such as Lloyd's algorithm \cite{Lloyd1982_TIT}). However, the $K$-means clustering converges locally to a stationary point that depends on the initial partition. Recent theoretical studies in \cite{LuZhou2016} show that, given a proper initialization (such as spectral clustering), Lloyd's algorithm for optimizing the $K$-means objective function can consistently recover the clustering structures. Exact and approximate recovery of various convex relaxations for the $K$-means and mixture models are studied in literature \cite{PengWei2007_SIAMJOPTIM,LiLiLingStohmerWei2017,FeiChen2018,Royer2017_NIPS,BuneaGiraudRoyerVerzelen2016}. To the best of our knowledge, existing theoretical guarantees developed for the convex relaxed $K$-means clustering assumes that the clusters are sampled in a neighborhood of a centroid. 
Thus results derived for $K$-means in literature cannot be directly compared with our results. 

On the other hand, spectral clustering methods \cite{ShiMalik2000_IEEEPAMI,NgJordanWeiss2001_NIPS} take the similarity matrix as the input and solve the clustering problem by applying $K$-means to top eigenvectors of the graph Laplacian matrix or its normalized versions \cite{Chung1996_SpectralGraphTheory}. In essence, spectral clustering contains two steps: (i) the Laplacian eigenmaps embed data into feature spaces, and (ii) $K$-means on top eigenvectors serves as a rounding procedure to obtain the true clustering structure~\cite{vanLuxburg2007_spectralclustering}. Conventional intuition for spectral clustering is that the embedding step (i) often ``magnifies" the cluster structure from the dataset to the feature space such that it can be revealed by a relatively simple algorithm (such as $K$-means) in step (ii). However, theoretical guarantees (such as exact recovery) for the spectral clustering is rather vague in literature, partially due to its two-step nature which complicates its theoretical analysis. For instance, \cite{vonLuxburgBelkinBousquet2008_AoS} study the convergence of spectral properties of random graph Laplacian matrices constructed from sample points and they establish the consistency of the spectral clustering in terms of eigenvectors. However, they do not address the problem of the exact recovery property of the clustering structure. Similar results along this direction can be found in~\cite{Rosasco2010_JMLR,schiebinger2015,TrillosHoffmannHosseini2019}. \cite{LingStrohmer2019} propose similar SDP relaxations for the spectral clustering as in the present paper with the diffusion distance metric replaced with the graph Laplacian. Specifically, it is shown in~\cite{LingStrohmer2019} that those SDP relaxations can exactly recover the true clustering structure under a spectral proximity condition. Such condition is deterministic and difficult to check for general data generation models. (A particular checkable random model is the stochastic ball model~\cite{LingStrohmer2019}.) During the preparation of this work, we notice a recent work~\cite{maggioni2018learning} which proposes a similar idea of applying the diffusion distance as the similarity metric for clustering based on fast search and find of density peaks clustering (FSFDPS) \cite{rodriguez2014clustering}. To prove exact recovery, \cite{maggioni2018learning}
requires strong deterministic assumptions on the Markov transition matrix associated with the diffusion process that could be difficult to check under their stochastic clustering model.

Literature on theoretical guarantees for manifold clustering is rather scarce, with a few exceptions~\cite{Arias-Castro2011_IEEETIT,LittleMaggioniMurphy2017}. Near-optimal recovery of some emblematic clustering methods based on pairwise distances of data is derived under a condition that the {\it minimal} signal separation strength over all pairs of submanifolds is larger than a threshold. Compared with our diffusion $K$-means with local scaling, results established in \cite{Arias-Castro2011_IEEETIT} are non-adaptive to the local density and (geometric) structures of the submanifolds (cf. Theorem~\ref{thm:main_adaptive_h} and~\ref{thm:main_adaptive} ahead). \cite{LittleMaggioniMurphy2017} derive recovery guarantees for manifold clusters using a data-dependent metric called the longest-leg path distance (LLPD) that adapts to the geometry of data, where the data points are drawn from a mixture of uniform distributions on disjoint low-dimensional geometric objects.

\subsection{Notation}

For a matrix $A\in\bR^{n\times n}$ and index subsets $G,G'\subset [n]$, we use notation $A_{GG'}$ to denote the submatrix of $A$ with rows being selected by $G$ and columns by $G'$, and $\mbox{diag}(A)$ the $n$-dimensional vector composed of all diagonal entries of $A$. Let $\|A\|_{\infty} =\max_{1 \leq i,j \leq n} |A_{ij}|$ and $\|A\|_1=\sum_{i,j=1}^{n} |A_{ij}|$ denote the $\ell_\infty$ and the $\ell_1$ norm of the vectorization $\mbox{vec}(A)$ of matrix $A$. Let $\matnorm{A}_{\ast}$ and $\matnorm{A}_{\mbox{\scriptsize op}}$ denote the nuclear norm and the operator norm of matrix $A$. For two matrices $A$ and $B$ in $\bR^{n\times n}$, we use $\langle A, B\rangle = \tr(A^TB)$ to denote a matrix inner product. We shall use $c, c', c_{1}, c_{2},\dots,C,C', C_{1},C_{2},\dots$ to denote positive and finite (non-random) constants whose values may depend on the submanifolds $\{\cM_{k}\}_{k=1}^{K}$ and the probability distributions $\{\mu_{k}\}_{k=1}^{K}$ supported on $\{\cM_{k}\}_{k=1}^{K}$ and whose values may vary from place to place. 

\smallskip

The rest of the paper is organized as follows. In Section~\ref{sec:prelims}, we discuss some related background on diffusion distances and nonlinear embeddings in Euclidean spaces, as well as the Laplace-Beltrami operator for the heat diffusion process on Riemannian manifolds. In Section~\ref{sec:diffusion_Kmeans}, we introduce the diffusion $K$-means and its SDP relaxations. Regularized and localized diffusion $K$-means clustering methods are also proposed in this section. In Section~\ref{sec:main_results}, we present our main results on the exact recovery property of the SDP relaxed diffusion $K$-means. Simulation studies are presented in Section~\ref{sec:simulations}. Proofs are relegated to Section~\ref{sec:proofs}.


\section{Preliminaries}
\label{sec:prelims}

Let $S \subset \bR^{p}$ and $\mu$ be a probability distribution on $S$. Let $S_n=\{X_1,X_2,\ldots,X_n\}$ be $n$ i.i.d.~random variables in $S$ sampled from $\mu$, and $\mu_n =n^{-1} \sum_{i=1}^n \delta_{X_i}$ the empirical distribution. In Section~\ref{sec:prelims}, we discuss the Euclidean embedding via the diffusion distance in the general setting. Thus in this section, we do not assume $S$ to be a disjoint union of $K$ submanifolds in $\bR^{p}$ and the sample $S_{n}$ does not necessarily have a clustering structure. 

\subsection{Euclidean embedding and diffusion distances}
\label{subsec:Euclidean_embedding_diffusion_dist}


Let $\kappa : S \times S \to \bR$ be a positive semidefinite kernel that satisfies: 

\begin{enumeratei}
\item symmetry: $\kappa(x, y) = \kappa(y, x)$, 
\item positivity preserving: $\kappa(x, y) \geq 0$. 
\end{enumeratei}
A kernel is a similarity measure between points of $S$. A widely used example is the Gaussian kernel: 
\begin{equation}\label{eqn:gaussian_kernel}
\kappa(x, y) = \exp \left( -{ \|x-y\|^{2} \over 2 h^{2}} \right),
\end{equation}
where $h > 0$ is the bandwidth parameter that captures the local similarity of points in $S$. Given a kernel $\kappa$ with property (i) and (ii), we can define a reversible Markov chain on $S$ via the normalized graph Laplacian constructed as follows. Specifically, for any $x \in S$, let 
\[
d(x) = \int_{S} \kappa (x, y)\,  \rd \mu(y)  
\]
be the degree function of the graph on $S$. For simplicity, we assume $d(x)>0$ for all $x\in S$. Define 
\begin{equation}
\label{eqn:transition_kernel}
p(x, y) = {\kappa(x, y) \over d(x)}, 
\end{equation}
which satisfies the positivity preserving property (ii) and the conservation property 
\[
\int_{S} p(x, y) \, \rd \mu(y) = 1. 
\]
Thus $p(x, y)$ can be viewed as the one-step transition probability of a (stationary) Markov chain on $S$ from $x$ to $y$. We shall write this Markov chain (i.e., random walk) as $\cW = (S, \mu, p)$, where $p(\cdot, \cdot)$ is called the {\it transition kernel} of $\cW$. Equivalently, we can describe $\cW$ by the bounded linear operator $\mathscr P : L^{2}(\rd\mu) \to L^{2}(\rd\mu)$ defined as 
\[
\mathscr Pf(x) = \int_{S} p(x, y) f(y) \, \rd \mu(y). 
\]
Here $L^{2}(\rd\mu) := L^{2}(S, \rd\mu)$ is the class of squared integrable functions on $S$ with respect to $\mu$. In literature, $\mathscr P$ is often called the {\it diffusion operator} for the following reason. If we denote $p_{t}(x, y)$ as the $t$-step transition probability of the Markov chain $\cW$ from $x$ to $y$ in $S$, then 
\[
\mathscr P^{t}f(x) = \int_{S} p_{t}(x, y) f(y) \, \rd \mu(y),\qquad t=1,2,\ldots,
\]
form a semi-group of bounded linear operators on $L^{2}(\rd\mu)$. Let $\Pi$ be a stationary distribution of the Markov chain $\cW$ over $S$. Then $\Pi$ is absolutely continuous with respect to $\mu$, and the probability density function $\pi$ of $\Pi$ with respect to $\mu$ is given by the Radon-Nikodym derivative 
\begin{equation}
\label{eqn:stationary_dist}
\pi(x) = \frac{\rd\Pi}{\rd\mu} (x) = { d(x) \over  \int_{S} d(y)\, \rd\mu(y)}.
\end{equation}
Since $\Pi$ is the stationary measure of the Markov chain $\cW$ with transition $P$, we have 
\[
\Pi(\mathscr P^{t}f) = \Pi(f) 
\] 
for all bounded measurable functions $f$, where $\Pi(f) := \int_{S} f(x) \, \rd \Pi(x)$. Note that, since the kernel $\kappa$ is symmetric, $\cW$ is reversible and satisfies the detailed balance condition: 
\[
\pi(x)\, p(x, y) = \pi(y) \,p(y, x),\quad\forall x,y\in S.
\]

\begin{lem}[Spectral decomposition of the Markov chain $\cW$]
\label{lem:spectral_decomposition_Markov_chain}
Let 
\[
\mathscr R(x,y) = {\kappa(x,y) \over \sqrt{d(x)} \, \sqrt{d(y)}},\quad\forall x,y\in S.
\]
If 
\begin{equation}
\label{eqn:kernel_integrability_condition}
\int_{S} \int_{S} \mathscr R(x,y)^{2} \, \rd \mu(x) \, \rd \mu(y) < \infty, 
\end{equation}
then the following statements hold. 
\begin{enumerate}
\item There exists a sequence of nonnegative eigenvalues $\lambda_{0} \geq \lambda_{1} \geq \cdots \geq 0$ such that 
\[
\mathscr R(x,y) = \sum_{j=0}^{\infty} \lambda_{j} \phi_{j}(x) \phi_{j}(y), 
\]
where $\{\phi_{j}\}_{j=0}^{\infty}$ is the set of associated eigenfunctions to $\{\lambda_{j}\}_{j=0}^{\infty}$, and $\{\phi_{j}\}_{j=0}^{\infty}$ forms an orthonormal basis of $L^{2}(\rd\mu)$. In particular, $\lambda_{0}=1$ and $\phi_{0}(x) = \sqrt{\pi(x)}$.

\item The transition probability $p(x,y)$ admits the following decomposition 
\[
p(x,y) = \sum_{j=0}^{\infty} \lambda_{j} \psi_{j}(x) \widetilde\psi_{j}(y), 
\]
where $\psi_{j}(x) = \phi_{j}(x) / \sqrt{d(x)}$ and $\widetilde\psi_{j}(y) = \phi_{j}(y) \sqrt{d(y)}$. 

\item The diffusion operator $\mathscr P$ satisfies 
\[
\mathscr P\, \psi_{j} = \lambda_{j} \,\psi_{j}, \quad j = 0, 1, \dots.
\]
In addition, $\lambda_0=1$ and $\psi_0\equiv (\int_{S} d(y) \, \rd \mu(y))^{-1/2}$.
\end{enumerate}
\end{lem}

The proof of Lemma~\ref{lem:spectral_decomposition_Markov_chain} is given in Appendix~\ref{app:A}, and our argument is similar to Lemma 12.2 in \cite{levin2017markov} in the finite-dimensional setting. If $\cW$ is irreducible (i.e., the graph on $S$ is connected in that for all $x,y \in S$, there is some $t>0$ such that $p_{t}(x,y) > 0$), then the stationary distribution $\pi$ is unique. Thus if we run this Markov chain $\cW$ forward in time, then the local geometry (captured by the kernel $\kappa$ which is parameterized by the bandwidth $h$) will be integrated to reveal global structures of $S$ at multiple (time) scales. In particular, we can define a class $\{\mathscr D_{t}\}_{t \in \bN_{+}}$ of {\it diffusion distances} \cite{coifman2006diffusion} on $S$ by 
\[
\mathscr D_{t}(x, y) := \|\, p_{t}(x, \cdot) - p_{t}(y, \cdot)\,\|_{L^{2}(\rd\mu/d)} = \left\{ \int_{S} [\, p_{t}(x, z) - p_{t}(y, z)]^{2}\,  {\rd \mu(z) \over d(z)} \right\}^{1/2}. 
\]
Roughly speaking, for each $t\in \bN_{+}$ and $x,y\in S$, the diffusion distance $\mathscr D_{t}(x,y)$ quantifies the total number of paths with length $t$ connecting $x$ and $y$ (see Figure~\ref{fig:diffussion_dist}), thereby reflecting the local connectivity at the time scale $t$. 

\begin{figure}[t!] 
    \centering
    \includegraphics[scale=0.45]{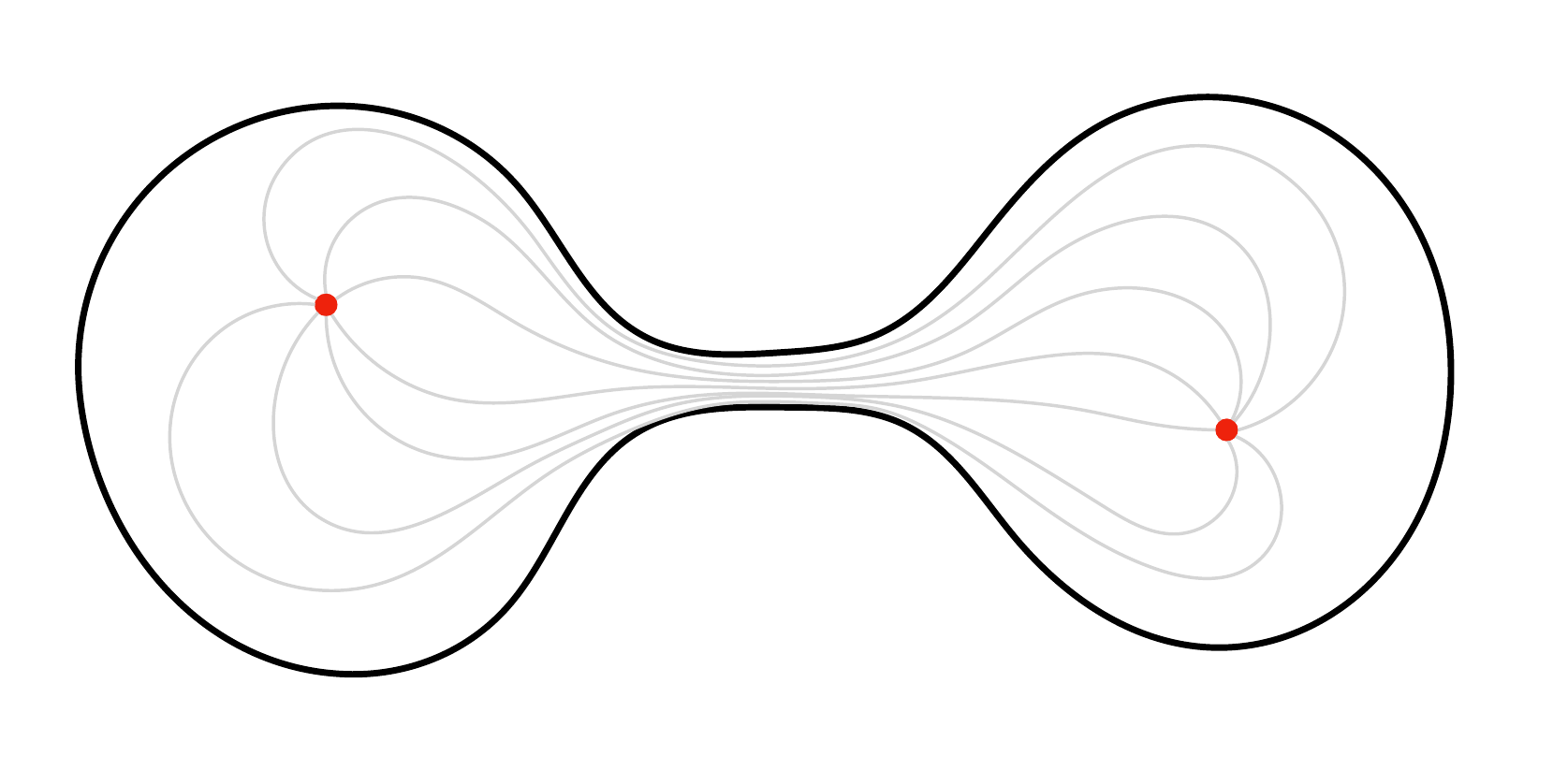}
    \includegraphics[scale=0.45]{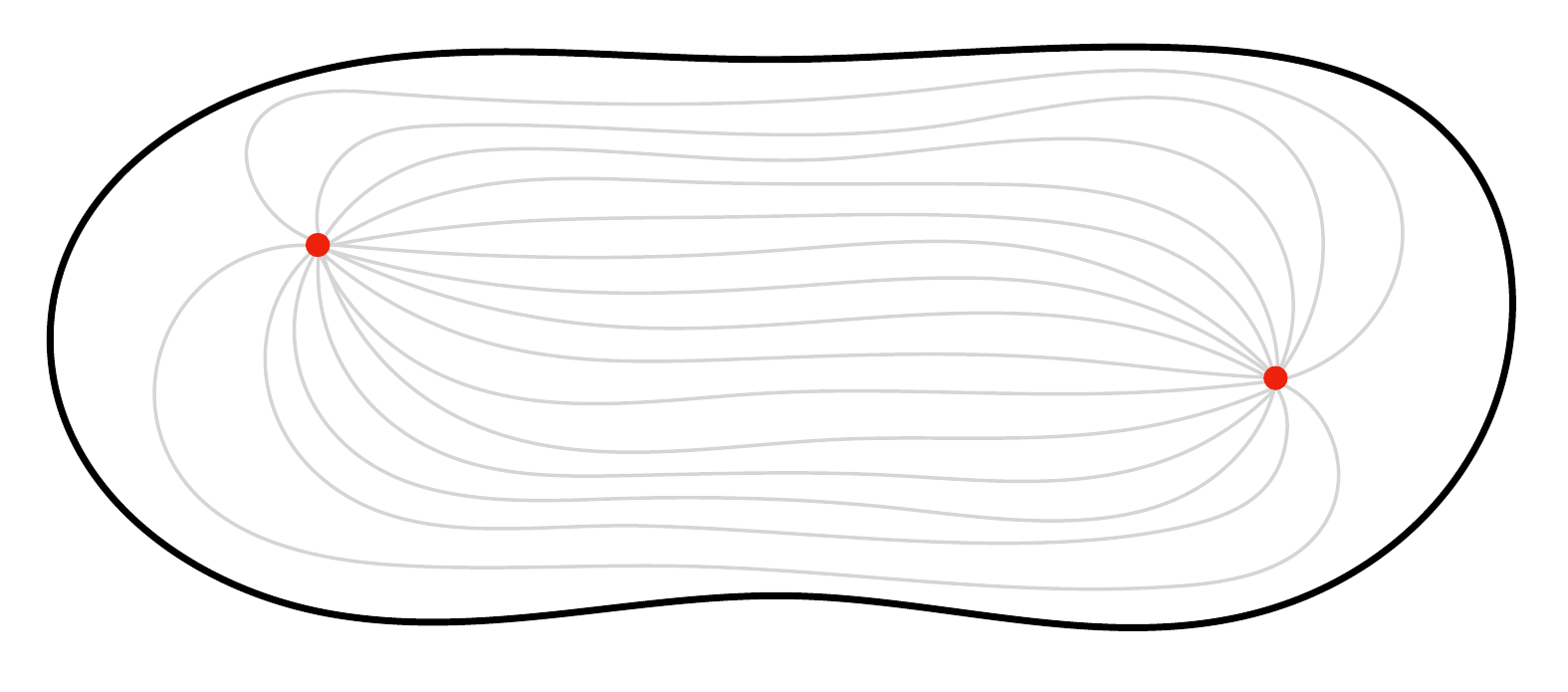}
    \caption{Illustration of the diffusion distance between two red dots as the total number of paths connecting them. The region on the left panel (the Cheeger dumbbell) is ``less" connected than the region on the right as there are fewer paths in the former due to the narrow bottleneck in the middle. In particular, the second smallest eigenvalue associated with the Laplace-Beltrami operator (or the Cheeger isoperimetric constant) of the left region is smaller that of the right.}
    \label{fig:diffussion_dist}
\end{figure}

\begin{lem}[Spectral representation of diffusion distances]
\label{lem:spectral_representation_diffusion_distances}
If the Markov chain $\cW = (S, \mu, p)$ is irreducible, then we have 
\[
\mathscr D_{t}^2(x, y) = \sum_{j=0}^{\infty} \lambda_{j}^{2t} \, [\psi_{j}(x) - \psi_{j}(y)]^{2} 
\]
for all $t\in\bN_{+}$ and $x,y\in S$. 
\end{lem}

The proof of Lemma~\ref{lem:spectral_representation_diffusion_distances} is given in Appendix~\ref{app:A}. For an irreducible Markov chain, the spectral gap is strictly positive (i.e., $|\lambda_{j}| < 1$ for all $j > 0$). Based on the spectral decomposition in Lemma \ref{lem:spectral_representation_diffusion_distances} and noting that $\psi_{0} \equiv (\int_{S} d(y) \, \rd \mu(y))^{-1/2}$ is constant, we see that the diffusion distance can be written as 
\begin{align}\label{Eqn:Spectral_rep}
\mathscr D_{t}(x, y) = \left\{ \sum_{j=1}^{\infty} \lambda_{j}^{2t} \, [\psi_{j}(x) - \psi_{j}(y)]^{2} \right\}^{1/2}.
\end{align}
In this case, the diffusion distance $\mathscr D_{t}(x, y)$ decays to zero as $t$ increases, provided that $x$ and $y$ belong to a connected component of the graph on $S$. In particular, the decay rate of the spectrum quantifies the connectivity of points in the graph on $S$. Given a positive integer $q \in \bN_{+}$, the diffusion maps $\{\Psi_{t}\}_{t \in \bN}$ are defined as 
\[
\Psi_{t}^{(q)}(x) = (\lambda_{0}^{t} \psi_{0}(x), \lambda_{1}^{t} \psi_{1}(x), \dots, \lambda_{q}^{t} \psi_{q}(x))^{T}, 
\]
where the $\ell$-th component $\Psi_{t\ell}^{(q)}(x)$ is the $\ell$-th diffusion coordinate in $\bR^{q}$. Thus we obtain an embedding of $(S, \mu)$ into the Euclidean space $\bR^{q}$ in the limiting sense that 
\[
\mathscr D_{t}(x, y) = \lim_{q \to \infty} \|\Psi_{t}^{(q)}(x) - \Psi_{t}^{(q)}(y)\|_{2}. 
\]

\subsection{Empirical diffusion embedding}
\label{subsec:empirical_diffusion_embedding}

Recall that $S_n=\{X_1,X_2,\ldots,X_n\}$ are $n$ i.i.d.~random variables in $S$ sampled from $\mu$, and $\mu_n =n^{-1} \sum_{i=1}^n \delta_{X_i}$ is the empirical distribution. Given $S_n$, we can consider finite sample approximations $\{\mathscr D_{n,t}\}_{t\in\bN_{+}}$ to the underlying population level quantities $\{\mathscr D_{t}\}_{t\in\bN_{+}}$. More precisely, consider a weighted graph with nodes corresponding to the elements in $S_n$, where the weight between a pair $(X_i,X_j)$ of nodes is $\kappa(X_i,X_j)$, for $i,j\in [n]$.
Define the (rescaled) empirical degree function $d_n:\, S_n\to \bR_{+}$ by
\begin{align*}
d_n(x) = n\,\int_{S_n} \kappa(x, y)\, \rd \mu_n (y) =\sum_{i=1}^n \kappa(x, X_i),\quad \forall x\in S_n,
\end{align*}
where we added an extra $n$-factor so that $d_n(X_i)$ is also the degree of node $X_i$ in the weighted graph. Let $D_n$ denote the $n$-by-$n$ diagonal matrix whose $i$-th diagonal entry is $d_n(X_i)$.
Consider the (empirical) random walk $\cW_n=(S_n, \mu_n, P_n)$ over $S_n$ with transition probability 
\begin{align*}
 P_n (x, y) = \frac{\kappa(x,y)}{d_n(x)}\quad\forall x,y\in S_n.
\end{align*}
The (empirical) stationary distribution $\pi_n$ of the random walk $\cW_n$ over $S_n$ becomes
\begin{align*}
\pi_n(x) = \frac{d_n(x)}{\sum_{i=1}^{n} d_n(X_i)} \quad \forall x\in S_n.
\end{align*}
For any vector $v\in\bR^n_{+}$, let $L^2(v)=\{u=(u_1,\ldots,u_n)\in\bR^n:\, \|u\|_{L^2(v)} = \sum_{i=1}^n v_i\, u_i^2\}$ denote a weighted $L^2$ space over $S_n$. 
We define the {\it empirical diffusion distances} $\{\mathscr D_{n,t}\}_{t\in \bN_{+}}$ as
\begin{align*}
\mathscr D_{n,t}(x,y) =\|P^t_n(x,\cdot) - P^t_n(y,\cdot)\|_{L^2(\mbox{diag}(D_n^{-1}))} = \left\{ \sum_{i=1}^n [\, P_{n}^t(x, X_i) - P^t_n(y, X_i)]^{2}\,  {1 \over d_n(X_i)} \right\}^{1/2},
\end{align*}
for all $x,y\in S_n$ and $t\in\bN_{+}$. Roughly speaking, $\mathscr D_{n,t}$ provides an empirical estimate to $\mathscr D_t$.
Similar to the spectral representation~\eqref{Eqn:Spectral_rep} for $\mathscr D_t$, we also have the following spectral representation of $\mathscr D_{n,t}$ (see Appendix~\ref{app:B}),
\begin{align*}
\mathscr D_{n,t}(x, y)& = \left\{ \sum_{j=0}^{n-1} \lambda_{n, j}^{2t} \, [\psi_{n,j}(x) - \psi_{n,j}(y)]^{2} \right\}^{1/2},
\quad\forall t\in\bN_{+} \mbox{ and }x,y\in S_n,
\end{align*}
where $1=\lambda_{n,0}\geq \lambda_{n,1} \geq\cdots \geq \lambda_{n,n-1} \geq 0$ are the nonnegative eigenvalues (due to the positive semidefiniteness of the kernel $\kappa$) of the transition probability operator $P_n$, which can be identified with a matrix in $\bR^{n\times n}$ with $[P_{n}]_{ij} = P_n(X_i,X_j)$ as its $(i,j)$-th element, and $\psi_{n,0},\psi_{n,1},\ldots,\psi_{n,n-1}:\,S_n \to\bR$ are the associated eigen-functions on $S_n$ with unit $L^2(\mbox{diag}(D_n))$ norm, which can be identified with vectors in $\bR^n$ with $[\psi_{n,j}]_i = \psi_{n,j}(X_i)$ as the $i$-th element of $\psi_{n,j}$ for $i\in [n]$ and $j =0,1\ldots,n-1$. The empirical diffusion distance $\mathscr D_{n,t}(X_i, X_j)$ between two nodes $X_i$ and $X_j$ is also the Euclidean distance between their embeddings $\Psi_{n,t}(X_i)$ and $\Psi_{n,t}(X_j)$ via the empirical diffusion map
\begin{align*}
\Psi_{n,t}:\, S_n \to \bR^{n},\quad x\mapsto \big(\lambda_{n,0}^t \psi_{n,0}(x),\lambda_{n,1}^t \psi_{n,1}(x),\ldots,\lambda_{n,n-1}^t\psi_{n,n-1}(x)\big)^T.
\end{align*} 

\subsection{The Laplace-Beltrami operator on Riemannian manifolds}
\label{subsec:Laplace-Beltrami_operator}

The Laplace-Beltrami operator on Riemannian manifolds is a generalization of the Laplace operator on Euclidean spaces. 
Let $f : \cM \to \bR$ be an (infinitely) differentiable function with continuous derivatives on a $q$-dimensional compact and smooth Riemannian manifold and $\nabla_{\cM} f$ be the gradient vector field on $\cM$ (i.e., $\nabla_{\cM} f(x)$ is the deepest direction of ascent for $f$ at the point $x \in \cM$). The Laplace-Beltrami operator $\Delta_{\cM}$ is defined as the divergence of the gradient vector 
\[
\Delta_{\cM} f = -\mbox{div}(\nabla_{\cM} f), 
\]
where the $\mbox{div}$ operator is relative to the volume form $\mbox{Vol}_{\cM}$ of $\cM$. Here we adopt the convention with the minus sign of the divergence such that $\Delta_{\cM}$ is a positive-definite operator. With integration-by-parts, we have for any two differentiable functions $f$ and $g$, 
\[
\int_{\cM} g(x) \Delta_{\cM} f \,\rd \mbox{Vol}_{\cM}(x) = \int_{\cM} \langle \nabla_{\cM} g(x), \nabla_{\cM} f(x) \rangle \,\rd \mbox{Vol}_{\cM}(x), 
\]
where the inner product is taken in the $q$-dimensional tangent space of $\cM$ (at the point $x$). In a Euclidean space (i.e., $\cM = \bR^{q}$), the Laplace-Beltrami operator is the usual Laplace operator 
\[
\Delta f = -\sum_{j=1}^{q}{\partial^{2}f \over \partial x_{j}^{2}}.
\]
On a general $q$-dimensional Riemannian manifold $\cM$, the Laplace-Beltrami operator in a local coordinate system $(e^{1},\dots,e^{q})$ with a metric tensor $\vG = (g_{ij})_{i,j=1}^{q}$ is given by 
\[
\Delta_{\cM} f = -{1 \over \sqrt{\det(\vG)}} \sum_{j=1}^{q} {\partial \over \partial e^{j}} \left( \sqrt{\det(\vG)} \sum_{i=1}^{q} g^{ij} {\partial f \over \partial e^{i}} \right), 
\]
where $g^{ij}$ are the entries of $\vG^{-1}$. In the special case $\cM = \bR^{q}$, $\vG$ is the $q \times q$ identity matrix.

More generally, if there is a distribution with a positive and Lipschitz continuous density $\rho > 0$ on $\cM$, then the drift Laplace-Beltrami operator is defined as
\begin{equation}
\label{eqn:drift_Laplace-Beltrami}
\Delta_{\cM} f := \Delta_{\cM, \rho} f = -{1 \over \rho^{2}} \mbox{div}(\rho^{2} \nabla f) = \Delta_{\cM} f - 2 \langle \nabla \log\rho, \nabla f \rangle.
\end{equation}
Note that $\Delta_{\cM}$ is a self-adjoint positive-definite compact operator, its spectrum contains a sequence of nonnegative eigenvalues $0 \leq \lambda_{0} \leq \lambda_{1} \leq \cdots$. If in addition $\cM$ is connected, then the second smallest eigenvalue $\lambda_{1} > 0$. As we will show, $\lambda_{1}$ depends on the connectivity of the manifold (Figure~\ref{fig:diffussion_dist}), thus characterizing the limiting mixing time of the empirical random walk $\cW_n$ over the $S_n$ as $n\to \infty$ and $h\to 0^{+}$, when $S_n$ is sampled from the manifold $\cM$.

\section{Diffusion $K$-means}
\label{sec:diffusion_Kmeans}

Recall that in our clustering model, $S_n=\{X_{1},X_2,\ldots,X_n\}$ is a sample of independent random variables taking values in $S$, where $S$ is the union of $K$ disjoint Riemannian submanifolds $\cM_{1},\dots,\cM_{K}$ embedded in the ambient space $\bR^p$. The clustering problem is to divide these $n$ data points into $K$ clusters, so that points in the same cluster belongs to the same connected component in $S$, based on certain similarity measures between the points. In particular, the (classical) $K$-means clustering method minimizes the total intra-cluster squared Euclidean distances in $\bR^p$
\[
\min_{G_{1},\dots,G_{K}} \sum_{k=1}^{K} {1 \over |G_{k}|} \sum_{i,j \in G_{k}} \|X_{i}-X_{j}\|^{2}
\]
over all possible partitions on $[n]$, where $|G_{k}|$ is the cardinality of $G_{k}$. Dropping the sum of squared norms $\sum_{i=1}^{n} \|X_{i}\|^{2}$, we see that the $K$-means clustering is equivalent to the maximization of the total within-cluster covariances
\[
\max_{G_{1},\dots,G_{K}} \sum_{k=1}^{K} {1 \over |G_{k}|} \sum_{i,j \in G_{k}} a_{ij}, \quad\mbox{with }a_{ij} = X_{i}^{T} X_{j}.
\]
Here, $a_{ij}=X_{i}^{T} X_{j}$ can be viewed as a similarity measure specified by the Euclidean space inner product $\langle X_{i}, X_{j}\rangle_{\bR^p}$. In general, we can replace the Euclidean inner product with any other inner product over $S_n$ \cite{ChenYang2018}. For manifold clustering, we replace it with the inner product induced from the empirical diffusion distance, that is,
\begin{align*}
\langle x,\, y\rangle_{\mathscr D_{n,t}} = \langle \Psi_{n,t}(x), \,\Psi_{n,t}(y) \rangle_{\bR^{n}}
=\sum_{j=0}^{n-1} \lambda_{n,j}^{2t} \,\psi_{n,j}(x) \, \psi_{n,j}(y),\quad \forall x,y\in S_n.
\end{align*}
Henceforth, we will refer to $\langle \cdot,\, \cdot\rangle_{\mathscr D_{n,t}}$ as the {\it diffusion affinity}. 
Interestingly, we can obtain this diffusion affinity value without explicitly conducting eigen-decomposition (spectral decomposition) to the transition probability matrix $P_n=D_n^{-1} \mathcal K_n$ (or the symmetrized matrix $D_n^{-1/2} \mathcal K_n D_n^{-1/2}$), where recall that $D_n=\mbox{diag}\big(d_n(X_1),\ldots,d_n(X_n)\big)\in\bR^n$ is the degree diagonal matrix, and $\mathcal K_n=\big[\kappa(X_i,X_j)\big]_{n\times n}\in\bR^{n\times n}$ is the empirical kernel matrix. In fact, we may use the following relation that links the empirical diffusion affinity with entries in matrix $P_n$ raising to power $2t$ (see Appendix~\ref{app:B} for details),
\begin{align*}
\langle x,\, y\rangle_{\mathscr D_{n,t}} = \sum_{j=0}^{n-1} \lambda_{n,j}^{2t} \,\psi_{n,j}(x) \, \psi_{n,j}(y)=[P_n^{2t}D_n^{-1}](x,y).
\end{align*}
This motivates a $K$-means clustering method via diffusion distances, referred to as the \emph{diffusion $K$-means} as
\begin{align}\label{Eqn:Diffussion_K_Means}
\max_{G_{1},\dots,G_{K}} \sum_{k=1}^{K} {1 \over |G_{k}|} \sum_{i,j \in G_{k}} [P_n^{2t}D_n^{-1}]_{ij},
\end{align}
for the tuning parameter $t$ interpreted as the number of steps in the empirical random walk $\cW_n$.
Note that here the affinity matrix $P_n^{2t}D_n^{-1}= D_n^{-1/2} (D_n^{-1/2} \mathcal K_n D_n^{-1/2})^{2t} D_n^{-1/2} \in\bR^{n\times n}$ is symmetric. In light of the connections between the diffusion distance and the random walk $\cW_n$ over $S_n$ in Section~\ref{subsec:empirical_diffusion_embedding}, the diffusion $K$-means attempts to maximize the total within-cluster connectedness. 

\begin{rem}[Intuition of diffusion $K$-means]
\label{rem:intution_DKM}
In Section~\ref{subsec:Euclidean_embedding_diffusion_dist}, we see that, on a connected submanifold $\cM_{k}$, the (population) diffusion process~\eqref{eqn:transition_kernel} converges to the stationary distribution~\eqref{eqn:stationary_dist}: 
\[
p_{t}(x,y) \to \pi(y) = {\int_{\cM_{k}} \kappa(x,y) \,  \rd \mu(x) \over \iint_{\cM_{k} \times \cM_{k}} \kappa(x,z) \,  \rd \mu(x)\,  \rd \mu(z)} \quad \text{as } t \to \infty.
\]
In fact, since the kernel $\kappa$ is positive semidefinite, this convergence holds at a geometric rate governed by the spectral gap of the Laplace-Beltrami operator on $\cM_{k}$ (cf.~\eqref{eqn:mixing_T2} in the proof of Lemma~\ref{Lem:T_2}). Thus the empirical version of the diffusion (i.e., the Markov chain on the random graph generated by $X_{1}^{n}$ and $\kappa$) obeys 
\[
{P^{t}_{n}(X_{i}, X_{j}) \over \sum_{\ell \in G_{k}^{*}} \kappa(X_{\ell}, X_{j})} \approx {1 \over \sum_{\ell,\ell' \in G_{k}^{*}} \kappa(X_{\ell}, X_{\ell'})}
\] 
for any two data points $X_{i}, X_{j} \in \cM_{k}$. On the other hand, if the separation between the submanifolds is large enough and $t$ is not so large, then the probability to diffuse from one cluster to another one is small (cf. Lemma~\ref{lem:between_cluster_random_walk}). Thus we expect that 
\[
{P^{t}_{n}(X_{i}, X_{j}) \over \sum_{\ell \in G_{k}^{*}} \kappa(X_{\ell}, X_{j})} \approx 0
\] 
for any two data points $X_{i} \in \cM_{k}$ and $X_{j} \in \cM_{m}$ such that $k \neq m$. This means that the within-cluster entries of the empirical diffusion affinity are larger than the between-cluster entries. In particular, for suitably large $t \in \bN_{+}$, the empirical diffusion affinity matrix $A_{n} := A_{n,t} = P_{n}^{2t} D_{n}^{-1}$ tends to become close to a block-diagonal matrix 
\begin{align}\label{Eqn:approx_form_A_n}
A_n \approx \begin{pmatrix} 
\displaystyle \frac{1}{N_1} \mathbf{1}_{G_1^\ast} \mathbf{1}^T_{G_1^\ast} & 0 & \cdots & 0\\
0 & \displaystyle \frac{1}{N_2} \mathbf{1}_{G_2^\ast} \mathbf{1}^T_{G_2^\ast} & \cdots & 0\\
\vdots & \vdots & \ddots & \vdots \\
0 & \cdots & 0 & \displaystyle \frac{1}{N_K} \mathbf{1}_{G_K^\ast} \mathbf{1}^T_{G_K^\ast}
\end{pmatrix},
\end{align}
where $N_{k} = \sum_{\ell,\ell' \in G_{k}^{*}} \kappa(X_{\ell}, X_{\ell'})$. Since each diagonal block of $A_{n}$ tends to be a constant matrix, if we run the Markov chain for a suitably long time, then this block-diagonal structure in the limit precisely conveys the true clustering structure in that $i, j \in G_{k}^{*}$ if and only if $\lim_{t \to \infty} [A_{n,t}]_{i,j} = N_{k}^{-1} > 0$. The trade-off regime of $t$ (cf.~\eqref{eqn:exact_recovery_condition_SDP_diffusion_Kmeans_LB_eigenval} in Theorem~\ref{thm:main}) is determined by the non-asymptotic bounds on the convergence of the empirical diffusion maps to its population version (cf. Lemma~\ref{lem:within_cluster_random_walk} and~\ref{lem:between_cluster_random_walk}), as well as the submanifolds separation. \qed
\end{rem}

Note that, for every partition $G_{1},\dots,G_{K}$, there is a one-to-one $n \times K$ {\it assignment matrix} $H = (h_{ik}) \in \{0,1\}^{n \times K}$ such that $h_{ij} = 1$ if $i \in G_{k}$ and $h_{ij} = 0$ if $i \notin G_{k}$. Thus the diffusion $K$-means clustering problem can be recast as a 0-1 integer program:
\begin{equation}
\label{eqn:kernel_Kmeans_integer_program}
\max \left\{ \langle P_n^{2t}D_n^{-1}, H B H^{T} \rangle : H \in \{0,1\}^{n \times K}, H \vone_{K} = \vone_{n} \right\},
\end{equation}
where $\vone_{n}$ denotes the $n \times 1$ vector of all ones and $B = \diag(n_{1}^{-1},\dots,n_{K}^{-1})$, where $n_{k}=|G_{k}|$ for $k=1,\ldots,K$ is the size of the $k$-th cluster. 

The diffusion $K$-means clustering problem (\ref{eqn:kernel_Kmeans_integer_program}) is often computationally intractable, namely, polynomial-time algorithms with exact solutions only exist in special cases \cite{SongSmolaGrettonBorgwardt2007_ICML}. For instances, the (classical) $K$-means clustering is an $\mathsf{NP}$-hard integer programming problem with a non-linear objective function \cite{PengWei2007_SIAMJOPTIM}. Exact and approximate recovery of various SDP relaxations for the $K$-means \cite{PengWei2007_SIAMJOPTIM,LiLiLingStohmerWei2017,FeiChen2018,Royer2017_NIPS,GiraudVerzelen2018} are studied in literature. However, it remains a challenging task to provide statistical guarantees for clustering methods that can capture non-linear features of data taking values on manifolds.

\subsection{Semidefinite programming relaxations}

We consider the SDP relaxations for the diffusion $K$-means clustering. Note that every partition $G_{1},\dots,G_{K}$ of $[n]$ can be represented by a partition function $\sigma : [n] \to [K]$ via $G_{k}=\sigma^{-1}(k), k=1,\dots,n$. If we change the variable $Z = H B H^{T}$ in the 0-1 integer program formulation (\ref{eqn:kernel_Kmeans_integer_program}) of the diffusion $K$-means, then $Z$ satisfies the following properties: 
\begin{equation}
\label{eqn:constraints_clustering_generic_integer_program}
Z^{T} = Z, \quad Z \succeq 0, \quad \tr(Z) = \sum_{k=1}^{K} n_{k} b_{kk}, \quad (Z \vone_{n})_{i} = \sum_{k=1}^{K} n_{k} b_{\sigma(i)k}, \; i=1,\dots,n.
\end{equation}
For the diffusion $K$-means $B = \diag(n_{1}^{-1},\dots,n_{K}^{-1})$, the last constraint in (\ref{eqn:constraints_clustering_generic_integer_program}) reduces to $Z \vone_{n} = \vone_{n}$, which does not depend on the partition function $\sigma$. Thus we can relax the diffusion $K$-means clustering to the SDP problem: 
\begin{equation}
\label{eqn:clustering_Kmeans_sdp}
\begin{gathered}
\hat{Z} = \argmax \left\{ \langle A, Z \rangle : Z \in \sC_K \right\} \\
\qquad \mbox{with } \sC_K = \{Z \in \bR^{n \times n} : Z^{T} = Z, Z \succeq 0, \tr(Z) = K, Z \vone_{n} = \vone_{n}, Z \geq 0 \}, 
\end{gathered}
\end{equation}
where $Z \succeq 0$ means that $Z$ is positive semidefinite and $Z \geq 0$ means that all entries of $Z$ are non-negative, and matrix $A=[a_{ij}]:=A_{n}=P_n^{2t}D_n^{-1}\in\bR^{n\times n}$. We shall use $\hat{Z}$ to estimate the true ``membership matrix" $Z^{*}$, where 
\begin{equation}
\label{eqn:Kmeans_true_membership_matrix}
Z_{ij}^{*} = \left\{
\begin{array}{cc}
1/n_{k} & \text{if } i, j \in G_{k}^{*} \\
0 & \text{otherwise} \\
\end{array}
\right. .
\end{equation}
Note that $Z^{*}$ is a block diagonal matrix (up to a permutation) of rank $K$. If $X_{1},\dots,X_{n} \in \bH$ (i.e., $\bS = \bH$) for some Hilbert space $\bH$ and $a_{ij} = \langle X_{i}, X_{j} \rangle_{\bH}$ is the inner product between $X_i$ and $X_j$, then (\ref{eqn:clustering_Kmeans_sdp}) is the SDP for kernel $K$-means proposed in \cite{ChenYang2018}. In particular, \cite{PengWei2007_SIAMJOPTIM} consider the special case for the (Euclidean) $K$-means, where $\bH = \bR^{p}$ and $a_{ij} = X_{i}^{T} X_{j}$. Observe that the SDP relaxation (\ref{eqn:clustering_Kmeans_sdp}) does not require the knowledge of the cluster sizes other than the number of clusters $K$. Thus it can handle the general case for unequal cluster sizes. 

\subsection{Regularized diffusion $K$-means}
\label{subsec:regularized_diffusion_Kmeans}

In practice, the number $K$ of clusters is rarely known. Note that the SDP problem~\eqref{eqn:clustering_Kmeans_sdp} depends on $K$ only through the constraint $\tr(Z) = K$. Therefore we propose a {\it regularized diffusion $K$-means} estimator by dropping the constraint on the trace and penalizing $\tr(Z)$ as follows,
\begin{equation}
\label{eqn:clustering_Kmeans_sdp_unknown_K}
\begin{gathered}
\tilde{Z} := \tilde{Z}_{\rho} = \argmax \left\{ \langle A, Z \rangle - n\,\rho \tr(Z) : Z \in \sC \right\} \\
\mbox{with } \sC = \{Z \in \bR^{n \times n} : Z^{T} = Z, Z \succeq 0,\, Z \vone_{n} = \vone_{n}, Z \geq 0 \},
\end{gathered}
\end{equation}
where $\rho>0$ is the regularization parameter. Recall that $\matnorm{Z}_{\ast}$ denotes the nuclear norm of a matrix $Z$ (i.e., $\matnorm{Z}_{\ast}$ is the sum of the singular values of $Z$). Since $\matnorm{Z}_{\ast} = \tr(Z)$ for $Z \in \sC$, it is interesting to note that~\eqref{eqn:clustering_Kmeans_sdp_unknown_K} is the same as the nuclear norm penalized diffusion $K$-means
\begin{equation}
\label{eqn:clustering_Kmeans_sdp_unknown_K_nuclear-norm_form}
\begin{gathered}
\tilde{Z} = \argmax \left\{ \langle A, Z \rangle - n\,\rho \matnorm{Z}_{\ast} : Z \in \sC \right\}.
\end{gathered}
\end{equation}
Recall that the true membership matrix $Z^{*}$ has a block diagonal structure with rank $K$, the nuclear norm penalized diffusion $K$-means~\eqref{eqn:clustering_Kmeans_sdp_unknown_K_nuclear-norm_form} can be thought as an $\ell^{1}$ norm convex relaxation of the $\mathsf{NP}$-hard rank minimization problem (i.e., minimizing the number of non-zero eigenvalues). Thus the parameter $K$ in the clustering problem plays a similar role as the sparsity (or low-rankness) parameter in the matrix completion context \cite{CandesRecht2009_FoCM}. Hence the SDP problem~\eqref{eqn:clustering_Kmeans_sdp_unknown_K_nuclear-norm_form} can be viewed as a (further) convex relaxation of the infeasible SDP problem~\eqref{eqn:clustering_Kmeans_sdp} when $K$ is unknown. Note that similar regularizations have been considered in~\cite{BuneaGiraudRoyerVerzelen2016} for the $G$-latent clustering models and in~\cite{YanSarkarCheng2018_AISTATS} for stochastic block models.

It remains a question to choose the value of $\rho$. Larger values of $\rho$ will lead to solutions containing less number of clusters (with larger sizes).
In particular, when matrix $A$ is positive-definite, the following Lemma~\ref{lem:feasibility_SDP_lambda_infinity} shows that the solution reduces to a rank one matrix that assigns all points into a giant cluster when $\rho$ is large enough, and becomes the identity matrix that assigns $n$ points into $n$ distinct clusters when $\rho$ is small enough. In addition, the trace $\tr(\tilde Z_\rho)$ of the solution is nonincreasing in $\rho>0$. 

\begin{lem}[Monotonicity of tuning path]
\label{lem:feasibility_SDP_lambda_infinity}
Suppose $A$ is a positive definite matrix, and let $\lambda_{\max}(A)$ and $\lambda_{\min}(A)$ denote its respective largest and smallest eigenvalues. 
\newline
(1) If $n \rho > \lambda_{\max}(A)$, then $Z^{\diamond} = n^{-1} J_{n}$, where $J_{n}$ is the $n \times n$ matrix of all ones, is the unique solution of the SDP~\eqref{eqn:clustering_Kmeans_sdp_unknown_K}. 
\newline
(2) If $n \rho < \lambda_{\min}(A)$, then $Z^{\dagger} = I_n$, the $n \times n$  identity matrix, is the unique solution of the SDP~\eqref{eqn:clustering_Kmeans_sdp_unknown_K}. 
\newline
(3) If $\tilde Z_{\rho_1}$ and $\tilde Z_{\rho_2}$ are two solutions of the SDP~\eqref{eqn:clustering_Kmeans_sdp_unknown_K} with the regularization parameter taking values $\rho_1$ and $\rho_2$, respectively. If $\rho_1 < \rho_2$, then $\tr(\tilde Z_{\rho_1})\geq \tr(\tilde Z_{\rho_2})$. Furthermore, if at least one of the two SDPs has a unique solution, then $\tr(\tilde Z_{\rho_1})> \tr(\tilde Z_{\rho_2})$.
\end{lem}


According to the interpretation of the SDP~\eqref{eqn:clustering_Kmeans_sdp}, the trace $\tr(Z)$ of the solution can be viewed as the fitted number of clusters.
Consequently, Lemma~\ref{lem:feasibility_SDP_lambda_infinity} implies that smaller values of $\rho$ will result in more clusters (with smaller sizes) in $\tilde{Z}_\rho$. In practice, we need to properly select the tuning parameter $\rho$. We propose the following decision rule for this purpose. For each $\rho$, we run the SDP problem~\eqref{eqn:clustering_Kmeans_sdp_unknown_K} and extract the value of $\tr(\tilde{Z}_\rho)$. Then we plot the solution path $\tr(\tilde{Z}_\rho)$ versus $\log(\rho)$ and pick the values of $\rho$ which spend the longest time (on the logarithmic scale) with a flat value of $\tr(\tilde{Z}_\rho)$. Here we recommend using the logarithmic scale since values of $\rho$ with non-trivial solutions $\tr(\tilde Z_\rho)$ tends to be close to zero.
Algorithm~\ref{alg1} below summarizes this decision rule.
\vspace{1em}

\RestyleAlgo{boxruled}
\LinesNumbered
 \begin{algorithm}[H]\label{alg1}
 	Set an increasing sequence $\{\rho_j\}_{j=1}^J$ of candidate values for $\rho$, for example, a geometric sequence in the interval $[n^{-1}\lambda_{\min}(A),\,n^{-1}\lambda_{\max}(A)]$. Set an upper bound $K_{\max}$ of $K$ and a tolerance level $\varepsilon\in (0,1/2)$.\\
 	\For{$j=1:J$}{
 		Solve the SDP~\eqref{eqn:clustering_Kmeans_sdp_unknown_K} with $\rho=\rho_j$;\\
 	}
 	\For{$k=2:K_{\max}$}{
 	Let $j_{k,1}$ be the smallest index such that $\tr(\tilde Z_{\rho_{j_{k,1}}}) \leq k +\varepsilon$;\\
 	Let $j_{k,2}$ be the largest index such that $\tr(\tilde Z_{\rho_{j_{k,2}}}) \geq k-\varepsilon$;\\
 	Let $L_k = \log \rho_{j_{k,2}} - \log\rho_{j_{k,1}}$;\\
 	}
 	\KwResult{Choose $\hat K=\underset{2\leq k \leq K_{\max}}{\arg\max} L_k$ as the estimated number of clusters, and $\hat \rho = \eta_{\lfloor (j_{\hat K, 1}+j_{\hat K, 2})/2\rfloor}$ as the selected regularization parameter.
 	}
 	\caption{Selection of regularization parameter $\rho$ and estimation of number of clusters $K$.} 
 \end{algorithm} 
 \vspace{1em}

Figure~\ref{fig:diffusion_kmeans_lambda_demo} shows the empirical result on the three clusters (one disk and two annuli) example in Section~\ref{sec:introduction}. According to Lemma~\ref{lem:feasibility_SDP_lambda_infinity}, the estimated number of clusters, proxied by $\tr(\tilde{Z}_\rho)$, is a non-increasing function of $\rho$.
In particular, the trace $\tr(\tilde Z_\rho)$ in the solution path in the upper left panel of Figure~\ref{fig:diffusion_kmeans_lambda_demo} stabilizes around $2$ and $3$, indicating that both $2$ and $3$ are candidate values for the number of clusters. In particular, the interval of $\rho$ (on the logarithmic scale) corresponding to value $3$ is much larger than that to value $2$, indicating that $3$ is more likely to be the true number of clusters (cf. the (correct) case in Figure~\ref{fig:diffusion_kmeans_lambda_demo}). In Section~\ref{sec:main_results}, we will use our theory to explain this stabilization phenomenon, which partially justifies our $\rho$ selection rule.

In addition, as we can see from the rest three panels in Figure~\ref{fig:diffusion_kmeans_lambda_demo}, by gradually increasing $\rho$, the adaptive diffusion $K$-means method produces a hierarchical clustering structure. Unlike the top-down or bottom-up clustering procedures which are based on certain greedy rule and can incur inconsistency, the hierarchical clustering structure produced by our approach is consistent --- it does not depend on the order of partitioning or merging due to the uniqueness of the global solution from the convex optimization via the SDP. 
 Similar observations can be drawn on another example shown in Figure~\ref{fig:diffusion_kmeans_lambda_demo_DGP2} containing a uniform sample on three rectangles (see DGP 2 in our simulation studies Section~\ref{sec:simulations} for details).

Further, it is interesting to observe in Figure~\ref{fig:diffusion_kmeans_lambda_demo} that the regularized diffusion $K$-means tuned with two clusters yields a merge between the outer annulus and the disk, which gives the largest total diffusion affinity in the objective function~\eqref{Eqn:Diffussion_K_Means} among the three possible combinations of the true clusters. Since diffusion affinity decays exponentially fast to zero in the squared Euclidean distance (for the Gaussian kernel), the diffusion affinity matrix $A = P_{n}^{2t} D_{n}^{-1}$ tends to have a block diagonal structure, as weights between points belonging to different clusters are exponentially small (cf.~\eqref{Eqn:approx_form_A_n} and Lemma~\ref{Lemma:total_degree}). Thus running SDP for examples with relatively well separated clusters, such as the one in Figure~\ref{fig:diffusion_kmeans_lambda_demo}, tends to merge two clusters with largest
within-cluster diffusion affinities that is irrespective of the between-cluster Euclidean distances. This may lead to a visually less appealing merge as in the Euclidean distance case (cf. the (under) case in Figure~\ref{fig:diffusion_kmeans_lambda_demo}). On the other hand, the regularized diffusion $K$-means is able to produce more reasonable partition in splitting the clusters (cf. the bottom-left panel in Figure~\ref{fig:diffusion_kmeans_lambda_demo}). In particular, if the regularization parameter $\rho$ is chosen such that the corresponding number $\hat{K}$ of clusters in the SDP solution is greater than $K$, then this will cause a split in one of the true clustering structures that
minimizes the between-cluster diffusion affinities after the splitting.
Moreover, in our simulation studies (setup DGP 3 in Section~\ref{sec:simulations}), we observe that the SDP relaxed regularized diffusion $K$-means performs much better in harder cases than the spectral clustering methods when the true clusters are not well separated. 

\begin{figure}[h!] 
    \centering
    \includegraphics[scale=0.34]{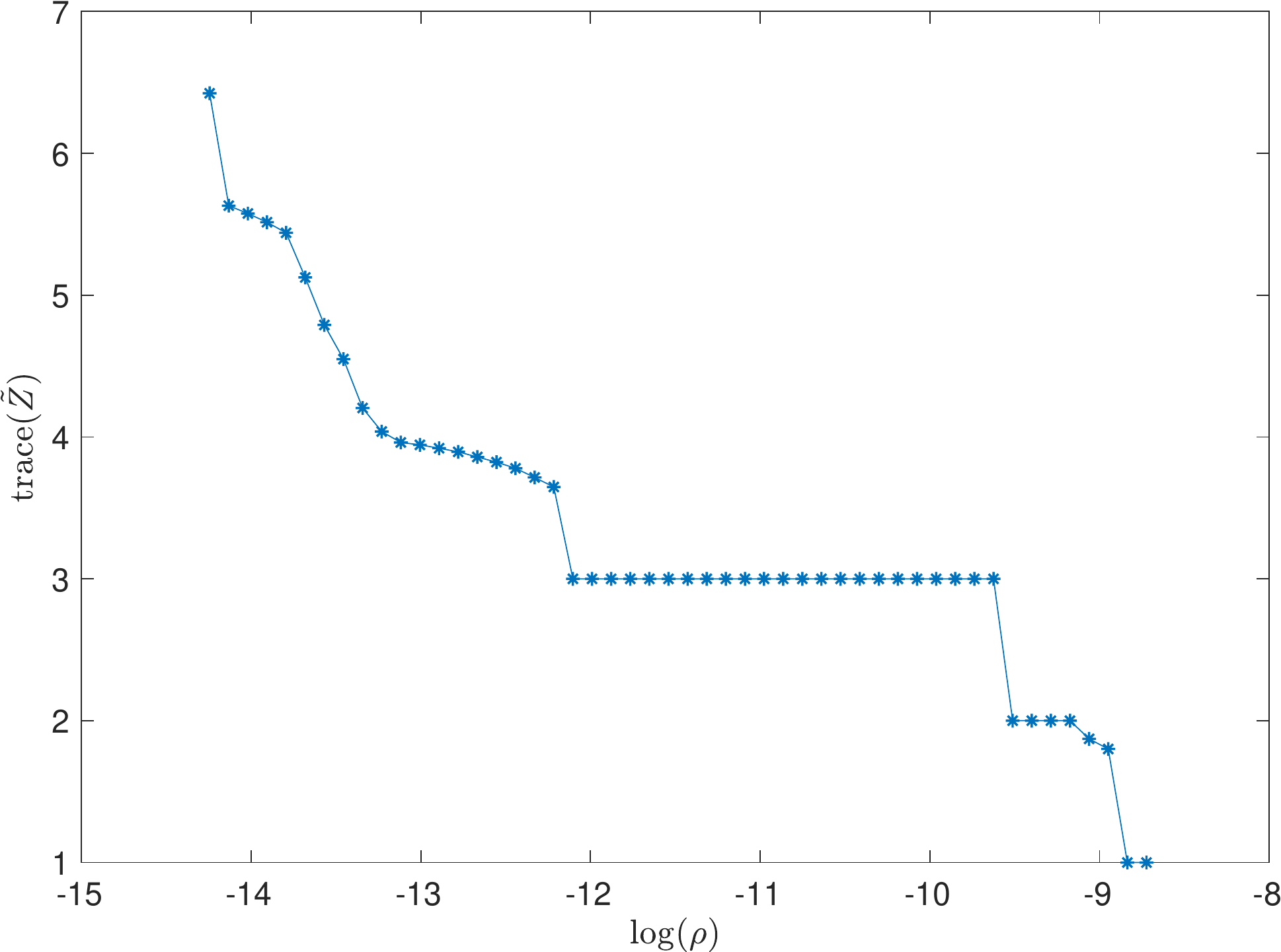}
    \includegraphics[scale=0.52]{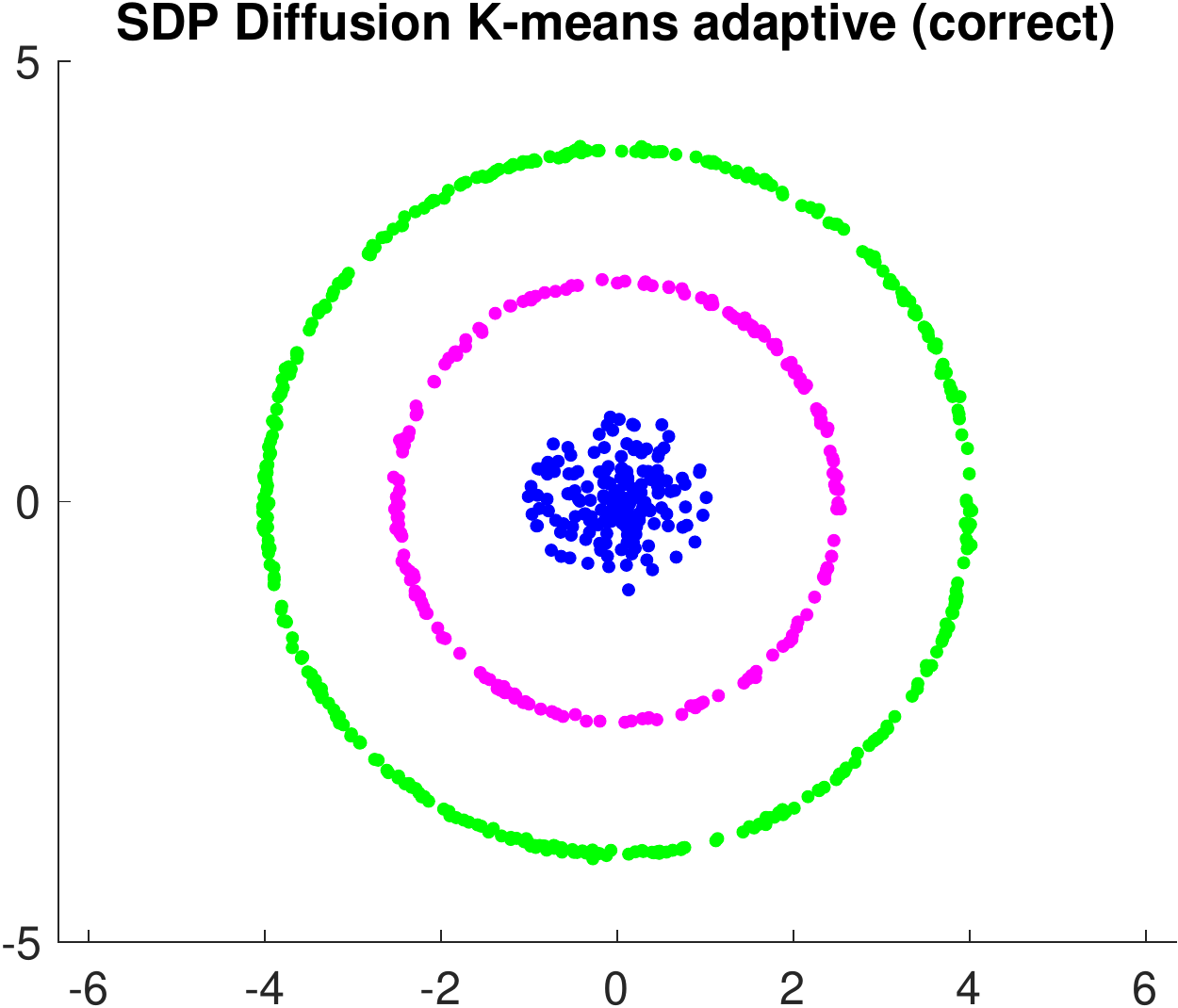}  \\
    \includegraphics[scale=0.52]{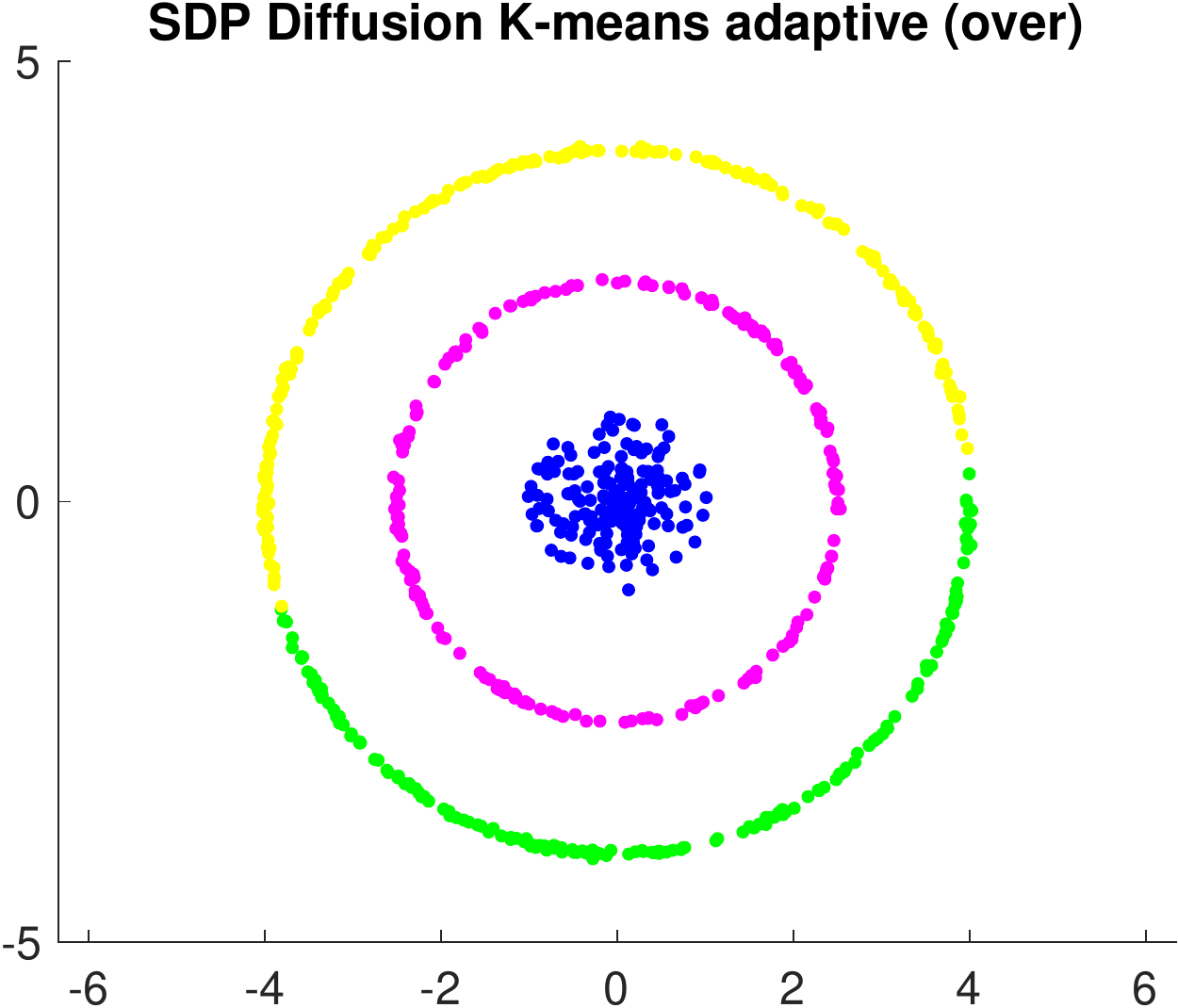}
    \includegraphics[scale=0.52]{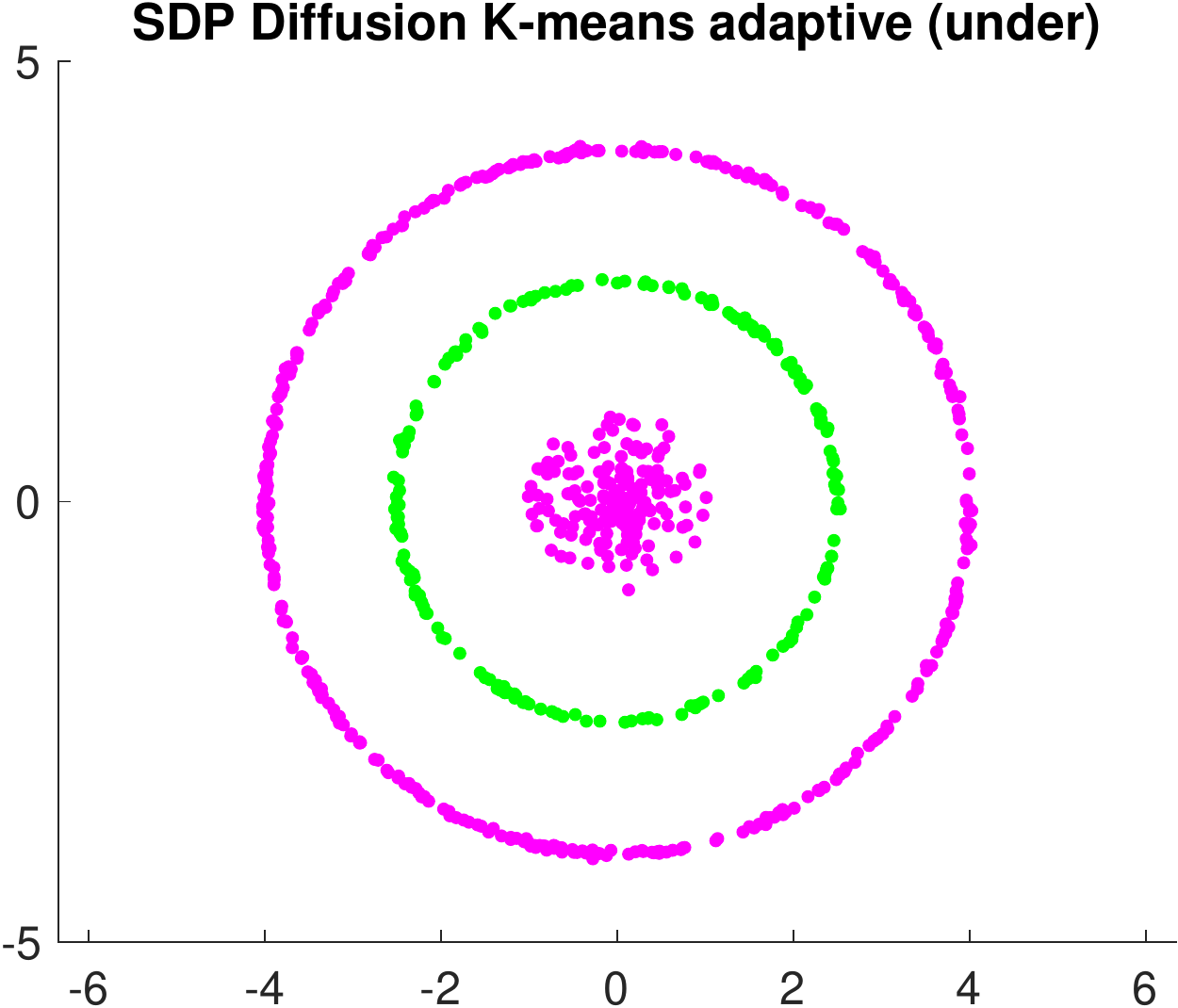}
    \caption{Plot of the estimated trace norms along the path of tuning parameter $\rho$ (on the log-scale) in the SDP relaxed regularized diffusion $K$-means clustering method~\eqref{eqn:clustering_Kmeans_sdp_unknown_K} on the synthetic data in Figure~\ref{fig:kmeans_demo}. The clustered data are shown for estimated number of clusters 2 (under), 3 (correct), 4 (over).}
    \label{fig:diffusion_kmeans_lambda_demo}
\end{figure}

\begin{figure}[h!] 
    \centering
    \includegraphics[scale=0.34]{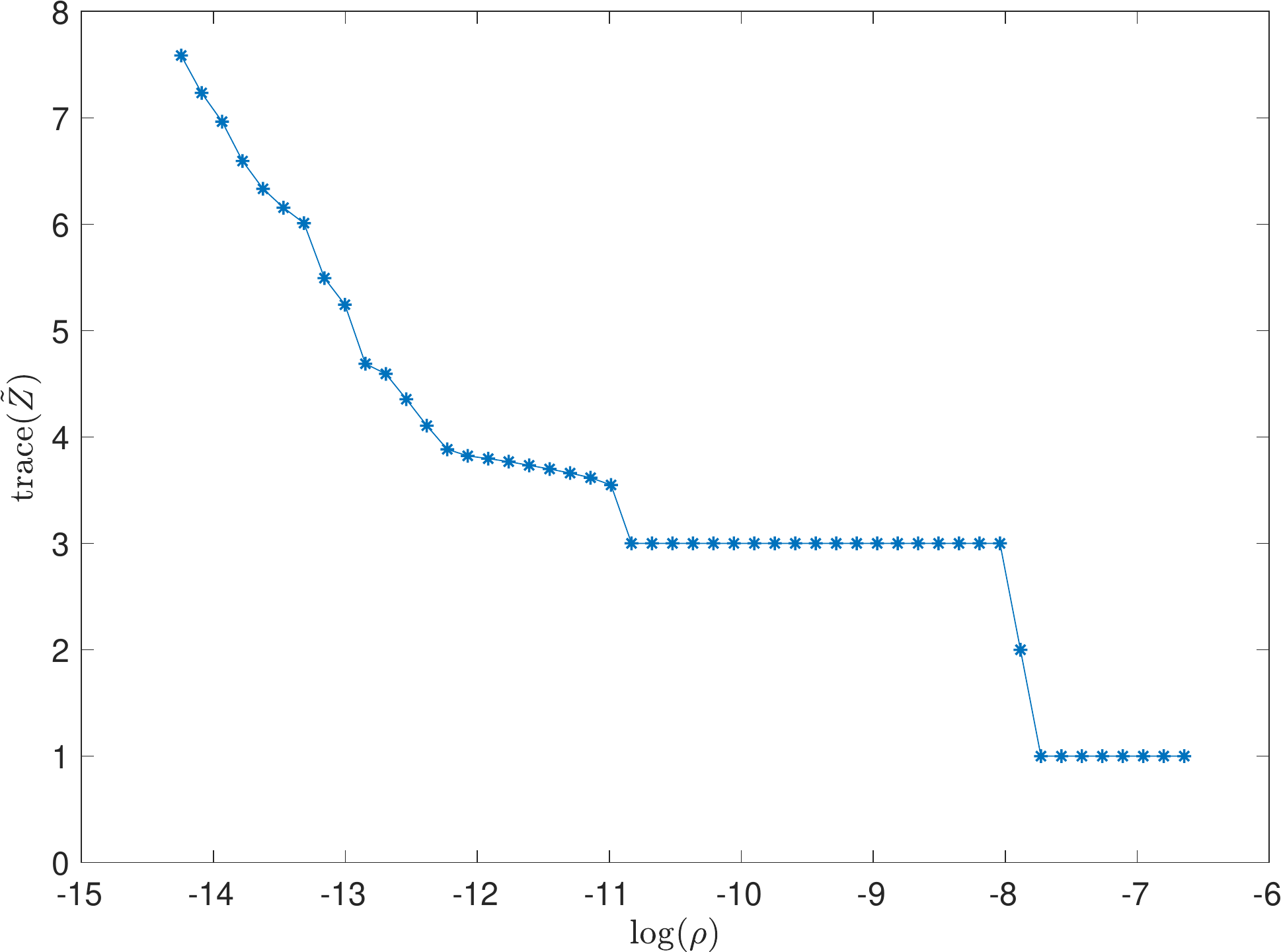}
    \includegraphics[scale=0.38]{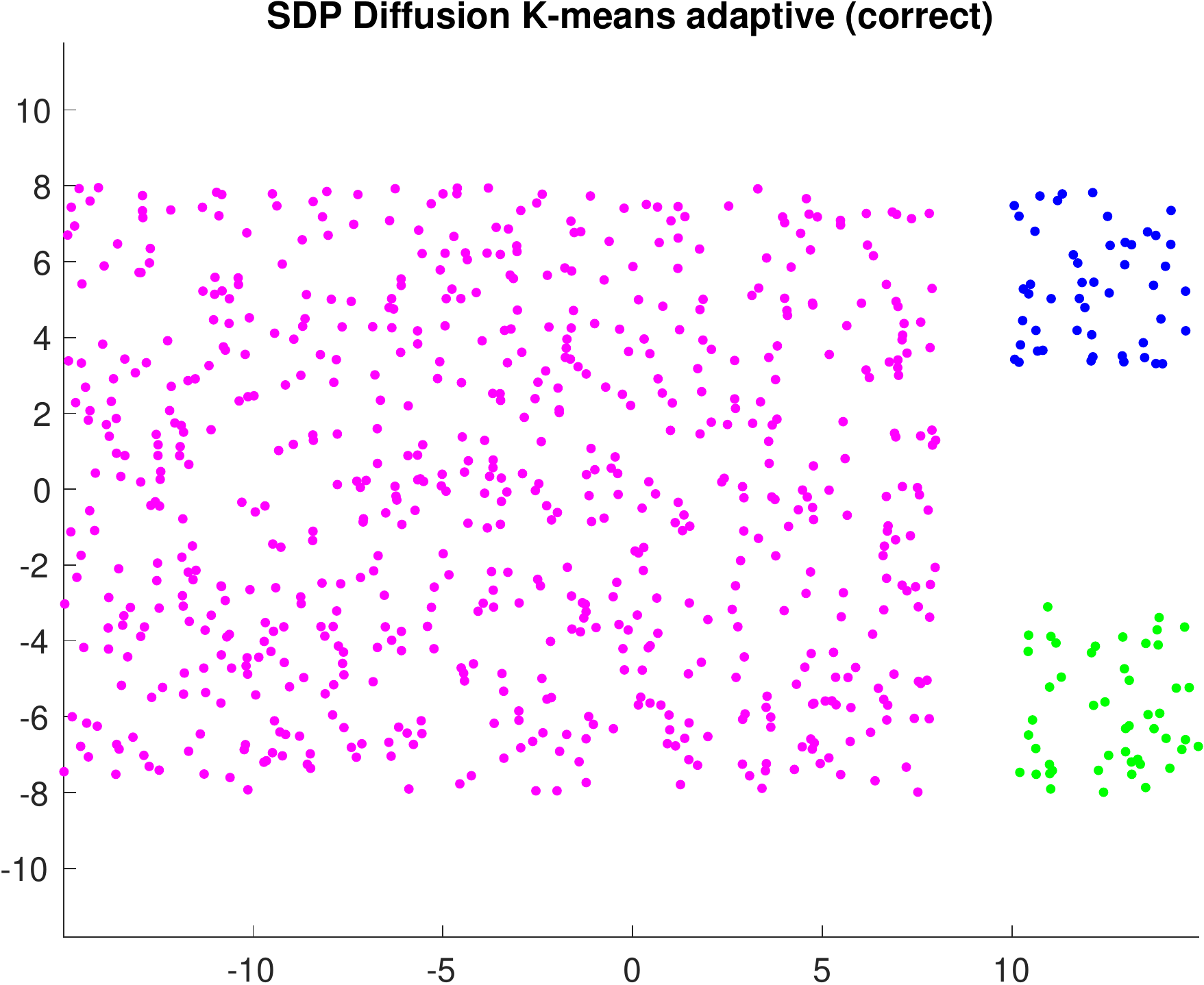}  \\
    \includegraphics[scale=0.38]{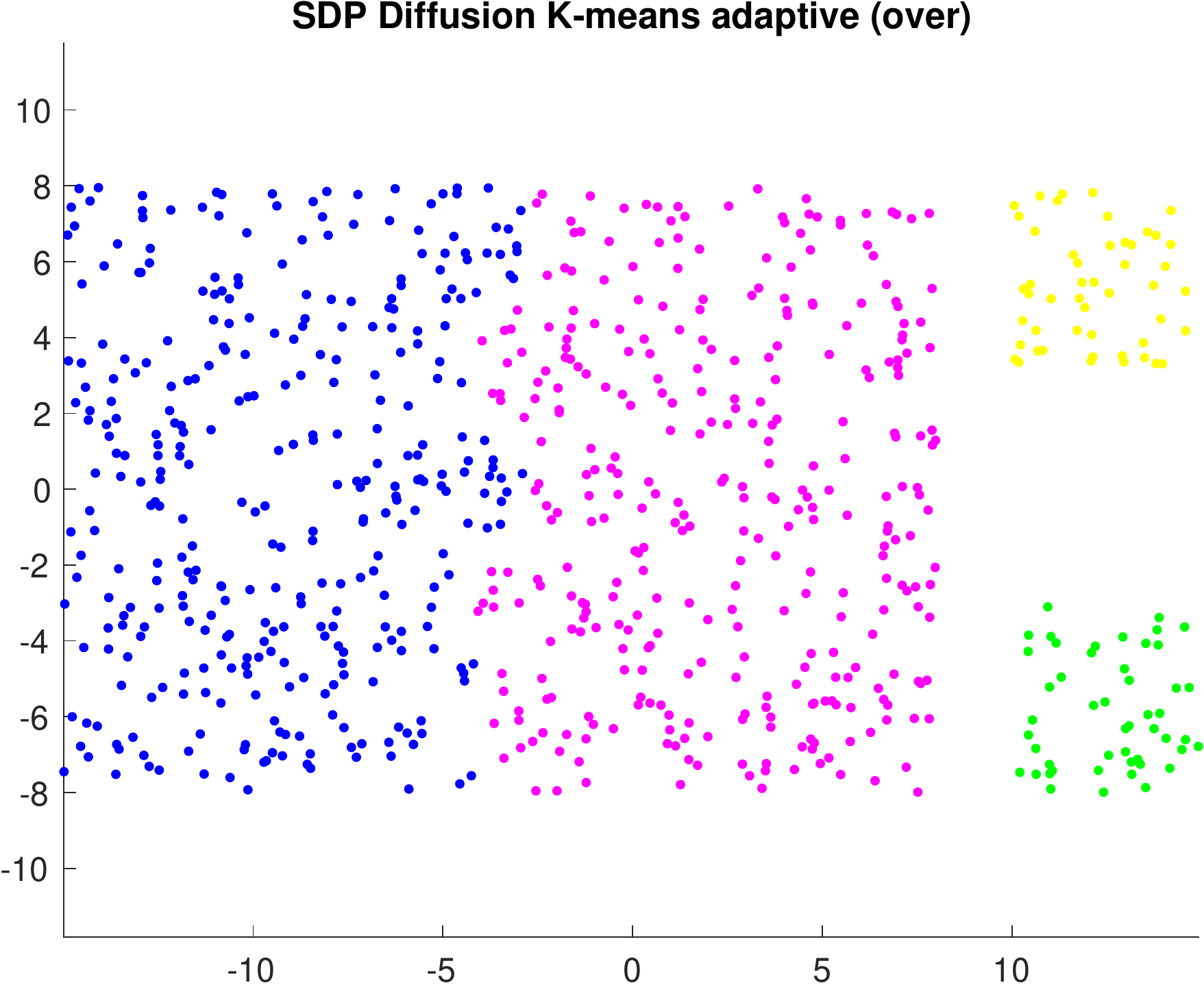}
    \includegraphics[scale=0.38]{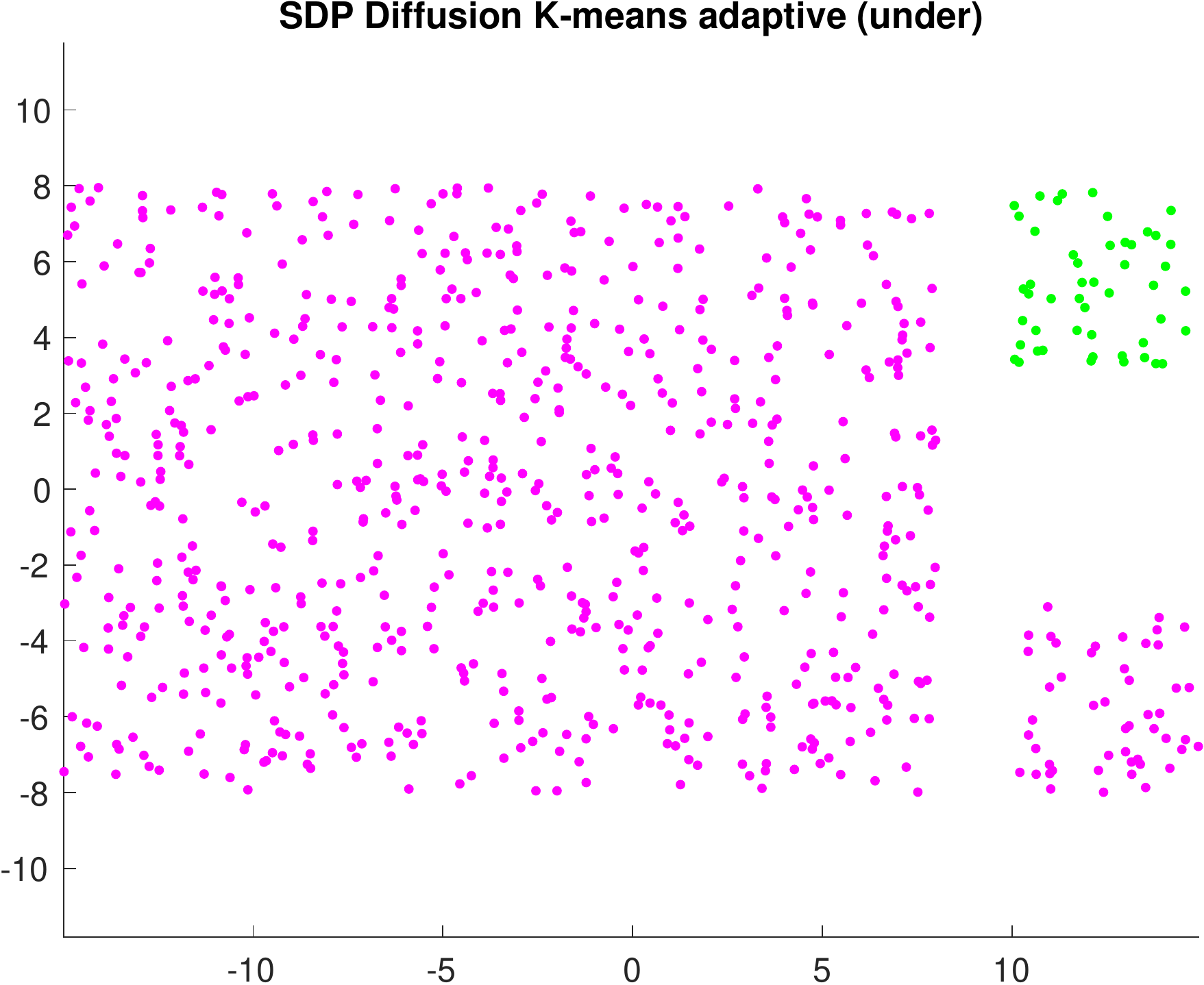}
    \caption{Plot of the estimated trace norms along the path of tuning parameter $\rho$ (on the log-scale) in the SDP relaxed regularized diffusion $K$-means clustering method~\eqref{eqn:clustering_Kmeans_sdp_unknown_K} for DGP=2. The clustered data are shown for estimated number of clusters 2 (under), 3 (correct), 4 (over).}
    \label{fig:diffusion_kmeans_lambda_demo_DGP2}
\end{figure}

\begin{figure}[h!] 
    \centering
        \includegraphics[scale=0.52]{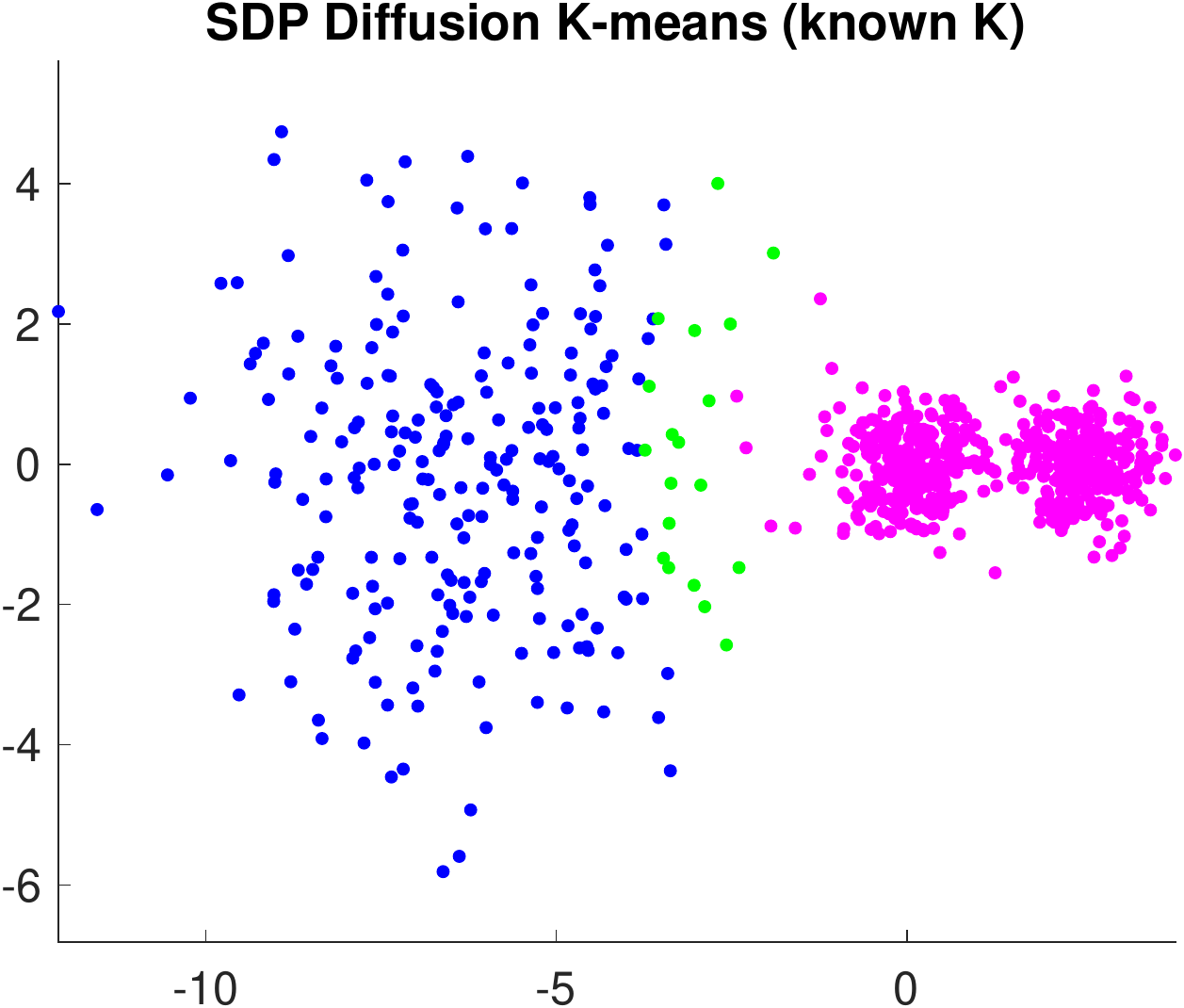}
    \includegraphics[scale=0.52]{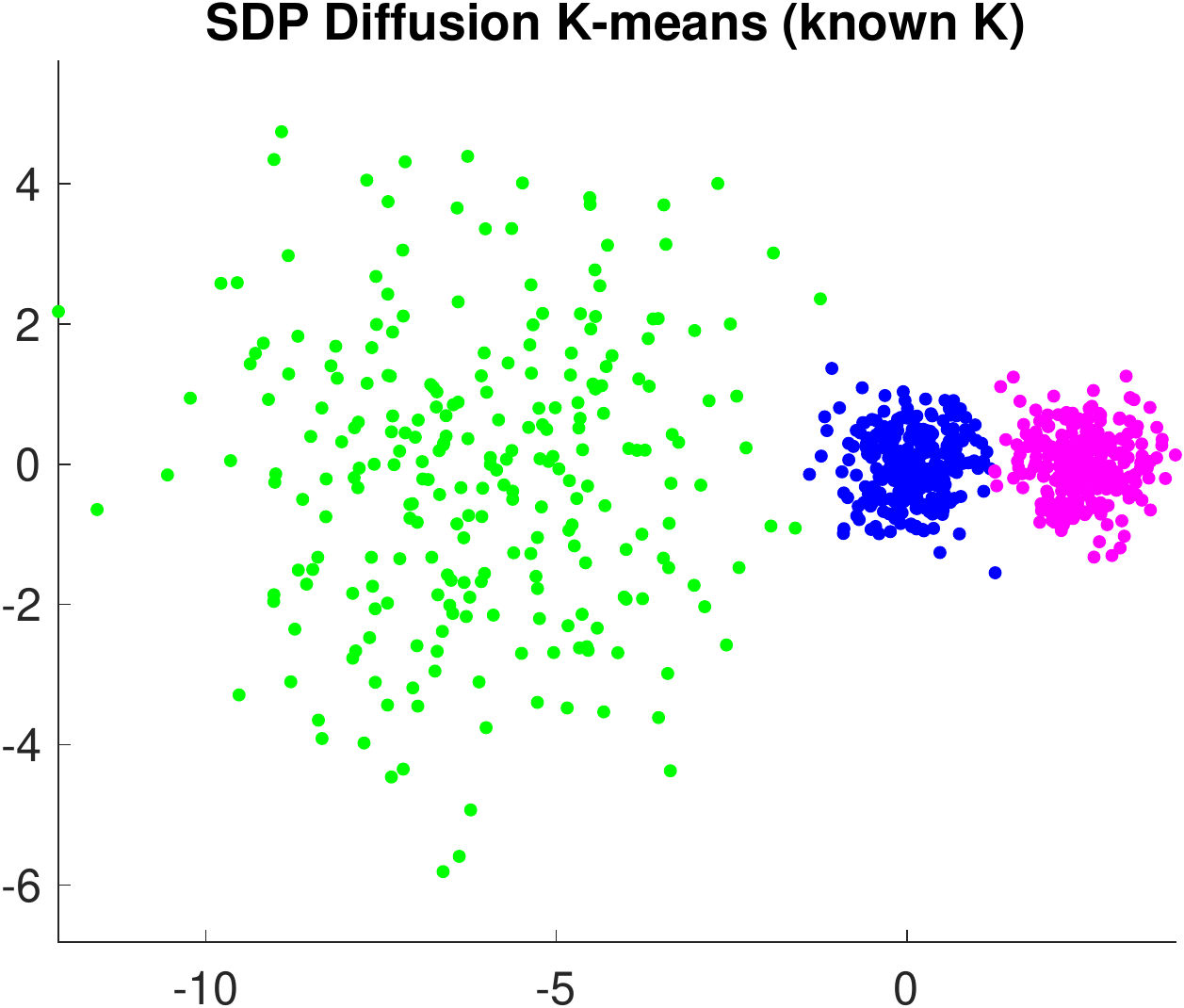}
    \caption{Plots of the SDP diffusion $K$-means clustering method without (left) and with (right) local scaling for the data generation mechanism DGP=3 in the simulation studies Section~\ref{sec:simulations}.}
    \label{fig:diffusion_kmeans_local_scaling_demo}
\end{figure}

\subsection{Localized diffusion $K$-means}
\label{subsec:localized_diffusion_Kmeans}

For clustering problems with different sizes, dimensions and densities, the diffusion $K$-means may have limitations since only one bandwidth parameter $h$ is used to control the local geometry on the domain. More precisely, according to our theory (for example, Theorem~\ref{thm:main} below), the optimal choice of the bandwidth parameter for a $q$-dimensional submanifold as one connected component in our clustering model is $h\asymp (\log \tilde{n}/\tilde{n})^{1/q}$, where $\tilde n$ corresponds to the sample size within this cluster and thus depends on the local cluster size or density level. 
Figure~\ref{fig:diffusion_kmeans_local_scaling_demo} demonstrates such an example for a mixture of three bivariate Gaussians, which consist of one larger Gaussian component with low density and two smaller Gaussian components with high density. Empirically, the diffusion $K$-means fails on this example (even after tuning) for the reason that the larger and smaller clusters have very different local densities. This motivates us to consider a variant of diffusion $K$-means, termed as \emph{localized diffusion $K$-means}, by using local adaptive bandwidth $h_i=h(X_i)$ for each $X_i$, $i\in[n]$. In particular, we adopt the self-tuning procedure from \cite{zelnik2005self} by setting $h_i$ to be $\|X_i-X_i^{(k_0)}\|$, where $X_i^{(k_0)}$ denotes the $k_0$-th nearest neighbor to $X_i$, and replacing $\mathcal K_n$ with $\mathcal K_n^\dagger=[\kappa^\dagger(X_i,X_j)]_{n\times n}$ (and accordingly replacing $D_n$ with $D_n^\dagger$ corresponding to the diagonal degree matrix associated with $\mathcal K^\dagger_n$) given by 
\begin{align*}
\kappa^\dagger (X_i,X_j) = \exp\Big(-\frac{\|X_i-X_j\|^2}{2h_ih_j}\Big),
\end{align*}
in the SDP~\eqref{eqn:clustering_Kmeans_sdp}.
Note that $\mathcal K^\dagger_n$ is generally no longer a positive semidefinite matrix. Intuitively for $i \in G_{k}^{*}$, the local scaling $h_{i}$ automatically adapts to the local density $p_{k}(X_{i})$ about $X_{i}$, the cluster size $n_k$ and the dimension $q_{k}$ for the $k$-th Riemannian submanifold. 
Specifically, for each cluster $k=1,\ldots,K$,  the $n_k$-by-$n_k$ submatrix $[\mathcal K_n^\dagger]_{G^\ast_kG^\ast_k}$ resembles the Gaussian kernel matrix with a homogeneous bandwidth $h_k \asymp (\log  n_k/n_k)^{1/q_k}$ that adapts to the local geometry in $\cM_k$. For points $X_i$ and $X_j$ belonging to distinct clusters  that are properly separated, $\kappa^\dagger (X_i,X_j)$ tends to be close to zero and is less affected by the choice of $h_i$ and $h_j$.
Heuristically, $h_{i}$ is larger for lower density regions where the degree function of $X_{i}$ is smaller so that the random walk can speed up mixing at such lower density regions. Overall, such a locally adaptive choice of bandwidth improves the mixing time of the random walk within each cluster, while leaves the between cluster jumping probabilities remaining small. As a consequence, the pairwise diffusion affinity matrix $A$ in our SDP formulation~\eqref{eqn:clustering_Kmeans_sdp} tends to exhibit a clearer block form reflecting the clustering structure.
 To compute $h_{i}$, we only need to specify $k_{0}$ to replace the (non-adaptive) bandwidth parameter $h$ whose value depend on the unknown cluster sizes $n_{k}$, dimensions $q_{k}$ of submanifolds $\cM_{k}$, and the underlying probability density functions $p_{k}$ on $\cM_{k}$. In contrast, the simple choice $k_{0} = \lfloor C \log{n} \rfloor$ guarantees that the local scaling $h_{i}$ adapts to the local density (cf. Theorem~\ref{thm:main_adaptive_h} in Section~\ref{sec:main_results}).

\section{Main results}
\label{sec:main_results}


In this section, we assume each $\cM_k$ is a compact connected $q_k$-dimensional Riemannian submanifold embedded in $\bb R^p$ with bounded diameter ${\rm Diam_k}$, absolute sectional curvature value ${\rm Sec_k}$, injectivity radius ${\rm Inj_k}$, and inverse reach ${\rm Rch_k^{-1}}$, where $\rm Rch_k:\,=\sup\big\{t > 0 :\, \forall x\in \bb R^p,\, \mbox{dist}(x,\cM_k)\leq t,\,  \exists!y\in\cM_k\, \mbox{s.t.~dist}(x,\cM_k)=\|x-y\|\big\}$. Throughout the rest of the paper, we assume that each $\mu_k$ has a ${\rm Lip_k}$-Lipschitz continuous density function $p_k$ with respect to the Riemannian volume measure on $\cM_k$, such that
\begin{align}\label{Eqn:density_condition}
\eta_k \leq p_k(x) \leq \frac{1}{\eta_k}\quad \mbox{for all }x\in \cM_k,
\end{align}
for some constant $\eta_k>0$.  For simplicity, we fix the Gaussian kernel in~\eqref{eqn:gaussian_kernel} for the diffusion $K$-means and its variants, which can be relaxed to more general isotropic kernel with exponential decay in the squared Euclidean distance (cf. Remark~\ref{rem:kernel_choice}).

\subsection{Exact recovery of diffusion $K$-means}
Let $\delta = \min_{1\leq k\neq k'\leq K} \|\cM_k-\cM_{k'}\|$, where $\|A-B\|=\inf_{x\in A, y\in B} \|x-y\|$ denote the (Euclidean) distance between two disjoint sets $A$ and $B$. Recall that the size of the true cluster $n_{k} =|G_{k}^{*}|$ for $k=1,\ldots,K$ and let $\underline{n} = \min_{1 \le k \le K} n_{k}$ denote the minimal cluster size. 

\begin{thm}[Exact recovery of SDP for diffusion $K$-means]\label{thm:main}
Fix any $\xi\geq 1$. Let $c_1,c_2, c_3, M_0$ be some positive constants only depending on $\{q_k,{\rm Sec_k},{\rm Inj_k}, {\rm Rch_k},\eta_k,{\rm Lip_k}\}_{k = 1}^{K}$, and $C_1,C_2$ only depending on  $\{q_k,{\rm Sec_k},{\rm Inj_k}, {\rm Rch_k},\eta_k,{\rm Lip_k}\}_{k = 1}^{K}$ and $\xi$. If $\underline n \geq M_0$ and 
\begin{equation}\label{eqn:bandwidth_condition_nonadaptive}
c_1 (\log n_k/ n_k)^{1/q_k}\,(\log n_k)^{\mathbf 1(q_k=2)/4}\leq h \leq c_2 \quad \mbox{for each } k\in[K],
\end{equation} 
where $\mathbf 1(q_k=2)$ equals to $1$ if $q_k=2$ and $0$ otherwise, then we can achieve exact recovery, that is $\hat Z= Z^\ast$, with probability at least $1- c_3K\,\underline{n}^{-\xi}$ as long as
\begin{align}\label{eqn:exact_recovery_condition_SDP_diffusion_Kmeans_LB_eigenval}
C_1\,nt\, \exp\Big\{-\frac{\delta^2}{2h^2}\Big\} +C_1\,\exp\bigg\{-\min_{1 \leq k \leq K}\Big\{\frac{\nu_{q_k}^2\lambda_1(\cM_k)}{4(q_k+2)(2\pi)^{q_k/2}}\Big\}\, h^2 t\bigg\}<  \frac{C_2}{n\, \max_{1 \leq k \leq K} \{n_k h^{q_k}\}},
\end{align}
where $\nu_d=\frac{\pi^{d/2}}{\Gamma(d/2+1)}$ denotes the volume of the $d$-dim unit ball, and $\lambda_1(\cM_k)>0$ denotes the second smallest eigenvalue of the drift Laplace-Beltrami operator (with the minus sign) on $\cM_k$ defined in~\eqref{eqn:drift_Laplace-Beltrami}.
\end{thm}

A proof of this theorem is provided in Section~\ref{Sec:proof_thm_main}. A crucial step in the proof is bounding from below the absolute spectral gap of the transition matrix associated with the restricted random walk $\mathcal W_n$ onto each submanifold $\mathcal M_k$ for $k=1,\ldots,K$.
We do so by applying a comparison theorem of Markov chains to connect this spectral gap with the eigensystem of the drift Laplace-Beltrami operator over the submanifold $\cM_k$ that captures the diffusion geometry. 
In particular, our proof borrows existing results \cite{burago2014graph,trillos2018error} on error estimates of using the spectrum of a random geometric graph to approximate the eigensystem of the drift Laplace-Beltrami operator in the numerical analysis literature.

\begin{rem}[Comments on the exact recovery condition~\eqref{eqn:exact_recovery_condition_SDP_diffusion_Kmeans_LB_eigenval}]\label{rem:comments_condition_main_thm}
We begin with comments on the second term on the left hand side of \eqref{eqn:exact_recovery_condition_SDP_diffusion_Kmeans_LB_eigenval}. We point out that this is essentially a mixing condition on random walks over the $K$ submanifolds. In particular, it is due to a relation between the mixing times of the heat diffusion process on each submanifold and its discretized random walk $\cW_n$ over vertices sampled from the submanifold. 

For simplicity, we illustrate this relation for $K = 1$ and uniform density $p_{1}$ on $\cM_{1} = \bS^{1}$, where $\bS^{1}$ is the unit circle in $\bR^{2}$.  As a one-dimensional compact smooth manifold (without boundary), $\bS^{1}$ can be (globally) parametrized by the angle $\theta \in [-\pi, \pi)$.
Under this parametrization, the density function $u(\tau,\theta)$ of the heat diffusion process on $\bS^{1}$, as a function of time $\tau$ and location $\theta$, is determined by the corresponding heat equation,
\begin{equation}
\label{eqn:heat_eq}
{\partial \over \partial \tau}u(\tau,\theta) + \Delta u(\tau,\theta) = 0, \quad (\tau, \theta) \in (0,\infty) \times \bS^{1},
\end{equation}
where $\Delta$ is the Laplace-Beltrami operator on $\bS^{1}$.
Under the same parametrization, the Laplace operator $\Delta = -{\rd^{2} \over \rd \theta^{2}}$ on $\bS^{1}$ (with the minus sign) admits the following eigen-decomposition 
\[
\Delta e^{\iota n \theta} = n^{2} e^{\iota n \theta}, \quad n=0,\pm1,\pm2,\dots,
\]
where $\iota = \sqrt{-1}$. That is, $(\lambda_{n}(\bS^{1}), e_{n}) := (n^{2}, e^{\iota n \theta})$ is an eigen-pair of $\Delta$, which implies that $\Delta$ is a positive semidefinite and unbounded operator on $L^{2}(\bS^{1})$ functions. 

Now we can solve the heat equation~\eqref{eqn:heat_eq} by expanding it with respect to this orthonormal basis.
More precisely, for any $f \in L^{2}(\bS^{1})$, the Fourier transform of $f$ is given by 
\[
f(\theta) = \sum_{n=-\infty}^{\infty} a_{n} e_{n} := \sum_{n=-\infty}^{\infty} \langle f, e_{n} \rangle e_{n}, 
\]
where $\langle f, g \rangle = (2\pi)^{-1} \int_{\bS^{1}} f(\theta) \overline{g(\theta)} \, \rd \theta$ is the standard inner product on $\bS^{1}$. 
Then the solution to heat equation \eqref{eqn:heat_eq} with the initial distribution $u(0,\theta) = f(\theta)$ is given by 
\[
u(\tau,\theta) = \sum_{n=-\infty}^{\infty} a_{n} e^{-n^{2}\tau} e_{n}.
\]
So if $\bS^{1}$ is insulated, then as $\tau \to \infty$ the heat flow has a constant equilibrium state with the value equal to the average of the initial heat distribution, namely $\lim_{\tau\to\infty} u(\tau,\theta) = (2\pi)^{-1} \int_{\bS^{1}} f(\theta)\,\rd\theta$. In particular, the second smallest eigenvalue $\lambda_{1}(\bS^{1})=1$ characterizes the mixing rate of the heat diffusion process. 

Now we consider the ``inverse Fourier transform" by expressing the solution $u$ in terms of  Green's function, also called the heat kernel, on $\bS^{1}$ as 
\begin{equation}
\label{eqn:heat_kernel_unit_circle}
H_{\bS^{1}}(\tau, \theta, \varphi) = \sum_{n=-\infty}^{\infty}  e^{-n^{2}\tau} {\tilde{e}}^{\iota n (\theta-\varphi)}, 
\end{equation}
where $ {\tilde{e}}^{\iota n \theta} = e_{n} / \sqrt{2\pi}$ is the rescaled orthonormal basis of $L^{2}(\bS^{1})$. Then we obtain that 
\[
u(\tau,\theta) = \mathscr H_{\bS^{1}}^{\tau} f (\theta), \mbox{ where }  \mathscr H_{\bS^{1}}^{\tau} f (\theta) := \int_{\bS^{1}} H_{\bS^{1}}(\tau, \theta, \varphi) f(\varphi) \,\rd \varphi 
\]
defines an (integral) heat diffusion operator on $\bS^{1}$. Then the Laplace operator on $\bS^{1}$ can be seen as the generator of the heat diffusion process: 
\[
\Delta f(\theta) = - \left.{\partial \over \partial \tau}u(\tau,\theta) \right|_{\tau=0} = - \left.{\partial \over \partial \tau}\mathscr H_{\bS^{1}}^{\tau}f(\theta) \right|_{\tau=0} = \lim_{\tau \to 0^{+}} {f(x) - \mathscr H_{\bS^{1}}^{\tau} f(x) \over \tau}. 
\]
Similarly, the normalized graph Laplacian $L_n = I_n - P_n = I_n- D_n^{-1} \mathcal H_n$ corresponding to the random walk $\cW_n$ over a random sample of $S_n$ on $\bS^{1}$ can also be seen as a discrete generator of $\cW_n$.

Recall that in the heat diffusion process, the second smallest eigenvalue $\lambda_{1}(\bS^{1})$ of its generator, i.e.,~the Laplace-Beltrami operator, characterizes the mixing rate of the heat diffusion process. Similarly, in the random walk $\cW_n$ (as a discretization of the heat process), the second smallest eigenvalue $\lambda_{j}(L_n)$ of its discrete generator, i.e.,~the normalized graph Laplacian operator $L_n$, characterizes its mixing rate. From Lemma \ref{lem:eigenval_convergence_normalized_graph_Laplacian}, the spectrum of these two operators are related in the sense that for each $j=1,2\dots,$ with probability at least $1-c_{1} n^{-c_{2}}$,  
\[
\lambda_{j}(\bS^{1}) \asymp h^{2} \lambda_{j}(L_n),
\] 
where $\lambda_{j}(\bS^{1})$ is the $j$-th eigenvalue of $\Delta$ and $\lambda_{j}(L_n)$ is the $j$-th eigenvalue of the normalized graph Laplacian. 
This means that we must change the time clock unit of the random walk on the graph by multiplying a factor of $h^{2}$ to approximate its underlying heat diffusion process. Thus the term $h^{2}t$ in the second term of  \eqref{eqn:exact_recovery_condition_SDP_diffusion_Kmeans_LB_eigenval} is the right time scale $\tau$ for running the heat diffusion process on the manifold, and we need $h^{2}t \to \infty$ for the heat diffusion process converges to an equilibrium distribution. On finite data, this means that the random walk converges to its stationary distribution over the points $S_n$ sampled from the submanifold. Using this correspondence, the second term on the left hand side of \eqref{eqn:exact_recovery_condition_SDP_diffusion_Kmeans_LB_eigenval} is a mixing condition on random walks over the $K$ submanifolds. 


The first term on the left hand side of \eqref{eqn:exact_recovery_condition_SDP_diffusion_Kmeans_LB_eigenval} can be seen as a separation requirement of the $K$ disjoint submanifolds $\cM_{1},\dots,\cM_{K}$. In particular, if $t = n^{\epsilon}$ for some $\epsilon > 0$ (i.e., we run the random walk in polynomial times/steps), then the minimal separation should obey 
\begin{equation}\label{eqn:lower_bound_Delta}
\delta \gtrsim h \sqrt{\log n}
\end{equation}
in order to achieve the exact recovery for the manifold clustering problem. 

Combining the two terms of~\eqref{eqn:exact_recovery_condition_SDP_diffusion_Kmeans_LB_eigenval}, we see that steps of the random walk must be properly balanced: we would like the random walk on the similarity graph to sufficiently mix within each cluster (second term of~\eqref{eqn:exact_recovery_condition_SDP_diffusion_Kmeans_LB_eigenval}), while it does not overly mix to merge the true clusters (first term of~\eqref{eqn:exact_recovery_condition_SDP_diffusion_Kmeans_LB_eigenval}). This reflects the {\it multi-scale} property of the diffusion $K$-means. 
\qed
\end{rem}

\begin{rem}[Comparisons with simple thresholding]
\label{rem:thresholding}
In view of the approximation property~\eqref{Eqn:approx_form_A_n} of the empirical diffusion affinity matrix $A_{n} = P_{n}^{2t} D_{n}^{-1}$ to a block-diagonal matrix in Remark~\ref{rem:intution_DKM}, one can show that a simple thresholding of the matrix $A_{n}$ also yields the exact recovery for a properly chosen threshold. Indeed, by the triangle inequality, we have for any $i,j \in G_k^\ast$ in the same cluster, 
\[
[A_n]_{ij} \geq N_{k}^{-1}  -\max_{1 \leq k \leq K} \big\|\,[A_n]_{G_k^\ast G_k^\ast}- N_k^{-1} \mathbf{1}_{G_k^\ast}\mathbf{1}_{G_k^\ast}^T\,\big\|_\infty.
\]
Choosing a threshold value $\gamma$ such that 
\begin{equation}
\label{eqn:thresholding_value}
\max_{1 \leq k\neq m \leq K} \|[A_n]_{G_k^\ast G_m^\ast}\|_\infty < \gamma < \min_{1\leq k\leq K}\Big\{\frac{1}{N_k}\Big\} - \max_{1 \leq k \leq K} \big\|\,[A_n]_{G_k^\ast G_k^\ast}- N_k^{-1} \mathbf{1}_{G_k^\ast}\mathbf{1}_{G_k^\ast}^T\,\big\|_\infty,
\end{equation}
one can completely separate the block-diagonal entries from the off-diagonal ones, thus achieving exact recovery. This argument leads to the following lemma. 

\begin{lem}[Master condition for thresholding to achieve exact recovery]
\label{lem:thresholding_master_bound}
If 
\begin{equation}
\label{eqn:thresholding_exact_recover_master_condition}
\max_{1 \leq k\neq m \leq K} \|[A_n]_{G_k^\ast G_m^\ast}\|_\infty + \max_{1 \leq k \leq K} \big\|\,[A_n]_{G_k^\ast G_k^\ast}- N_k^{-1} \mathbf{1}_{G_k^\ast}\mathbf{1}_{G_k^\ast}^T\,\big\|_\infty < \min_{1\leq k\leq K}\Big\{\frac{1}{N_k}\Big\},
\end{equation}
then thresholded estimator on the empirical diffusion affinity matrix $A_{n}$ with the threshold value satisfying~\eqref{eqn:thresholding_value} yields exact recovery.
\end{lem}
Note that condition~\eqref{eqn:thresholding_exact_recover_master_condition} in Lemma~\ref{lem:thresholding_master_bound} is slightly weaker than the master condition~\eqref{eqn:DKM_SDP_exact_recover_master_condition} of the SDP relaxed diffusion $K$-means in Lemma~\ref{lem:DKM_SDP_master_bound} (up to a factor of $K^{-1}$ for balanced clusters, say). On the other hand, thresholding has a tuning parameter $\gamma$ and it is unclear how to develop a principled procedure to choose the threshold value satisfying~\eqref{eqn:thresholding_value}. On the contrary, our SDP relaxed diffusion $K$-means only requires the knowledge of the number of cluster $K$ and it is tuning-free in that sense.
\qed
\end{rem}

It is worthy to note that the second smallest eigenvalue of the Laplace-Beltrami operator in condition \eqref{eqn:exact_recovery_condition_SDP_diffusion_Kmeans_LB_eigenval} of Theorem \ref{thm:main} can be regarded as characterizations of the connectedness of the submanifolds, where the latter can be formally quantified by the Cheeger isoperimetric constant defined as follows. 
\begin{defn}[Cheeger isoperimetric constant]
Let $\cM$ be a $q$-dimensional compact Riemannian manifold. Let $\mbox{Vol}(\cA)$ denote the volume of a $q$-dimensional submanifold $\cA \subset \cM$ and $\mbox{Area}(\cE)$ denote the $(q-1)$-dimensional area of a submanifold $\cE$. The {\it Cheeger isoperimetric constant} of $\cM$ is defined to be 
\[
\fh(\cM) = \inf_{\cE} \left\{ \frac{\mbox{Area}(\cE)}{\min(\mbox{Vol}(\cM_{1}), \mbox{Vol}(\cM_{2}))} \right\}, 
\]
where the infimum of the normalized manifold cut (in the curly brackets) is taken over all smooth $(q-1)$-dimensional submanifolds $\cE$ of $\cM$ that cut $\cM$ into two disjoint submanifolds $\cM_{1}$ and $\cM_{2}$ such that $\cM = \cM_{1} \bigsqcup \cM_{2}$. 
\end{defn}


In words, $\fh(\cM)$ quantifies the minimal area of a hypersurface that bisects $\cM$ into two disjoint pieces (cf.~Figure~\ref{fig:diffussion_dist}). Smaller values of $\fh(\cM)$ mean that $\cM$ is less connected -- in particular, $\fh(\cM) = 0$ implies that there are two disconnected components in $\cM$. The Cheeger isoperimetric constant may also be analogously defined for a graph and its value (i.e., the conductance of the graph) is closely related to the normalized graph cut problem. Suppose we have an i.i.d. sample $X_{1},\dots,X_{n}$ drawn from the uniform distribution on $\cM$ and $\cG_{n}$ is the neighborhood random graph with an edge between $X_{i}$ and $X_{j}$ if $\|X_{i}-X_{j}\| \le h$. It is shown in \cite{Arias-CastroPelletierPudlo2012_AAP} that the normalized graph cut (after a suitable normalization) converges to the normalized manifold cut, yielding an asymptotic upper bound on the conductance of $\cG_{n}$ based on $\fh(\cM)$. See also~\cite{Trillos:2016_JMLR} for improved results.

\begin{cor}\label{cor:main}
Under the setting of Theorem \ref{thm:main}, if 
\begin{align}\label{eqn:exact_recovery_condition_SDP_diffusion_Kmeans_Cheeger}
C_1\,nt\, \exp\Big\{-\frac{\delta^2}{2h^2}\Big\} +C_1\,\exp\Big\{- \underline{\fh}^{2}\, h^2 t\Big\} <  \frac{C_2}{n\,\max_{1 \leq k \leq K} \{n_kh^{q_k}\}},
\end{align}
where $\underline{\fh} = \min_{1 \leq k \leq K} \frac{\nu_{q_k}\,\fh(\cM_k)}{4\sqrt{q_k+2}\,(2\pi)^{q_k/4}}$, then $\hat Z= Z^\ast$ with probability at least $1- c_3K\,\underline{n}^{-\xi}$. 
\end{cor}

Since $\delta$ reflects the separation of the submanifolds of $S = \bigsqcup_{k=1}^{K} \cM_{k}$ and $\fh(\cM_k)$ reflects the degree of connectedness of the submanifold $\cM_{k}$, the (overall) hardness of the manifold clustering problem is determined by $(\delta, \underline{\fh})$. In particular, if $t = n^{\epsilon}$ for some $\epsilon > 0$, then we require that 
\[
\underline{\fh} \gtrsim {1 \over h} \sqrt{\log n \over n^{\epsilon}}, 
\]
in addition to \eqref{eqn:lower_bound_Delta}. Our results in the rest subsections can also be stated via this geometric quantity of the Cheeger isoperimetric constant.

\begin{rem}[Comments on the Gaussian kernel and bandwidth parameter]\label{rem:kernel_choice}
In Theorem \ref{thm:main} and Corollary \ref{cor:main}, the kernel $k$ is assumed to be the Gaussian kernel in \eqref{eqn:gaussian_kernel}. However, these exact recovery results do not rely on the particular choice of the Gaussian kernel. Specifically, Theorem \ref{thm:main} and Corollary \ref{cor:main} still hold, as long as the kernel is isotropic and satisfies the exponential decay in the squared Euclidean distance. For  the heat kernel on $\bR$ (i.e., Green's function associated with the heat equation on $\bR$)
\begin{equation}
\label{eqn:heat_kernel_real_line}
H(\tau,x,y) = (4\pi \tau)^{-1/2} e^{-{(x-y)^{2} \over 4\tau}}, \quad x,y \in \bR, 
\end{equation}
it can be viewed as an approximation to the short time dynamics of the heat kernel~\eqref{eqn:heat_kernel_unit_circle} on $\bS^{1}$ in Remark~\ref{rem:comments_condition_main_thm} as $\tau \to 0^{+}$ (cf. Chapter 1 in \cite{Rosenberg1997}). Hence, we can approximate the short time behavior of the heat flow on the compact manifold $\bS^{1}$ by that of the non-compact manifold $\bR$, where the latter is governed by the Gaussian heat kernel on $\bR$. Setting $h^{2} = 2\tau$ in \eqref{eqn:heat_kernel_real_line} and noticing that the normalization $(4\pi \tau)^{-1/2}$ does not affect the results in Theorem \ref{thm:main} and Corollary \ref{cor:main} since the SDP solution in \eqref{eqn:clustering_Kmeans_sdp} is invariant under scaling. Thus the bandwidth parameter in the Gaussian kernel in \eqref{eqn:gaussian_kernel} has the time scale interpretation in terms of the heat flow dynamics, in addition to capturing the local neighborhood geometry of the submanifolds.
\qed
\end{rem}

\subsection{Exact recovery of the regularized and localized diffusion $K$-means}
In this subsection, we extend the exact recovery results to the two variants of the diffusion $K$-means. First, we consider the regularized diffusion $K$-means~\eqref{eqn:clustering_Kmeans_sdp_unknown_K} that does not require knowledge of the true number of clusters $K$.

\begin{thm}[Exact recovery of SDP for regularized diffusion $K$-means]\label{thm:main_adaptive_lambda}
Suppose all conditions in Theorem~\ref{thm:main} are true. In addition, if the regularization parameter satisfies
\begin{align}\label{Eqn:lambda_condition}
C_1\,nt\, \exp\Big\{-\frac{\delta^2}{2h^2}\Big\} &+C_1\,\exp\bigg\{-\min_{1 \leq k \leq K}\Big\{\frac{\nu_{q_k}^2\lambda_1(\cM_k)}{4(q_k+2)(2\pi)^{q_k/2}}\Big\}\, h^2 t\bigg\} < \rho \leq \frac{C_2}{n\max_{1 \leq k \leq K}\{n_kh^{q_k}\}},
\end{align}
then we can achieve exact recovery for $\tilde{Z}$ from the regularized diffusion $K$-means  with probability at least $1- c_3K\,\underline{n}^{-\xi}$.
\end{thm}

\begin{rem}[Implication on our $\rho$ selection Algorithm~\ref{alg1}]
Condition~\eqref{Eqn:lambda_condition}, as a sufficient condition for the exact recovery, provides some justification of our $\rho$ selection Algorithm~\ref{alg1}, in particular, the reason of why using the logarithmic scale. More precisely, we observe from Lemma~\ref{lem:feasibility_SDP_lambda_infinity} that an upper bound for $\rho$ to produce non-trivial clustering is $n^{-1}$ times the largest eigenvalue of the affinity matrix $A$ in the SDP, which is of order $n^{-1}\max_{k}\{n_kh^{q_k}\}^{-1} =\mathcal O((n\log n)^{-1})$ (from the approximating form~\eqref{Eqn:approx_form_A_n} and Lemma~\ref{Lemma:total_degree} in the proof of Theorem~\ref{thm:main}). As a consequence, in the original scale, the interval length of those $\rho$ that underestimates $K$ is of order $(n\log n)^{-1}$, which is comparable to the range of $\rho$ corresponding to exact recovery (correct $K$) implied by~\eqref{Eqn:lambda_condition} as $(n\log n)^{-1}$. On the other hand, the range of $\rho$ corresponding to exact recovery will dominate if we instead consider the logarithmic scale. Precisely, on the logarithmic scale, the interval length for underestimating $K$ is of order $\log n$, while the range of $\log \rho$ implied by~\eqref{Eqn:lambda_condition} becomes $\log n - C'\,\big(\log n - \min\big\{\delta^2/(2h^2), \min_{1 \leq k \leq K}\{\lambda_1(\cM_k)\}\, h^2 t\big\}\big)$, which is of order $\Omega(n^{\iota})$ for some constant $C'>0$ and $\iota>0$ as long as both $\delta^2/h^2$ and $h^2t$ are bounded below by $n^{\iota}$. (The latter requirement can be easily satisfied, for example, in our simulations the choice of $t=n^{1.2}$ is good enough to produce robust results.) In particular, this suggests that the interval length of $\log \rho$ for exact recovery is proportional to $\delta^2/h^2$, which can be viewed as a signal-to-noise ratio characteristic. 
\qed
\end{rem}


Now let us turn to the localized diffusion $K$-means that locally selects the node-wise bandwidth adapting to the local geometric structure.

\begin{thm}[Exact recovery of SDP for localized diffusion $K$-means]\label{thm:main_adaptive_h}
Let $\delta_{kk'}=\|\cM_k-\cM_{k'}\|$. 
Fix any $\xi\geq 1$. Let $c_1,c_2,c_3, c_4, M_0$ be some positive constants only depending on $\{q_k,{\rm Sec_k},{\rm Inj_k}, {\rm Rch_k},\eta_k,{\rm Lip_k}\}_{k = 1}^{K}$, and $C,C',C_1,C_2$ only depending on $\{q_k,{\rm Sec_k},{\rm Inj_k}, {\rm Rch_k},\eta_k$, ${\rm Lip_k}\}_{k = 1}^{K}$ and $\xi$.
If  $\underline n \geq M_0$, the number of neighbor parameter $k_0$ satisfies $k_0=\lfloor C\log n\rfloor$, and $$\delta_{kk'} \geq C'\, \max\{(\log n/ n_k)^{1/q_k},(\log n/ n_{k'})^{1/q_{k'}}\}$$ for each distinct $k,k'\in[K]$, then with probability at least $1- c_3K\,\underline{n}^{-\xi}$, the followings are true. 
\newline
(1)~For each $i\in G^\ast_k$, its local bandwidth parameter $h_i$ satisfies
\begin{align}\label{Eqn:local_h_nounds}
c_1 (\log n/ n_k)^{1/q_k}\leq h_i \leq c_2 (\log n/ n_k)^{1/q_k}.
\end{align}
\newline
(2)~We can achieve exact recovery for $\tilde{Z}$ from the localized diffusion $K$-means as long as
\begin{align}
&C_1\,nt\, \exp\Big\{-c_4\,\Big(\min_{k,k'\in[K]}\frac{\delta_{kk'}}{\max\{(\log n/ n_k)^{1/q_k},(\log n/ n_{k'})^{1/q_{k'}}\}}\Big)^2\Big\}\notag \\
&\qquad\qquad\qquad+C_1\,\exp\Big\{-c_4\, \min_{1 \leq k \leq K}\{\lambda_1(\cM_k)\, (\log n/ n_k)^{2/q_k}\}\, t\Big\} <  \frac{C_2}{n\log n}.\label{eqn:local_condition}
\end{align}
\end{thm}

\begin{rem}[Advantages of the localization]\label{rem:localization}
The first result~ Part (1) in Theorem~\ref{thm:main_adaptive_h} shows that our localized selection scheme via nearest neighbors truly leads to bandwidth adaptation to the unknown submanifold dimension $q_k$ and the unknown true cluster size $n_k$, by only sacrificing a $\log n$ term (from $\log n_k$ to $\log n$ as compared with the optimal bandwidth choice $(\log n_k/ n_k)^{1/q_k}$ from Theorem~\ref{thm:main}).
The second result~Part (2) in Theorem~\ref{thm:main_adaptive_h} indicates the advantages of using the localized node-wise bandwidth, by comparing the condition
~\eqref{eqn:local_condition} with those in Theorem~\ref{thm:main}. In particular, in order for the lower bound condition on the global bandwidth $h$ in Theorem~\ref{thm:main} to hold, the smallest $h$ would be $\max_{k} h_k$, where $h_k=(\log n_k/n_k)^{q_k}$ denotes the optimal bandwidth in the $k$-th cluster $\cM_k$. Note that this lower bound on $h$ is uniformly larger than the magnitudes of localized bandwidth provided in~\eqref{Eqn:local_h_nounds}. As a consequence, this large $h$ would require the same separation condition as $\delta_{kk'} \geq \max_{k}h_k$ for each pair $(\cM_k,\cM_{k'})$ of distinct clusters.
In comparison, the new sufficient condition~\eqref{eqn:local_condition} for exact recovery only needs a cluster-dependent separation condition as $\delta_{kk'} \geq \max\{h_k,h_{k'}\}$, which can be substantially weaker than $\delta_{kk'} \geq \max_{k}h_k$ if clusters are highly unbalanced with unequal sizes and mixed dimensions.
\qed
\end{rem}

Finally, we can further combine the regularized  diffusion $K$-means with local adaptive bandwidths into the \emph{localized and regularized diffusion $K$-means}. The following result is an immediate consequence by combining the proofs of Theorem~\ref{thm:main_adaptive_lambda} and  Theorem~\ref{thm:main_adaptive_h}, and thus its proof is omitted.

\begin{thm}[Exact recovery of SDP for localized and regularized diffusion $K$-means]\label{thm:main_adaptive}
Suppose all conditions in Theorem~\ref{thm:main} and Theorem~\ref{thm:main_adaptive_h} are true. In addition, if the regularization parameter satisfies
\begin{align}
&C_1\,nt\, \exp\Big\{-c_4\,\Big(\min_{k,k'\in[K]}\frac{\delta_{kk'}}{\max\{(\log n/ n_k)^{1/q_k},(\log n/ n_{k'})^{1/q_{k'}}\}}\Big)^2\Big\} \notag\\
&\qquad\qquad\qquad+C_1\,\exp\Big\{-c_4\, \min_{1 \leq k \leq K}\{\lambda_1(\cM_k)\, (\log n/ n_k)^{2/q_k}\}\, t\Big\}< \rho \leq \frac{C_2}{n\log n},\label{Eqn:local_lambda_condition}
\end{align}
then we can achieve exact recovery for $\tilde{Z}$ from the localized and regularized diffusion $K$-means with probability at least $1- c_3K\,\underline{n}^{-\xi}$.
\end{thm}

\begin{rem}[Comments on bandwidth adaptivity]
\label{rem:bandwidth_adaptivity}
It is interesting to note the signal separation and random walk mixing of the localized diffusion $K$-means and its regularized version are adaptive to local probability density and local geometric structures of the Riemannian submanifolds. In \cite{Arias-Castro2011_IEEETIT}, nearly-optimal exact recovery of a collection of clustering methods based on pairwise distances of data is derived under a condition that the {\it minimal} signal separation strength over all pairs of submanifolds is larger than a threshold (even for their local scaling version, cf. Proposition 3 therein). Thus results established in \cite{Arias-Castro2011_IEEETIT} are non-adaptive to the local density and (geometric) structures of the submanifolds.  In addition, the $\max_k\{n_k h_k^{q_k}\}$ on the right hand side of \eqref{eqn:exact_recovery_condition_SDP_diffusion_Kmeans_LB_eigenval} now reduces to $\log(n)$ as in \eqref{Eqn:local_lambda_condition}. This means that the localized diffusion $K$-means tends to increase the signal-to-noise ratio (cf.~Remark~\ref{rem:localization}) as well as the upper bound on $\rho$ for exact recovery, thereby widening the interval length of $\log(\rho)$ corresponding to the true clustering structure and improves the performance of the $\rho$ selection Algorithm~\ref{alg1}.
\qed
\end{rem}

%
%

\section{Simulations}
\label{sec:simulations}

In this section, we assess the empirical performance of the diffusion $K$-means on some simulation examples. We generate $n=768$ data points from the following three data generation mechanisms (DGPs). 

\begin{enumeraten}
\item The clustering structure contains three disjoint submanifolds: 
\begin{itemize}
\item $\cM_{1} = \mbox{unit disk}$ and $n/4$ data points are uniformly sampled on $\cM_{1}$,
\item $\cM_{2} = \mbox{annulus with radius } 2.5$ and $n/4$ data points are uniformly sampled on $\cM_{2}$,
\item $\cM_{3} = \mbox{annulus with radius } 4$ and $n/2$ data points are uniformly sampled on $\cM_{3}$, 
\end{itemize}
where all $\cM_{1},\cM_{2},\cM_{3}$ are centered at the origin $(0,0)$. 

\item The clustering structure contains three disjoint rectangles:  
\begin{itemize}
\item $\cM_{1} = \{(-15,-8), (-15,8), (-8,8), (8,8)\}$,
\item $\cM_{2} = \{(10,3), (10,8), (15,3), (15,8)\}$,
\item $\cM_{3} = \{(10,-8), (10,-3), (15,-8), (15,-3)\}$,
\end{itemize}
where data points are uniformly distributed on $\cM_{1} \bigsqcup \cM_{2} \bigsqcup \cM_{3}$. 

\item The clustering structure is a mixture of three bivariate Gaussians: 
\[
\alpha_{1} N(\mu_{1}, \sigma_{1}^{2} \Id_{2}) + \alpha_{2} N(\mu_{2}, \sigma_{2}^{2} \Id_{2}) + \alpha_{3} N(\mu_{3}, \sigma_{3}^{2} \Id_{2}),
\]
where $(\alpha_{1}, \alpha_{2}, \alpha_{3}) = (1/3, 1/3, 1/3)$, $\mu_{1} = (-6,0), \mu_{2} = (0,0), \mu_{3} = (2.5,0)$, $\sigma_{1} = 2$, and $\sigma_{2} = \sigma_{3} = 0.5$. 
\end{enumeraten}

Our simulation setups are similar to \cite{zelnik2005self,NadlerGalun2006_NIPS}. Note that the sampling density in DGP 1 and 2 is uniform on the disjoint submanifolds, the hardness of the problems is mainly determined by the geometry, and we thus expect the diffusion $K$-means and its localized version can both succeed in these two cases. In addition, since DGP 1 contains two annuli that are less connected than the rectangles and ellipsoids, we expect that, for the localized diffusion $K$-means (with self-tuned bandwidths), more random walk steps are needed for DGP 1 to correctly identify the clusters than those for DGP 2 and DGP 3. In our simulation studies, we use $t = n^{2}$ for DPG 1 and $t = n^{1.2}$ for both DGP 2 and DGP 3 (all with local scaling). Further, DGP 3 has a mixture of Gaussian densities, the local scaling is expected to improve the performance of the diffusion $K$-means. In fact, we have observed in Figure~\ref{fig:diffusion_kmeans_local_scaling_demo} that the diffusion $K$-means without local scaling does not work for DGP 2. It is also known that spectral clustering methods fail on such setup \cite{NadlerGalun2006_NIPS}. Thus we do not report results on DGP 2 without local scaling for all competing methods since it does not provide meaningful comparisons with other setups. 

For the SDP relaxed diffusion $K$-means clustering methods, we report the $\ell^{1}$ estimation error of for estimating the true clustering membership $Z^{*}$ and the (normalized) Hamming distance error for classifying the clustering labels. For an SDP estimator $\hat{Z}$, the $\ell^{1}$ estimation error is defined as $n^{-1} \|\hat{Z}-Z^{*}\|_{1}$. To get the clustering labels from the SDP estimator $\hat{Z}$, we extract its top $K$ eigenvectors and perform a $K$-means rounding algorithm to get the estimated partition $\hat{G}_{1},\dots,\hat{G}_{K}$. Then the classification error is defined as $n^{-1} \sum_{k=1}^{K} \|\vone_{\hat{G}_{k}} - \vone_{G^{*}_{k}}\|_{1}$, where $\vone_{G_{k}} = (1_{(X_{1} \in G_{k})},\dots,1_{(X_{n} \in G_{k})})^{T}$.

In each setup, our results are reported on 1,000 simulations. For brevity, DKM stands for the diffusion $K$-means, RDKM for the (nuclear norm) regularized diffusion $K$-means, LDKM for the localized diffusion $K$-means, and LRDKM for the localized and regularized diffusion $K$-means. In the cases of no local scaling, the steps of random walks is fixed as $t = n^{1.2}$ in all setups. In the cases of local scaling, the nearest neighborhood size is chosen as $\lfloor \log{n} \rfloor$ for DPG 1 and DGP 3, and as $\lfloor 0.5 \log{n} \rfloor$ for DGP 2. 

For the comparison purpose, we also include three spectral clustering methods: the unnormalized spectral clustering (SC-UN), the random walk normalized spectral clustering (SC-RWN) \cite{ShiMalik2000_IEEEPAMI}, a symmetrically normalized spectral clustering (SC-NJW) proposed in \cite{NgJordanWeiss2001_NIPS}. For each spectral clustering method, we also consider their localized versions (LSC-UN, LSC-RWN, LSC-NJW) by replacing the kernel matrix $\mathcal K_{n}$ with $\mathcal K^{\dagger}_{n}$. 

We can draw several observations from the simulation studies. First, the estimation error agrees well with our exact recovery theory for SDP relaxed DKM and LDKM, given the number of clusters (cf. Table~\ref{tab:estimation_errors}). Second, all methods works relatively better for DGP 1 and DGP 2 since the separation signal strength is stronger than DGP 3 (cf. Table~\ref{tab:classification_errors}). Third, the RDKM and LRDKM perform well in selecting the true number of clusters (cf. Table~\ref{tab:percentages_correctly_estimated_noc}) using Algorithm~\ref{alg1}.

We also modify DGP 3 to make the problem harder. We consider the mixture of three Gaussian with parameters $(\alpha_{1}, \alpha_{2}, \alpha_{3}) = (1/4, 1/4, 1/2)$, $\mu_{1} = (-6,0), \mu_{2} = (0,0), \mu_{3} = (1.45,0)$, $\sigma_{1} = 2$, and $\sigma_{2} = \sigma_{3} = 0.5$. This setup is denoted as DGP 3'. For DGP 3', LDKM has much smaller classification errors than all spectral methods with local scaling (i.e., LSC-UN, LSC-RWN, LSC-NJW); see last column of Table~\ref{tab:classification_errors}.

 \begin{table}[h!]
  \begin{center}
    \caption{$\ell^{1}$ estimation errors of the SDP solutions of various diffusion $K$-means clustering methods.}
    \label{tab:estimation_errors}
    \begin{tabular}{c|ccc}
    \hline
      \multirow{2}{*}{Method} & \multicolumn{3}{c}{Estimation error} \\
      & DGP=1 & DGP=2 & DGP=3 \\
    \hline
     DKM & $4.7642 \times 10^{-6}$ & $3.3258 \times 10^{-4}$ & {\bf --} \\
     LDKM & $5.2835 \times 10^{-5}$ & 0.0049 & 0.0451 \\
    \hline  
    \end{tabular}
  \end{center}
\end{table}

\begin{table}[h!]
  \begin{center}
    \caption{Classification errors of various diffusion $K$-means and spectral clustering methods.}
    \label{tab:classification_errors}
    \begin{tabular}{c|cccc}
    \hline
      \multirow{2}{*}{Method} & \multicolumn{4}{c}{Classification error} \\
      & DGP=1 & DGP=2 & DGP=3 & DGP=3' \\
    \hline
     DKM & 0 & $1.3021 \times 10^{-4}$ & {\bf --} & {\bf --} \\
     SC-UN & 0 & 0 & {\bf --} & {\bf --} \\
     SC-RWN & 0 & 0 & {\bf --} & {\bf --} \\
     SC-NJW & 0 & 0 & {\bf --}  & {\bf --} \\
     LDKM & 0 & 0.0018 & 0.0086 & 0.0594 \\
     LSC-UN & 0 & $4.7917 \times 10^{-4}$ & 0.0098 & 0.0801 \\
     LSC-RWN & 0 & 0.0016 & 0.0105 & 0.0802 \\
     LSC-NJW & 0 & 0.0399 & 0.0084 & 0.0884 \\
    \hline  
    \end{tabular}
  \end{center}
\end{table}

\begin{table}[h!]
  \begin{center}
    \caption{Percentages of correctly estimated number of clusters by the regularized diffusion $K$-means and its local scaling version using Algorithm~\ref{alg1} among 1,000 simulations.}
    \label{tab:percentages_correctly_estimated_noc}
    \begin{tabular}{c|ccc}
    \hline
   Method  & DGP=1 & DGP=2 & DGP=3 \\
    \hline
   RDKM  & 94.30\% & 95.10\% & {\bf --} \\
   LRDKM  & 99.20\% & 83.10\% & 97.70\% \\
    \hline  
    \end{tabular}
  \end{center}
\end{table}



\section{Proofs}
\label{sec:proofs}

Recall that $n_k=|G_k^\ast|$ is the size of $k$-th true cluster index set $G_k^\ast$, and let $N_k = \sum_{i,j\in G_k^\ast} \kappa(X_i,\,X_j)$ denote the total within-weight in $G_k^\ast$. For any subset $G\subset [n]$, we use $\mathbf{1}_{G}$ to denote the all-one vector whose size equal to the size of $G$. 

\subsection{Proof of Theorem~\ref{thm:main}}\label{Sec:proof_thm_main}
For simplicity of notation, we use $A:=A_n$ to denote the empirical diffusion affinity matrix $P_n^{2t}D_n^{-1}$ in the proof, and recall
\[
\begin{gathered}
\hat{Z} = \argmax \left\{ \langle A, Z \rangle : Z \in \sC_K \right\} \\
\qquad \mbox{with } \sC_K = \{Z \in \bR^{n \times n} : Z^{T} = Z, Z \succeq 0, \tr(Z) = K, Z \vone_{n} = \vone_{n}, Z \geq 0 \}.
\end{gathered}
\]
At a high level, our strategy is to show that for suitably large $t\in\bN_{+}$, the matrix $A_n$ tends to become close to a block-diagonal matrix, 
where each diagonal block tends to be a constant matrix (cf. equation~\eqref{Eqn:approx_form_A_n}). 
Based on this approximation, we expect the global optimum $\hat Z$ to share a similar block-diagonal structure, thereby recovers the true membership matrix $Z^\ast$ in~\eqref{eqn:Kmeans_true_membership_matrix} which takes the form of
\begin{align*}
Z^\ast = \begin{pmatrix} 
\displaystyle \frac{1}{n_1} \mathbf{1}_{G_1^\ast} \mathbf{1}^T_{G_1^\ast} & 0 & \cdots & 0\\
0 & \displaystyle \frac{1}{n_2} \mathbf{1}_{G_2^\ast} \mathbf{1}^T_{G_2^\ast} & \cdots & 0\\
\vdots & \vdots & \ddots & \vdots \\
0 & \cdots & 0 & \displaystyle \frac{1}{n_K} \mathbf{1}_{G_K^\ast} \mathbf{1}^T_{G_K^\ast}
\end{pmatrix}.
\end{align*}

To put this intuition in a technical form, since $Z^\ast$ defined in~\eqref{eqn:Kmeans_true_membership_matrix} is also a feasible solution belonging to the convex set $\sC_K$, we have by the optimality of $\hat Z$ that
\begin{align}\label{Eqn:basic_ineq}
0\leq \langle A_n,\, \hat Z - Z^\ast\rangle =\sum_{1\leq k\neq m\leq K} \big\langle\, [A_n]_{G_k^\ast G_m^\ast},\, [\hat Z - Z^\ast]_{G_k^\ast G_m^\ast}\,\big\rangle + \sum_{k=1}^K \big\langle\, [A_n]_{G_k^\ast G_k^\ast},\, [\hat Z - Z^\ast]_{G_k^\ast G_k^\ast}\,\big\rangle.
\end{align}
We analyze the two sums separately as follows.

\vspace{0.5em}
\noindent {\bf The first sum:}  By noticing that $Z^\ast_{G_k^\ast G_m^\ast}$ is a zero matrix for each pair $k\neq m \in [K]$, we have the following bound
\begin{align}
&\sum_{1\leq k\neq m\leq K} \big\langle\, [A_n]_{G_k^\ast G_m^\ast},\, [\hat Z - Z^\ast]_{G_k^\ast G_m^\ast}\,\big\rangle \notag \\
= &\,\sum_{1\leq k\neq m\leq K} \big\langle\, [A_n]_{G_k^\ast G_m^\ast},\, \hat Z_{G_k^\ast G_m^\ast}\,\big\rangle 
\leq  \max_{1\leq k\neq m\leq K} \|[A_n]_{G_k^\ast G_m^\ast}\|_\infty \, \sum_{1\leq k\neq m\leq K} \|\hat Z_{G_k^\ast G_m^\ast}\|_1,\label{eqn:k_neq_m}
\end{align}
where the leading factor $\max_{k\neq m} \|[A_n]_{G_k^\ast G_m^\ast}\|_\infty$ is expected to be small due to the approximating structure~\eqref{Eqn:approx_form_A_n}.

\vspace{0.5em}
\noindent {\bf The second sum:} Since we expect the $k$th block $[A_n]_{G_k^\ast G_k^\ast}$ of $A_n$ in the diagonal to be close to $N_k^{-1} \mathbf{1}_{G_k^\ast} \mathbf{1}^T_{G_k^\ast}$ (cf.~\eqref{Eqn:approx_form_A_n}), we can subtract and add the same term to decompose it into
\begin{align*}
&\sum_{k=1}^K \big\langle\, [A_n]_{G_k^\ast G_k^\ast},\, [\hat Z - Z^\ast]_{G_k^\ast G_k^\ast}\,\big\rangle\\
= &\, \sum_{k=1}^K \big\langle\, [A_n]_{G_k^\ast G_k^\ast} - N_k^{-1} \mathbf{1}_{G_k^\ast}\mathbf{1}_{G_k^\ast}^T,\, [\hat Z - Z^\ast]_{G_k^\ast G_k^\ast}\,\big\rangle
+ \sum_{k=1}^K N_k^{-1} \,\big\langle\, \mathbf{1}_{G_k^\ast}\mathbf{1}_{G_k^\ast}^T,\, [\hat Z - Z^\ast]_{G_k^\ast G_k^\ast}\,\big\rangle.
\end{align*}
The first term on the right hand side can be bounded by applying H\"{o}lder's inequality,
\begin{align*}
\sum_{k=1}^K& \big\langle\, [A_n]_{G_k^\ast G_k^\ast} - N_k^{-1} \mathbf{1}_{G_k^\ast}\mathbf{1}_{G_k^\ast}^T,\, [\hat Z - Z^\ast]_{G_k^\ast G_k^\ast}\,\big\rangle \\
&\,\leq \max_{1 \leq k \leq K} \big\|\,[A_n]_{G_k^\ast G_k^\ast}- N_k^{-1} 1_{G_k^\ast}1_{G_k^\ast}^T\,\big\|_\infty \, \sum_{k=1}^K \big\|\, [\hat Z -Z^\ast]_{G_k^\ast G_k^\ast}\big\|_1,
\end{align*}
where again the leading factor $\max_{k} \big\|\,[A_n]_{G_k^\ast G_k^\ast}- N_k^{-1} \mathbf{1}_{G_k^\ast}\mathbf{1}_{G_k^\ast}^T\,\big\|_\infty $ is expected to be small due to the approximating structure~\eqref{Eqn:approx_form_A_n}.
Now consider the second term. By the definition of $\sC_K$, the sums of entries in each row of $\hat Z$ and $Z^\ast$ are equal (to one), we have for fixed $k\in[K]$ and each $i\in G^\ast_k$, 
\begin{align*}
\sum_{j\in G_k^\ast} [\hat Z - Z^\ast]_{ij} + \sum_{j\not\in G_k^\ast} [\hat Z - Z^\ast]_{ij}=0.
\end{align*} 
Since $Z^\ast_{ij}=0$ for each pair $(i,j)$ with $i\in G_k^\ast$ and $j\notin G_k^\ast$, and $\hat Z$ has nonnegative entries, we can sum up the preceding display over all $i\in G_k^\ast$ to obtain 
\begin{align*}
\big\langle\, \mathbf{1}_{G_k^\ast}\mathbf{1}_{G_k^\ast}^T,\, [\hat Z - Z^\ast]_{G_k^\ast G_k^\ast}\,\big\rangle = - \sum_{m:\,m\neq k} \|\hat Z_{G_k^\ast G_m^\ast}\|_1,\quad \forall k\in[K].
\end{align*}
Putting pieces together, we obtain
\begin{equation}\label{Eqn:k_eq_m}
\begin{aligned}
&\sum_{k=1}^K \big\langle\, [A_n]_{G_k^\ast G_k^\ast},\, [\hat Z - Z^\ast]_{G_k^\ast G_k^\ast}\,\big\rangle  \\
\leq &\, \max_{1 \leq k \leq K} \big\|\,[A_n]_{G_k^\ast G_k^\ast}- N_k^{-1} \mathbf{1}_{G_k^\ast}\mathbf{1}_{G_k^\ast}^T\,\big\|_\infty \, \sum_{k=1}^K \big\|\, [\hat Z -Z^\ast]_{G_k^\ast G_k^\ast}\big\|_1- \sum_{1\leq k\neq m\leq K} \frac{1}{N_k}\|\hat Z_{G_k^\ast G_m^\ast}\|_1. 
\end{aligned}
\end{equation}
Now by combining inequalities~\eqref{Eqn:basic_ineq}, \eqref{eqn:k_neq_m} and~\eqref{Eqn:k_eq_m}, we can reach the following inequality
\begin{equation}\label{Eqn:key_ineq}
\begin{aligned}
&\sum_{1\leq k\neq m\leq K} \frac{1}{N_k}\|\hat Z_{G_k^\ast G_m^\ast}\|_1 \leq  \max_{1 \leq k\neq m \leq K} \|[A_n]_{G_k^\ast G_m^\ast}\|_\infty \, \sum_{1\leq k\neq m\leq K} \|\hat Z_{G_k^\ast G_m^\ast}\|_1\\
&\qquad \qquad\qquad +\max_{1 \leq k \leq K} \big\|\,[A_n]_{G_k^\ast G_k^\ast}- N_k^{-1} \mathbf{1}_{G_k^\ast}\mathbf{1}_{G_k^\ast}^T\,\big\|_\infty \, \sum_{k=1}^K \big\|\, [\hat Z -Z^\ast]_{G_k^\ast G_k^\ast}\big\|_1.
\end{aligned}
\end{equation}
Since $\hat{Z} \in \sC_{K} \subset \sC$, according to inequalities~\eqref{eqn:ineq_1_feasible_set}-\eqref{eqn:ineq_3_feasible_set} in Lemma~\ref{lem:some_ineq_feasible_set} in Appendix C, 
\begin{align}
\label{eqn:DKM_SDP_another_core_inequality}
\sum_{1\leq k\neq m\leq K} \frac{1}{n_k}\|\hat Z_{G_k^\ast G_m^\ast}\|_1 \geq \frac{1}{n}\,\big\|\hat Z - Z^\ast\big\|_1.
\end{align}
The last two displays~\eqref{Eqn:key_ineq} and~\eqref{eqn:DKM_SDP_another_core_inequality} imply the exact recovery $\hat Z=Z^\ast$ as long as
\begin{equation}
\label{eqn:DKM_SDP_exact_recover_master_condition}
\max_{1 \leq k\neq m \leq K} \|[A_n]_{G_k^\ast G_m^\ast}\|_\infty + \max_{1 \leq k \leq K} \big\|\,[A_n]_{G_k^\ast G_k^\ast}- N_k^{-1} \mathbf{1}_{G_k^\ast}\mathbf{1}_{G_k^\ast}^T\,\big\|_\infty < \frac{1}{n}\,\min_{1\leq k\leq K}\Big\{\frac{n_k}{N_k}\Big\}.
\end{equation}

\begin{lem}[Master condition for the diffusion $K$-means to achieve exact recovery]
\label{lem:DKM_SDP_master_bound}
If~\eqref{eqn:DKM_SDP_exact_recover_master_condition} holds, then we can achieve exact recovery $\hat{Z} = Z^{*}$.
\end{lem}

\begin{proof}[Proof of Lemma~\ref{lem:DKM_SDP_master_bound}]
Suppose $\hat Z \neq Z^{*}$. Then $\|\hat{Z}-Z^{*}\|_{1} > 0$. From~\eqref{Eqn:key_ineq} and~\eqref{eqn:DKM_SDP_exact_recover_master_condition}, we have
\begin{align*}
\sum_{1\leq k\neq m\leq K} \frac{1}{N_k}\|\hat Z_{G_k^\ast G_m^\ast}\|_1 < & \frac{1}{n}\,\min_{1\leq k\leq K}\Big\{\frac{n_k}{N_k}\Big\} \Big\{ \, \sum_{1\leq k\neq m\leq K} \|\hat Z_{G_k^\ast G_m^\ast}\|_1 + \sum_{k=1}^K \big\|\, [\hat Z -Z^\ast]_{G_k^\ast G_k^\ast}\big\|_1 \Big\} \\
= & \,\min_{1\leq k\leq K}\Big\{\frac{n_k}{N_k}\Big\} \frac{1}{n}\|\hat{Z}-Z^{*}\|_{1},
\end{align*}
where the last equality follows from $Z^{*}_{G_k^\ast G_m^\ast} = 0$ for $m \neq k$. From~\eqref{eqn:DKM_SDP_another_core_inequality},
\[
\min_{1\leq k\leq K}\Big\{\frac{n_k}{N_k}\Big\} {1 \over n} \|\hat{Z}-Z^{*}\|_{1} \leq \sum_{1\leq k\neq m\leq K} \frac{1}{N_k}\|\hat Z_{G_k^\ast G_m^\ast}\|_1.
\]
Combining the last two inequalities, we get a contradiction. So we must have $\hat{Z} = Z^{*}$.
\end{proof}

To further proceed, we will make use of following two lemmas to provide high probability bounds for the empirical diffusion affinity entries deviating from their expectations. Proofs of Lemma~\ref{lem:within_cluster_random_walk} and~\ref{lem:between_cluster_random_walk} are deferred to the following subsections.

\begin{lem}\label{lem:within_cluster_random_walk}
Let $\kappa = \max_{1\leq k\neq k'\leq K}\sup_{x\in \cM_k,\,x'\in \cM_{k'}} \kappa(x,\,x')$ and $\tau = \inf_{x,y\in S:\, \|x-y\| \leq h} \kappa(x,y)$. If $c\, (\log n_k/ n_k)^{1/q_k}\,(\log n_k)^{\mathbf 1(q_k=2)/4}\leq h$, then for any $\xi\geq 1$ and any $n_k\geq M_0$,
\begin{align*}
\big\|\,[A_n]_{G_k^\ast G_k^\ast}- N_k^{-1} \mathbf{1}_{G_k^\ast}\mathbf{1}_{G_k^\ast}^T\,\big\|_\infty \leq C'\,n\,t\,\kappa\, (n_k\,h^{q_k})^{-2}+ C'\, (n_k\,h^{q_k})^{-1}\, e^{-2t\,\gamma(P_{n,k})}
\end{align*}
holds with probability at least $1-c_{1} \, n_{k}^{-\xi}$, where the spectral gap $\gamma(P_{n,k})$ satisfies
\begin{align}\label{Eqn:Spectral_gap_lower}
\gamma(P_{n,k}) \geq \bigg(1- C\,\Big(c^{-1}+(\sqrt{\lambda_1(D_k)} + 1)\,h+ h^2\Big)\bigg)\,\frac{\nu_{q_k}^2}{4(q_k+2)(2\pi)^{q_k/2}}\,\tau\, \lambda_1(\cM_k)\,h^2.
\end{align}
Here, $\nu_{d}=\frac{\pi^{d/2}}{\Gamma(d/2+1)}$ is the volume of the $d$-dim unit ball, $\lambda_1(\cM_k)>0$ is the second smallest eigenvalue of the Laplace-Beltrami operator on $\cM_k$ (cf.~Section~\ref{subsec:Laplace-Beltrami_operator}), constants $(c_1,M_0)$ depend only on the submanifold $\cM_k$ through $\{q_k,{\rm Diam_k}, {\rm  Sec_k}$, ${\rm Inj_k}, {\rm Rch_k}, {\rm Rch_k}\}$ and the density $p_k$ through $\{\eta_k,{\rm Lip_k}\}$, and constants $(C,C')$ depend only on $\{q_k,{\rm Sec_k},{\rm Inj_k}, {\rm Rch_k}$, $\eta_k,{\rm Lip_k},\xi\}$.
\end{lem}

\begin{lem}\label{lem:between_cluster_random_walk}
Let $\kappa = \max_{1\leq k\neq k'\leq K}\sup_{x\in \cM_k,\,x'\in \cM_{k'}} \kappa(x,\,x')$. Suppose conditions~\eqref{eqn:bandwidth_condition_nonadaptive} and~\eqref{eqn:exact_recovery_condition_SDP_diffusion_Kmeans_LB_eigenval} in Theorem~\ref{thm:main} are satisfied. Then for any $\xi\geq 1$, it holds with probability at least $1 - c_1 K\, \underline{n}^{-\xi}$ that
\begin{align*}
\|[A_n]_{G_k^\ast G_m^\ast}\|_\infty \leq   C\, t \, (n_k h^{q_k}\,n_m h^{q_m})^{-1}\,n\,\kappa,
\quad\forall k\neq m \in[K],
\end{align*}
where recall $\underline n = \min_{k\in [K]} n_k$ and $C$ is a constant only depending on $\{q_k,{\rm Sec_k},{\rm Inj_k}, {\rm Rch_k}$, $\eta_k,{\rm Lip_k}\}_{k=1}^K$ and $\xi$. 
\end{lem}

Recall that $\kappa(x,y) = \exp\{-\|x-y\|^2/(2h^2)\}$ is the Gaussian kernel. Consequently, we may choose $\kappa = \exp\{-\delta^2/(2h^2)\}$ and $\tau = e^{-1/2}$ in these two lemmas. By Weyl's law for the growth of eigenvalues of the Laplace-Beltrami operator (cf. Remark 6 in \cite{trillos2018error}), we have $\lambda_{j}(\cM_{k}) \sim j^{2/q_{k}}$. Thus $\lambda_1(\cM_k)$ is bounded.  Since the bandwidth parameter satisfies $c_1 (\log n_k/ n_k)^{1/q_k}\,(\log n_k)^{\mathbf 1(q_k=2)/4}\leq h \leq c_2$ for some sufficient large constant $c_1$ and sufficiently small constant $c_2$ for all $k \in[K]$, inequality~\eqref{Eqn:Spectral_gap_lower} in Lemma~\ref{lem:within_cluster_random_walk} becomes
\begin{align*}
\gamma(P_{n,k}) \geq \frac{\nu_{q_k}^2}{8(q_k+2)(2\pi)^{q_k/2}}\, \lambda_1(\cM_k)\,h^2,\quad\forall k\in[K].
\end{align*}
Therefore, the following two bounds hold with probability at least $1 - c_3 K\, \underline{n}^{-\xi}$,
\begin{align*}
\max_{1 \leq k \leq K} \big\|\,[A_n]_{G_k^\ast G_k^\ast}- N_k^{-1} \mathbf{1}_{G_k^\ast}\mathbf{1}_{G_k^\ast}^T\,\big\|_\infty &\leq C\, nt\,\exp\Big\{-\frac{\delta^2}{2h^2}\Big\}\\
&+ C\,  \exp\bigg\{-\min_{1 \leq k \leq K}\Big\{\frac{\nu_{q_k}^2\lambda_1(\cM_k)}{4(q_k+2)(2\pi)^{q_k/2}}\Big\}\, h^2 t\bigg\},\\
\max_{1\leq k\neq m\leq K}\|[A_n]_{G_k^\ast G_m^\ast}\|_\infty & \leq  C\,nt\,\exp\Big\{-\frac{\delta^2}{2h^2}\Big\},
\end{align*}
for some constants $c_1,c_2,c_3$ only depending on $\{q_k,{\rm Sec_k},{\rm Inj_k}, {\rm Rch_k},\eta_k,{\rm Lip_k}\}_{k = 1}^{K}$ and $C$ only depending on $\{q_k,{\rm Sec_k},{\rm Inj_k}, {\rm Rch_k},\eta_k,{\rm Lip_k}\}_{k = 1}^{K}$ and $\xi$. The last two inequalities combining with Lemma~\ref{lem:DKM_SDP_master_bound} imply the exact recovery $\hat Z=Z^\ast$ as long as 
\begin{align*}
C\,nt\, \exp\Big\{-\frac{\delta^2}{2h^2}\Big\} +C\,\exp\bigg\{-\min_{1 \leq k \leq K}\Big\{\frac{\nu_{q_k}^2\lambda_1(\cM_k)}{4(q_k+2)(2\pi)^{q_k/2}}\Big\}\, h^2 t\bigg\} < \frac{1}{n}\,\min_{1\leq k\leq K}\Big\{\frac{n_k}{N_k}\Big\}. 
\end{align*}
Finally, the claimed result follows from the preceding display and the following lemma of a high probability bound on $N_k$, whose proof is postponed to Section~\ref{Sec:Proof_total_degree}.

\begin{lem}[Concentration for total within-weight in $G_{k}^{*}$]
\label{Lemma:total_degree}
Suppose the density $p_{k}$ satisfies~\eqref{Eqn:density_condition} and the bandwidth $h \leq c$ for some constant $c>0$. Then for any $\xi\geq 1$, it holds with probability at least $1-c_2\,n_k^{-\xi}$ that
\begin{align*}
\bigg| \frac{N_k}{n_k} - (\sqrt{2\pi}\, h)^{q_k} \,n_k \beta_k\bigg| \leq C \left(  n_{k} h^{q_{k}+1} + \sqrt{n_kh^{q_k}\log n_k} \right), 
\end{align*}
where $\beta_k = \mathbb E_{X\sim p_k}[p_k(X)]$ and $C$ is a constant only depending on $\{q_k,{\rm Sec_k},{\rm Inj_k}, {\rm Rch_k},\eta_k,{\rm Lip_k}\}_{k = 1}^{K}$ and $\xi$.
\end{lem}

\subsection{Proof of Corollary~\ref{cor:main}}
Corollary~\ref{cor:main} follows from Theorem~\ref{thm:main} and the Cheeger inequality \cite{Cheeger1970}: 
\[
\lambda_{1}(\cM_{k}) \geq \frac{\fh(\cM_{k})^{2}}{4}. 
\]

\subsection{Proof of Theorem~\ref{thm:main_adaptive_lambda}}\label{Sec:proof_thm_main_adaptive}
Similar to the proof of Theorem~\ref{thm:main} in Section~\ref{Sec:proof_thm_main}, we can use the optimality of $\tilde Z$ and the feasibility of $Z^\ast$ for the SDP program~\eqref{eqn:clustering_Kmeans_sdp_unknown_K} to obtain the following basic inequality,
\begin{align}
0&\leq \langle A_n,\, \tilde Z - Z^\ast\rangle + n \rho\big(\tr(Z^\ast) - \tr(\tilde Z)\big) \notag \\
&=\sum_{1\leq k\neq m\leq K} \big\langle\, [A_n]_{G_k^\ast G_m^\ast},\, [\tilde Z - Z^\ast]_{G_k^\ast G_m^\ast}\,\big\rangle + \sum_{k=1}^K \big\langle\, [A_n]_{G_k^\ast G_k^\ast},\, [\tilde Z - Z^\ast]_{G_k^\ast G_k^\ast}\,\big\rangle \notag\\
&\qquad\qquad\qquad+ n\,\rho\,\big(\tr(Z^\ast) - \tr(\tilde Z)\big).\label{Eqn:basic_inequality_adaptive}
\end{align}
Since the only place where we used the constraint $\tr(Z) =K$ in Section~\ref{Sec:proof_thm_main} is Lemma~\ref{lem:some_ineq_feasible_set} in Appendix C, 
the analysis of the first two sums in~\eqref{Eqn:basic_inequality_adaptive} still apply, leading to
\begin{equation}
\label{Eqn:final_bound_adaptive}
\begin{aligned}
 \sum_{1\leq k\neq m\leq K} \frac{1}{N_K}\|\tilde Z_{G_k^\ast G_m^\ast}\|_1 
\leq&\, C\,\bigg( nt\, \exp\Big\{-\frac{\delta^2}{2h^2}\Big\} +\exp\Big\{-c\, \min_{1 \leq k \leq K}\{\lambda_1(\cM_k)\}\, h^2 t\Big\}\bigg)\, \big\|\tilde Z - Z^\ast\big\|_1\\
&\,\qquad + n\,\rho\,\big(\tr(Z^\ast) - \tr(\tilde Z)\big),
\end{aligned}
\end{equation}
which holds with probability at least $1-c_3K\,\underline{n}^{-c_3}$.

Now we apply Lemma~\ref{lem:some_ineq_feasible_set_adaptive} in Appendix C to obtain
\begin{align*}
n\, \rho\, \big(\tr(Z^\ast) - \tr(\tilde Z)\big) \leq 4n\,\rho\sum_{1\leq k\neq m\leq K} \frac{1}{n_k}\|\tilde Z_{G_k^\ast G_m^\ast}\|_1-\rho\,\big\|\tilde Z - Z^\ast\big\|_1.
\end{align*}
Combining this inequality with~\eqref{Eqn:final_bound_adaptive}, we obtain
\begin{align*}
 \bigg(\rho - C\,\Big( nt\, \exp\Big\{-\frac{\delta^2}{2h^2}\Big\} &+\exp\Big\{-c\, \min_{1 \leq k \leq K}\{\lambda_1(\cM_k)\}\, h^2 t\Big\}\Big)\bigg)\, \big\|\tilde Z - Z^\ast\big\|_1\\
 &\qquad\qquad+\,\bigg(1-4n\rho \max_{1\leq k\leq K}\Big\{\frac {N_k}{n_k}\Big\}\bigg)\,\sum_{1\leq k\neq m\leq K} \frac{1}{N_k}\|\tilde Z_{G_k^\ast G_m^\ast}\|_1
\leq 0.
\end{align*}
This implies the exact recovering $\tilde Z=Z^\ast$ provided that 
\begin{align*}
C_1\, \Big( nt\, \exp\Big\{-\frac{\delta^2}{2h^2}\Big\} &+\exp\Big\{-c\, \min_{1 \leq k \leq K}\{\lambda_1(\cM_k)\}\, h^2 t\Big\}\Big) < \rho \leq \frac{C_2}{n}\,\min_{1\leq k\leq K}\Big\{\frac{n_k}{N_k}\Big\}.
\end{align*}
Finally, the claimed result follows by combining the above with Lemma~\ref{Lemma:total_degree}.

\subsection{Proof of Lemma~\ref{lem:within_cluster_random_walk}}
We consider a fixed $k\in[K]$ throughout this proof. 
Recall that $P_n=D_n^{-1} \mathcal K_n$ defines the random walk $\mc W_n$ over $S_n=\{X_1,X_2,\ldots,X_n\}$, and $A_n = P_n^{2t} D_n^{-1}$. Now consider a new random walk $\mc W_{n,k}$ over the $k$th cluster $G_k^\ast$ defined in the following way. For simplicity of notation, we may rearrange the nodes order so that $G_k^\ast = \{1,2,\ldots,n_k\}$. Then the transition probability matrix $P_{n,k}\in\bR^{n_k\times n_k}$ of $\mc W_{n,k}$ is defined as
\begin{align*}
[P_{n,k}]_{ij} =\frac{\kappa(X_i,\,X_j)}{d_{n,k}(X_i)},\quad \forall i,j\in G_k^\ast,\quad \mbox{where } d_{n,k}(x) = \sum_{j\in G_k^\ast} \kappa(x,\,X_j)
\end{align*}
is the induced degree function within cluster $G_k^\ast$. Similar to the diagonal degree matrix $D_n$, we denote by $D_{n,k}\in\bR^{n_k\times n_k}$ the diagonal matrix whose $i$th diagonal entry is $d_{n,k}(X_i)$ for $i\in[n_k]$. Note that $N_k=\sum_{i\in G_k^\ast} d_{n,k}(X_i)$ the total degrees within $G_k^\ast$ so that $N_k \geq n_k \min_{i\in G^\ast_k} d_{n,k}(X_i)$. It is easy to see that the limiting distribution
of $\mc W_{n,k}$ is $\pi_{n,k} = N_k^{-1} \mbox{diag}(D_{n,k})\in\bR^{n_k}$.
 Under the separation condition on $\delta$ in the lemma, the probability of moving out from $G_k^\ast$ is exponentially small, suggesting that we may approximate the sub-matrix $[P_n^{2t}]_{G_k^\ast G_k^\ast} $ of $P_n^{2t}$ with $P_{n,k}^{2t}$. We will formalize this statement in the rest of the proof.
 
First, we apply the triangle inequality to obtain
\begin{align}
&\big\|\,[A_n]_{G_k^\ast G_k^\ast} - N_k^{-1} \mathbf{1}_{G_k^\ast}\mathbf{1}_{G_k^\ast}^T\,\big\|_\infty \notag\\
& \leq \big\| \,[P_n^{2t}]_{G_k^\ast G_k^\ast} [D_{n}]_{G_k^\ast G_k^\ast}^{-1} - P_{n,k}^{2t} D_{n,k}^{-1}\big\|_\infty + \big\| \,P_{n,k}^{2t} D_{n,k}^{-1} - N_k^{-1} \mathbf{1}_{G_k^\ast} \mathbf{1}_{G_k^\ast}^T \big\|_\infty =: T_1 + T_2,\label{Eqn:within_random_walk}
\end{align}
where $T_1$ captures the difference between $[P_{n}^{2t}]_{G_k^\ast G_k^\ast}$ and $P_{n,k}^{2t}$, and $T_2$ characterizes the convergence of the Markov chain $\mc W_{n,k}$ to its limiting distribution $\pi_{n,k}$ after $2t$ steps.

Recall that $\kappa = \max_{1\leq k\neq k'\leq K}\sup_{x\in \cM_k,\,x'\in \cM_{k'}} \kappa(x,\,x')$ is the minimal between-cluster affinity. The first term $T_1$ and the second term $T_2$ can be bounded via two lemmas below.

\begin{lem}[Term $T_1$]\label{Lem:T_1}
If $n \kappa \leq \min_{i \in G_{k}^{*}} d_{n,k}(X_{i})$, then
\begin{align*}
T_1 \leq (2t+1)\,n\,\kappa\,\max_{i\in[n]}d_{n,k}^{-2}(X_i). 
\end{align*}
\end{lem}

\begin{lem}[Term $T_2$]\label{Lem:T_2}
Let $\tau = \inf_{x,y\in S:\, \|x-y\| \leq h} \kappa(x,y)$. For each $i\in G_k^\ast$, let $d_{k}^\dagger(X_{i})$ denote total number of points in $\{X_j\}_{j\in G_k^\ast}$ inside the $d$-ball centered at $X_i$ with radius $h$. For any $\xi\geq 1$, if $c\, (\log n_k/ n_k)^{1/q_k}\,(\log n_k)^{\mathbf 1(q_k=2)/4}\leq h$, then it holds with probability at least $1-c_1\,n_k^{-\xi}$ that
\begin{align*}
T_2 \leq e^{-2t\,\gamma(P_{n,k})} \,\max_{i\in G_k^\ast} d_{n,k}^{-1}(X_i),\quad\forall t=1,2,\ldots,
\end{align*}
with
\begin{align}\label{Eqn:spectral_gap_bound}
\gamma(P_{n,k}) \geq \bigg(1- C\,\Big(c^{-1}+(\sqrt{\lambda_1(\cM_k)} + 1)\,h+ h^2\Big)\bigg) \,\bigg[\min_{i\in G_k^\ast}  \frac{d_{k}^\dagger(X_{i})}{d_{n,k}(X_i)}\bigg]\,\frac{\tau\nu_{q_k}\, \lambda_1(\cM_k)\,h^2}{2(q_k+2)},
\end{align}
where recall that $\lambda_1(\cM_k)>0$ denotes the second smallest eigenvalue of the Laplace-Beltrami operator on $\cM_k$, and $\nu_{d}=\frac{\pi^{d/2}}{\Gamma(d/2+1)}$ denotes the volume of the $d$-dim unit ball. Here, constant $c_1$ depends only on the submanifold $\cM_k$ through $\{q_k,{\rm Diam_k}, {\rm  Sec_k}$, ${\rm Inj_k}, {\rm Rch_k}, {\rm Rch_k}\}$ and the density $p_k$ through $\{\eta_k,{\rm Lip_k}\}$, and constant $C$ depends on $\{q_k,{\rm Sec_k},{\rm Inj_k}, {\rm Rch_k},\eta_k,{\rm Lip_k},\xi\}$.
\end{lem}
\noindent Proofs of these two lemmas are provided in Sections~\ref{Sec:Proof_T_1} and \ref{Sec:Proof_T_2}.

Combining upper bounds for the two terms in inequality~\eqref{Eqn:within_random_walk} together, we can reach
\begin{align*}
\big\|\,[A_n]_{G_k^\ast G_k^\ast}- N_k^{-1} \mathbf{1}_{G_k^\ast}\mathbf{1}_{G_k^\ast}^T\,\big\|_\infty \leq (2t+1)\,n\,\kappa\, \max_{i\in[n]}d_{n,k}^{-2}(X_i) + e^{-2t\,\gamma(P_{n,k})}\,\max_{i\in[n]}d_{n,k}^{-1}(X_i),
\end{align*}
where the spectral gap $\gamma(P_{n,k})$ satisfies ~\eqref{Eqn:spectral_gap_bound}.
It remains to prove some high probability bounds for $d_{n,k}(X_i)$ and $d_{k}^\dagger(X_i)$, which are summarized in the following lemma.

\begin{lem}[Concentrations for node degrees]
\label{lem:node_degree}
Suppose the density $p_{k}$ satisfies~\eqref{Eqn:density_condition} and the bandwidth $h \leq c_1$ for some constant $c_1>0$. Then for any $\xi\geq 1$, it holds with probability at least $1-c_2\,n_k^{-\xi}$ that
\begin{align*}
\max_{i\in G_k^\ast} \bigg|\,\frac{d_{n,k}(X_i)}{n_k\,(\sqrt{2\pi} \,h)^{q_k}}  - p_k(X_i) \bigg| &\leq C\bigg( h + \sqrt{\frac{\log n_k}{n_k h^{q_{k}}}}\,\bigg), \quad \mbox{and} \\
\max_{i\in G_k^\ast} \bigg|\,\frac{d_{k}^\dagger(X_i)}{n_k\,\nu_{q_k}\,h^{q_k}}  - p_k(X_i) \bigg| &\leq C\bigg( h^2 + \sqrt{\frac{\log n_k}{n_k h^{q_{k}}}}\,\bigg),
\end{align*}
where $\nu_{d}=\frac{\pi^{d/2}}{\Gamma(d/2+1)}$ denotes the volume of an unit ball in $\bR^{d}$. Here constants $c_1,c_2$ only depend on the manifold $\cM_k$ through $\{q_k,{\rm Sec_k},{\rm Inj_k}\}$ and the density function $p_k$ through $\{\eta_k,{\rm Lip_k}\}$, and constant $C$ depends on $\{q_k,{\rm Sec_k},{\rm Inj_k},\eta_k,{\rm Lip_k},\xi\}$.
\end{lem}
\noindent A proof of this lemma is deferred to Section~\ref{Sec:Proof_lem:node_degree}.

Finally, by combining this lemma with the last display, and applying the uniform boundedness condition~\eqref{Eqn:density_condition} on $p_k$, we obtain that for any $\xi\geq 1$, if $c\, (\log n_k/ n_k)^{1/q_k}\,(\log n_k)^{\mathbf 1(q_k=2)/4}\leq h$, then for any $n_k\geq M_0$,
\begin{align*}
\big\|\,[A_n]_{G_k^\ast G_k^\ast}- N_k^{-1} \mathbf{1}_{G_k^\ast}\mathbf{1}_{G_k^\ast}^T\,\big\|_\infty \leq C'\,n\,t\,\kappa\, (n_k\,h^{q_k})^{-2}+ C'\, (n_k\,h^{q_k})^{-1}\, e^{-2t\,\gamma(P_{n,k})}
\end{align*}
holds with probability at least $1-c_{1} \, n_{k}^{-\xi}$, where the spectral gap $\gamma(P_{n,k})$ satisfies
\begin{align*}
\gamma(P_{n,k}) \geq \bigg(1- C\,\Big(c^{-1}+(\sqrt{\lambda_1(D_k)} + 1)\,h+ h^2\Big)\bigg)\,\frac{\nu_{q_k}^2}{4(q_k+2)(2\pi)^{q_k/2}}\,\tau\, \lambda_1(\cM_k)\,h^2.
\end{align*}
Here, constants $(c_1,M_0)$ depend only on the submanifold $\cM_k$ through $\{q_k,{\rm Diam_k}, {\rm  Sec_k}$, ${\rm Inj_k}, {\rm Rch_k}, {\rm Rch_k}\}$ and the density $p_k$ through $\{\eta_k,{\rm Lip_k}\}$, and constant $C,C'$ depends on $\{q_k,{\rm Sec_k},{\rm Inj_k}, {\rm Rch_k}$, $\eta_k,{\rm Lip_k},\xi\}$.

\subsection{Proof of Lemma~\ref{lem:between_cluster_random_walk}}
For each indices $i$ and $j$ that belong to two difference clusters $G_k^\ast$ and $G_m^\ast$ with $k \neq m$, we have
\begin{align*}
[A_n]_{ij}= [P_{n}^{2t}]_{ij} \cdot d_n(X_j)^{-1}.
\end{align*}
Let $\mc W_n=\{Y_t:\,t\geq 0\}$, with $Y_t$ denote the state of the Markov chain $\mc W_n$ at time $t$. 
Define $T_k(i)=\min\big\{t \in \bN_{+}:\, Y_t \not\in G_k^\ast, \, Y_{0} = i\big\}$ denote the first exit time from $G_k^\ast$ of $\mc W_n$ starting from $Y_0=i$. Then it is easy to see that
\begin{align*}
[P_{n}^{2t}]_{ij} = \Prob(Y_0=i,\,Y_{2t} = j) \leq  \Prob(T_k(i) \leq 2t) = 1 - \Prob(T_k(i) > 2t).
\end{align*}
Since for each $i\in G_k^\ast$, the one-step transition probability of moving out from $G_k^\ast$ is bounded by $n\,\kappa\,\max_{i\in G_k^\ast} d_{n}^{-1}(X_i)$, we have 
\begin{align*}
\Prob(T_k(i) > 2t) \geq (1-n\,\kappa\,\max_{i\in G_k^\ast} d_{n}^{-1}(X_i))^{2t} \geq 1-2n\,t \,\kappa\,\max_{i\in G_k^\ast} d_{n,k}^{-1}(X_i), 
\end{align*}
provided that $n\,\kappa \leq \min_{i\in G_k^\ast} d_{n}(X_i)$. Therefore, for each $i\in G_k^\ast$ and $j\in G_m^\ast$ with $k\neq m$, we have
\begin{align*}
\big|[A_n]_{ij}\big|  \leq 2n\,t \,\kappa\,\max_{i'\in G_k^\ast} d_{n,k}^{-1}(X_{i'}) \,\max_{j'\in G_m^\ast} d_{n,m}^{-1}(X_{j'}).
\end{align*}
By Lemma~\ref{lem:node_degree} and~\eqref{Eqn:density_condition}, we have that with probability at least $1-c_{1}n^{-\xi}$, $C_{1} n_{k} h^{q_{k}} \leq d_{n,k}(X_{i}) \leq C_{2} n_{k} h^{q_{k}}$ for some constants $C_1$ and $C_{2}$ only depending on $\cM_k$ and $p_k$. Note that condition~\eqref{eqn:exact_recovery_condition_SDP_diffusion_Kmeans_LB_eigenval} in Theorem~\ref{thm:main} yields that $n \kappa = n e^{-\delta^{2} / (2h^{2})} \leq C n_{k} h^{q_{k}}$. Then the claimed bound is implied by the above combined with the union bound.

\subsection{Proof of Lemma~\ref{Lem:T_1}}\label{Sec:Proof_T_1}
By adding and subtracting the same term we obtain
\begin{align}
T_1 \leq &\,  \big\| \big([P_n^{2t}]_{G_k^\ast G_k^\ast} - P_{n,k}^{2t} \big)\,D_{n,k}^{-1}\big\|_\infty +\big\| \,[P_n^{2t}]_{G_k^\ast G_k^\ast}\,\big( [D_{n}]_{G_k^\ast G_k^\ast}^{-1} - D_{n,k}^{-1}\big)\big\|_\infty \notag \\
\leq &\, \big\| \,[P_n^{2t}]_{G_k^\ast G_k^\ast} - P_{n,k}^{2t} \big\|_\infty \,\|D_{n,k}^{-1}\|_\infty+ \big\|  [D_{n}]_{G_k^\ast G_k^\ast}^{-1} - D_{n,k}^{-1}\big\|_\infty \notag \\
= &\, \big\| \,[P_n^{2t}]_{G_k^\ast G_k^\ast} - P_{n,k}^{2t} \big\|_\infty \,\max_{i\in G_k^\ast} d_{n,k}^{-1}(X_i)+\max_{i\in G_k^\ast} |d_n^{-1}(X_i)-d_{n,k}^{-1}(X_i)| \notag\\
\leq &\, \big\| \,[P_n^{2t}]_{G_k^\ast G_k^\ast} - P_{n,k}^{2t} \big\|_\infty \,\max_{i\in G_k^\ast}d_{n,k}^{-1}(X_i) +\max_{i\in G_k^\ast} |d_n(X_i)-d_{n,k}(X_i)|\,\max_{i\in G_k^\ast}d_{n,k}^{-2}(X_i),\label{Eqn:T_1_bound}
\end{align}
where the second inequality is due to the fact that each row sum of $[P_n^{2t}]_{G_k^\ast G_k^\ast}$ is at most one.

Now we consider the first term in~\eqref{Eqn:T_1_bound}. Recall that $\mc W_n=\{Y_t:\,t\geq 0\}$, where $Y_t$ is the state of the Markov chain $\mc W_n$ at time $t$.   Note that for each $i,j\in G_k^\ast$, we have
\begin{align*}
[P_{n,k}]_{ij} =\frac{\kappa(X_i, X_j)}{\sum_{j\in G_k^\ast} \kappa(X_i,X_j)}= \frac{\Prob(Y_{2t}=j\,|\, Y_{2t-1}=i)}{\Prob(Y_{2t}\in G_k^\ast\,|\, Y_{2t-1}=i)} = \Prob(Y_{2t}=j\, |\,Y_{2t-1}=i,\,Y_{2t}\in G_k^\ast).
\end{align*}
As a consequence, we have by the law of total probability and the Markov property of $\cW_n$ that for each $i,j\in G_k^\ast$ and $s\in\bN_{+}$,
\begin{align*}
[P_n^{s}]_{ij} &= \Prob(Y_{s} = j \,|\, Y_0 =i) \\
&=\sum_{\ell\in G_k^\ast} \Prob(Y_{s} = j \,|\, Y_{s-1} =\ell,\, Y_0=i, \,Y_{s}\in G_k^\ast) \\
& \qquad \ \ \ \  \cdot \Prob(Y_{s}\in G_k^\ast \,|\, Y_{s-1} =\ell,\, Y_0=i)\cdot \Prob(Y_{s-1} = \ell\,|\, Y_0=i) \\
& \ \ \ \ + \sum_{\ell\not\in G_k^\ast}  \Prob(Y_{s} = j\,|\, Y_{s-1} =\ell,\, Y_0=i) \cdot \Prob(Y_{s-1} =\ell\,|\, Y_0=i)\\
&= \sum_{\ell\in G_k^\ast} [P_n^{s-1}]_{i\ell} \cdot [P_{n,k}]_{\ell j}\cdot
\Prob(Y_{s}\in G_k^\ast \,|\, Y_{s-1} =\ell)\\
& \ \ \ \ + \sum_{\ell\not\in G_k^\ast}  \Prob(Y_{s} = j\,|\, Y_{s-1} =\ell) \cdot \Prob(Y_{s-1} =\ell\,|\, Y_0=i).
\end{align*}
For each pair $(j,\ell)$ belonging to different clusters, noting that $d_{n,k}(X_i) \leq d_n(X_i)$, we have
\begin{align*}
\Prob(Y_{s} = j\,|\, Y_{s-1} =\ell) = [P_n]_{\ell j} =\frac{\kappa(X_\ell, X_j)}{d_n(X_\ell)} \leq \kappa\, \max_{i\in[n]} d_n^{-1}(X_i) \leq \max_{i\in[n]} d_{n,k}^{-1}(X_i),
\end{align*}
which implies that for each $\ell \in G_k^\ast$, 
\begin{align*}
0\leq 1- \Prob(Y_{s}\in G_k^\ast \,|\, Y_{s-1} =\ell) = \Prob(Y_{s} \not\in G_{k}^{\ast} \,|\, Y_{s-1}=\ell) \leq n \,\kappa\, \max_{i\in[n]} d_{n,k}^{-1}(X_i) \leq 1.
\end{align*}
Combining the last three displays, we obtain for each $i,j\in G_k^\ast$ and $s\in\bN^{+}$,
\begin{align*}
\big(1- n \,\kappa\, \max_{i\in[n]} d_{n,k}^{-1}(X_i)\big)\,\sum_{\ell\in G_k^\ast} [P_n^{s-1}]_{i\ell} & \cdot [P_{n,k}]_{\ell j} \leq [P_n^{s}]_{ij} \\ 
\leq & \sum_{\ell\in G_k^\ast} [P_n^{s-1}]_{i\ell} \cdot [P_{n,k}]_{\ell j} +  n \,\kappa\, \max_{i\in[n]} d_{n,k}^{-1}(X_i),
\end{align*}
which can be further simplified into
\begin{align*}
\big(1- n \,\kappa\, \max_{i\in[n]} d_{n,k}^{-1}(X_i)\big)\,\big[ [P_n^{s-1}]_{G_k^\ast G_k^\ast} P_{n,k}\big]_{ij} \leq [P_n^{s}]_{ij}\leq \big[ [P_n^{s-1}]_{G_k^\ast G_k^\ast} P_{n,k}\big]_{ij} + n \,\kappa\, \max_{i\in[n]} d_{n,k}^{-1}(X_i).
\end{align*}
Now we can recursively apply this two-sided inequality to get
\begin{align*}
\big(1- n \,\kappa\, \max_{i\in[n]} d_{n,k}^{-1}(X_i)\big)^s\,\big[ P_{n,k}^s]_{ij} \leq [P_n^{s}]_{ij}\leq \big[  P_{n,k}^s ]_{ij} + n\, s \,\kappa\, \max_{i\in[n]} d_{n,k}^{-1}(X_i),\quad\forall i,j\in G_k^\ast.
\end{align*}
By taking $s=2t$ and applying the inequality $(1-x)^s\geq 1-xs$ for $s\in\bN_{+}$ and $x\in[0,1]$, the above can be further reduced into
\begin{align}\label{Eqn:T_1_first_term}
\big\| \,[P_n^{2t}]_{G_k^\ast G_k^\ast} - P_{n,k}^{2t} \big\|_\infty \leq 2n\, t \,\kappa\, \max_{i\in[n]} d_{n,k}^{-1}(X_i).
\end{align}
Then we get 
\[
\big\| \,[P_n^{2t}]_{G_k^\ast G_k^\ast} - P_{n,k}^{2t} \big\|_\infty  \,\max_{i\in G_k^\ast}d_{n,k}^{-1}(X_i) \leq  2n\, t \,\kappa\, \max_{i\in[n]} d_{n,k}^{-2}(X_i),
\]
which is an upper bound to the first term in inequality~\eqref{Eqn:T_1_bound} for $T_1$.

The second term in inequality~\eqref{Eqn:T_1_bound} can be bounded as
\begin{align}\label{Eqn:T_1_second_term}
& \max_{i\in G_k^\ast} |d_n(X_i)-d_{n,k}(X_i)| =\max_{i\in G_k^\ast} \sum_{j\not\in G_k^\ast} \kappa(X_i,X_j) \leq n \kappa.
\end{align}

Finally, by combining \eqref{Eqn:T_1_bound}, \eqref{Eqn:T_1_first_term} and \eqref{Eqn:T_1_second_term}, we obtain
\begin{align*}
T_1 \leq  (2t+1) \, n\,\kappa\,\max_{i\in[n]}d_n^{-2}(X_i).
\end{align*}

\subsection{Proof of Lemma~\ref{Lem:T_2}}\label{Sec:Proof_T_2}
Recall that $\gamma(P_{n,k}) = 1-\lambda_1(P_{n,k})$ denote spectral gap of the transition matrix $P_{n,k}$, where $\lambda_1(P_{n,k})$ denotes the second largest eigenvalue of $P_{n,k}\in \bR^{n_k\times n_k}$ (due to similar arguments as in Appendix B, $P_{n,k}$ has $n_k$ real eigenvalues with $1$ as the largest one). In addition, since the kernel function $k$ is positive semidefinite, all eigenvalues of $P_{n,k}$ are nonnegative, meaning that $\gamma(P_{n,k})$ is equal to the absolute spectral gap $1-\max\{\lambda_1(P_{n,k}),\, \lambda_{n_k-1}(P_{n,k})\}$ where $\lambda_{n_k-1}(P_{n,k})$ is the $n_k$th (smallest) eigenvalue of $P_{n,k}$. 

Therefore, according to the relationship between the mixing time of a Markov chain and its absolute spectral gap (see, for example, equation~(12.11) in \cite{levin2017markov}), we have for each $i,j\in G_k^\ast$,
\begin{equation}
\label{eqn:mixing_T2}
\bigg|\frac{[P_{n,k}^{2t}]_{ij}}{[\pi_{n,k}]_j} - 1\bigg| \leq \frac{e^{-2t\,\gamma(P_{n,k})}}{\min_{\ell \in G_k^\ast} [\pi_{n,k}]_\ell},\quad\forall t=1,2,\ldots,
\end{equation}
where $\pi_{n,k} =\big(d_{n,k}(X_1)/ N_k,\ldots, d_{n,k}(X_{n_k})/ N_k\big)^T\in\bR^{n_k}$ is the limiting distribution of induced Markov chain $\mc W_{n,k}$ over $G_k^\ast$ with transition probability matrix $P_{n,k}$. This leads to a bound on $T_2$ as
\begin{align*}
T_2 & =  \big\| \,P_{n,k}^{2t} D_{n,k}^{-1} - N_k^{-1} \mathbf{1}_{G_k^\ast} \mathbf{1}_{G_k^\ast}^T \big\|_\infty= \max_{i,j\in G_k^\ast}\frac{1}{N_k}\, \bigg|\frac{[P_{n,k}^{2t}]_{ij}}{[\pi_{n,k}]_j} - 1\bigg| \\
&\leq \frac{1}{N_k} \, \frac{e^{-2t\,\gamma(P_{n,k})}}{\min_{\ell \in G_k^\ast} [\pi_{n,k}]_\ell}=e^{-2t\,\gamma(P_{n,k})} \,\max_{i\in G_k^\ast} d_{n,k}^{-1}(X_i),\quad\forall t=1,2,\ldots.
\end{align*}
Therefore, it remains to provide a lower bound on the spectral gap $\gamma(P_{n,k}) = 1-\lambda_1(P_{n,k})$. 
We do so by applying a comparison theorem of Markov chains (Lemma 13.22 in \cite{levin2017markov}), where we compare the spectral gap of $P_{n,k}$ with that of a 
standard random walk on a random geometric graph over $\{X_i\}_{i\in G_k^\ast}$, where each pair of nodes are connected if and only if they are at most $h$ far away from each other. The spectrum of the normalized graph Laplacian of the latter is known to behave like the eigensystem of the Laplace-Beltrami operator over the submanifold corresponding to the $k$-th connected subset $\cM_k$.
In particular, we will use existing results \cite{burago2014graph,trillos2018error} on error estimates by using the spectrum of a random geometric graph to approximate the eigensystem of the Laplace-Beltrami operator in the numerical analysis literature.

Let us first formally define a random geometric graph over $\{X_i\}_{i\in G_k^\ast}$ as i.i.d.~samples from the compact connected $q_k$-dimensional Riemannian submanifold $\cM_k$ in $\bR^{p}$ with bounded diameter, absolute sectional curvature value, and injectivity radius. Recall that $\mu_k$ is a probability measure on $\cM_k$ that has a Lipschitz density $p_k$ with respect to the Riemannian volume measure on $\cM_k$ satisfying~\eqref{Eqn:density_condition}. $\{X_i\}_{i\in G_k^\ast}$ can be viewed as a sequence of i.i.d.~samples from $\mu_k$, and without loss of generality, we may assume $G_k^\ast=\{1,2,\ldots,n_{k}\}$.
Consider the random geometric graph $\mc G_k^{\dagger}=(V_k, E_k)$, with $V_k=\{X_i\}_{i\in G_k^\ast}$ being its set of vertices and $E_k$ set of edges, constructed by putting an edge between $X_i$ and $X_j$ (write $i\sim j$ and call $X_i$ to be a neighbor of $X_j$) if and only if $\|X_i-X_j\| \leq h$. We define the natural random walk $\mc W_{k}^\dagger$ as a reversible Markov chain on $V_k$ that moves to neighbors of the current state with equal probabilities. In other words, the transition probability matrix $\mc P_k^{\dagger}\in \bb R^{n_k\times n_k}$ satisfies
\begin{align*}
[\mc P_k^{\dagger}]_{ij} = \begin{cases}
\displaystyle (d^{\dagger}_{k,i})^{-1}, & \quad\mbox{if } j\sim i\\
\displaystyle 0, &\quad\mbox{otherwise},
\end{cases}
\end{align*}
where $d^\dagger_{k,i} := d^\dagger_{k}(X_{i}) = \sum_{j=1}^{n_k} 1(j\sim i)$ denotes the degree of vertex $X_i$. It is easy to see that $\pi^\dagger_k = (d^\dagger_{k,1}/d^\dagger_k,d^\dagger_{k,2}/d^\dagger_k,\ldots,d^\dagger_{k,n_k}/d^\dagger_k)^T$, where $d^\dagger_k=\sum_{i=1}^{n_k}d^{\dagger}_{k,i}$ denotes the total degree, is the stationary distribution of this random walk. Let $1= \lambda_0( \mc P_k^{\dagger})\geq \lambda_1(\mc P_k^{\dagger})\geq \ldots\geq \lambda_{n-1}(\mc  P_k^{\dagger})\geq -1$ denote the eigenvalues of matrix  $\mc P_k^{\dagger}$, and $\gamma(\mc P_k^{\dagger})=1-\lambda_1(\mc P_k^{\dagger})$ denote its spectral gap.

Let $L_{\mc G_k^\dagger}=D_k^\dagger - A_k^\dagger\in\bb R^{n_k\times n_k}$ denote the graph Laplacian matrix associated with graph $\mc G_k^\dagger=(V_k,E_k)$, where $D_k^\dagger\in\bb R^{n_k\times n_k}$ is a diagonal matrix with $[D^\dagger_k]_{ii} = d_{k,i}^\dagger$, and $A_k^\dagger\in \bb R^{n_k\times n_k}$ is the adjacency matrix with $[A_k^\dagger]_{ij}=1(i\sim j)$ for all distinct pair $(i,j)\in [n_k]^2$. Define the normalized Laplacian of $\mc G_k^\dagger$ as $L^N_{\mc G_k^\dagger}=(D_k^\dagger)^{-1/2}L_{\mc G_k^\dagger} (D_k^\dagger)^{-1/2}= I -  (D_k^\dagger)^{-1/2} A_k^\dagger (D_k^\dagger)^{-1/2}$, and denote its ordered eigenvalues by $0 \leq \lambda_0(L^N_{\mc G_k^\dagger}) \leq \lambda_1(L^N_{\mc G_k^\dagger})\leq \cdots\leq \lambda_{n_k-1}(L^N_{\mc G_k^\dagger})$.
Since $(D_k^\dagger)^{-1/2} A_k^\dagger (D_k^\dagger)^{-1/2}= (D_k^\dagger)^{1/2} \mc P_k^\dagger (D_k^\dagger)^{-1/2}$ is a similarity transformation of $\cP_k^\dagger$, they share the same eigenvalues. Therefore, we have the relation $\lambda_j(L^N_{\mc G_k^\dagger}) = 1-\lambda_j(\mc P_k^\dagger)$ for all $j=0,1,\ldots,n_k-1$. In particular, by taking $j=1$, we can relate the spectral gap of $\mc P_k^\dagger$ with the second smallest eigenvalue of the normalized Laplacian matrix $L^N_{\mc G_k^\dagger}$ as $\gamma(\mc P_k^\dagger) = \lambda_1(L^N_{\mc G_k^\dagger})$.

It is known that the eigenvalues of the normalized Laplacian $L^N_{\mc G_k^\dagger}$ of the geometric random graph $\mc G_k^{\dagger}$ approaches (up to a scaling factor) the eigenvalues of the drift Laplace-Beltrami operator on $\cM_k$. More concretely, let $L^2(\cM_k, \rd \mu_k)$ be the space of all square integrable functions on $\cM_k$, and $\Delta_{\cM_k}$ denote the drift Laplace-Beltrami operator on $\cM_k$ defined in~\eqref{eqn:drift_Laplace-Beltrami} in~Section~\ref{subsec:Laplace-Beltrami_operator}. Let $0=\lambda_0(\cM_k)\leq \lambda_1(\cM_k)\leq\cdots$ denote the sequence of nonnegative eigenvalues of $\Delta_{\cM_k}$. The connectedness of $\cM_k$ implies that its second smallest eigenvalue $\lambda_1(\cM_k)$ is strictly positive. We will invoke Corollary 2 of \cite{trillos2018error}, which generalizes Theorem 1 of \cite{burago2014graph} from the uniform density to any Lipschitz continuous density satisfying~\eqref{Eqn:density_condition}, for relating the spectrum of the drift Laplace-Beltrami operator $\Delta_{\cM_k}$ on $\cM_k$ with the spectrum of the discrete normalized graph Laplacian $L^N_{\mc G_k^\dagger}$, as summarized in the following.

\begin{lem}[Convergence of eigenvalues of normalized graph Laplacian]\label{lem:eigenval_convergence_normalized_graph_Laplacian}
Let $\nu_{q_k}$ to denote the volume of an unit ball in $\bb R^{q_k}$. For each $j=0,1,\ldots$, suppose the radius $h$ and the $\varepsilon_{n,k}$ to be defined below satisfy $(\sqrt{\lambda_j(\cM_k)} + 1)\,h +\varepsilon_{n,k}/h \leq c_1$. Then for any $\xi\geq 1$, it holds with probability at least $1-c_2\,n_k^{-\xi}$
\begin{align*}
\bigg| \frac{2(q_k+2)}{\nu_{q_k} h^2}\cdot\frac{\lambda_j(L^N_{\mc G_k^\dagger})}{\lambda_j(\cM_k)} - 1 \bigg| \leq C\, \bigg(\frac{\varepsilon_{n,k}}{h}+(\sqrt{\lambda_j(\cM_k)} + 1)\,h+ h^2\bigg),
\end{align*}
where constants $c_1,c_2, c_3,C>0$ depend only on the submanifold $\cM_k$ through $\{q_k,{\rm Diam_k}, {\rm  Sec_k}$, ${\rm Inj_k}, {\rm Rch_k}\}$ and the density $p_k$ through $\{\eta_k,{\rm Lip_k}\}$, and
\begin{align*}
\varepsilon_{n,k} = \begin{cases}
\displaystyle\frac{(\log n_k)^{3/4}}{n_k^{1/2}} , & \quad\mbox{if } q_k=2\\[2ex]
\displaystyle \Big(\frac{\log n_k}{n_k}\Big)^{1/q_k}, & \quad\mbox{otherwise}
\end{cases}.
\end{align*}
\end{lem}
In particular, this lemma (taking $j=1$) implies a lower bound on the spectral gap of $\cP_k^\dagger$ as
\begin{align}\label{Eqn:RGG_spectral_lower_bound}
\gamma(\cP_k^\dagger) \geq \bigg(1- C\,\Big(c_1^{-1}+(\sqrt{\lambda_1(\cM_k)} + 1)\,h+ h^2\Big)\bigg) \,\frac{\nu_{q_k}}{2(q_k+2)} \, \lambda_1(\cM_k)\,h^2
\end{align}
as long as $h \geq c_1 \varepsilon_{n,k}$.

Next, we will apply the following comparison theorem to relate the spectral gaps of Markov chains $\mc W_{n,k}$ and $\mc W_{k}^\dagger$. 

\begin{lem}[Markov chains comparison theorem (Lemma 13.22 in \cite{levin2017markov})]
\label{lem:comparison_thm}
Let $P$ and $P'$ be transition matrices of two reversible Markov chains on the same state space $\Omega$, whose stationary distributions are denoted by $\pi$ and $\pi'$, respectively.
Let $\mc E(f)$ and $\mc E'(f)$ denote the Dirichlet forms associated to the pairs $(P,\pi)$ and $(P',\pi')$, where
\begin{align}\label{Eqn:Dirichlet_form}
\mc E(f) =\frac{1}{2} \sum_{x,y\in\Omega} [\,f(x) - f(y)]^2\,\pi(x)\, P(x,y),\quad \forall f\in L^2(\Omega),
\end{align}
and $\mc E'(f)$ can be similarly defined. If there exists some constant $B>0$ such that $\mc E'(f) \leq B \,\mc E(f)$ for all $f$, then
\begin{align*}
\gamma(P') \leq \bigg[\max_{x\in\Omega} \frac{\pi(x)}{\pi'(x)}\bigg] \, B\,\gamma(P),
\end{align*}
where $\gamma(P)$ and $\gamma(P')$ denote the spectral gaps associated with $P$ and $P'$, respectively.
\end{lem}

We will apply this comparison theorem with $P_{n,k}\to P, \cP_{k}^\dagger \to P'$, and $\Omega = G_{k}^{*}$. Let us find the constant $B$ such that $\mc E'(f)\leq B\, \mc E(f)$ for all $f$, where in our setting,
\begin{align*}
\mc E(f) &=\frac{1}{2 N_k} \sum_{1\leq i,j\leq n_k} (\,f_i - f_j)^2\,\kappa(X_i,X_j),\quad\mbox{and}\\
\mc E'(f) &=\frac{1}{2 d_k^\dagger} \sum_{(i,j):\, \|X_i-X_j\| \leq h} (\,f_i - f_j)^2,\quad \forall f=(f_1,\ldots,f_{n_k})^T\in\bR^{n_k}.
\end{align*}
According to the definition of $\tau$ as $\inf_{x,y\in S:\, \|x-y\| \leq h} \kappa(x,y)$, we can simply choose $B = (N_k/d_k^\dagger)\, \tau^{-1}$. In addition, we have the bound
\begin{align*}
\max_{i\in G_k^\ast} \frac{[\pi_{n,k}]_i}{[\pi_k^\dagger]_i} = \frac{d_k^\dagger}{\tilde N_k}\, \max_{i\in G_k^\ast}  \frac{d_{n,k}(X_i)}{d_{k,i}^\dagger}.
\end{align*}
Therefore, we can apply Lemma~\ref{lem:comparison_thm} to get
\begin{align*}
\gamma(P_{n,k}) \geq \tau \bigg[\min_{i\in G_k^\ast}  \frac{d_{k,i}^\dagger}{d_{n,k}(X_i)}\bigg]\, \gamma(\cP_{k}^\dagger),
\end{align*}
which combined with inequality~\eqref{Eqn:RGG_spectral_lower_bound} leads to
\begin{align*}
\gamma(P_{n,k}) \geq \bigg(1- C\,\Big(c_1^{-1}+(\sqrt{\lambda_1(\cM_k)} + 1)\,h+ h^2\Big)\bigg) \,\bigg[\min_{i\in G_k^\ast}  \frac{d_{k,i}^\dagger}{d_{n,k}(X_i)}\bigg]\,\frac{\tau\nu_{q_k}\, \lambda_1(\cM_k)\,h^2}{2(q_k+2)}.
\end{align*}

\subsection{Proof of Lemma~\ref{lem:node_degree}}\label{Sec:Proof_lem:node_degree}
Recall that $\kappa(x,y)=\exp\{-\|x-y\|^2/(2h^2)\}$ is the Gaussian kernel with bandwidth parameter $h>0$. For $x \in S$, define $d_{n,k}(x) = \sum_{j \in G_{k}^{*}} \kappa(x, X_{j})$ as the induced degree function of $x$ within cluster $G_{k}^{*}$. Then for each $i\in G_k^\ast$ we have $d_{n,k}(X_{i}) = 1+\sum_{j\in G_k^\ast,\,j\neq i} \kappa(X_i,X_j) =: 1+\tilde{d}_{n,k}(X_{i})$. Denote $\alpha_{k}(x) = \E_{X \sim p_{k}} \kappa(x,\, X)$ and $v_{k}(x) = \Var_{X \sim p_{k}}[\kappa(x,\, X)]$. Applying Lemma~\ref{lem:expectation_variance_bound} with $\mc M=\cM_k$ and $f(x)=p_k(x)$, we have $v_{k}(x) \leq C\alpha_{k}(x) \leq C h^{q_{k}}$ for all $x \in \cM_{k}$, where constant $C$ only depends on $\{q_k,{\rm Sec_k},{\rm Inj_k},\eta_k,{\rm Lip_k}\}$. Then for any fixed $x \in \cM_{k}$, using the bound in~\eqref{Eqn:variance_bound} and the boundedness of $\kappa$, we may apply Bernstein inequality (cf. Lemma 2.2.9 in \cite{vandervaartwellner1996}) to obtain that for all $u>0$,
\begin{align*}
\Prob \left( \Big|d_{n,k}(x) - n_k \alpha_{k}(x) \Big| \geq u \right) \leq 2 \exp\left( -{u^{2} \over 2 C n_{k} \alpha_{k}(x) + {2 \over 3} u} \right). 
\end{align*}
Choosing $u = t n_{k} \alpha_{k}(x)$ for $t \in (0, C]$, we have 
\begin{equation}
\label{Eqn:degree_con}
\Prob \left( \Big|d_{n,k}(x) - n_k \alpha_{k}(x) \Big| \geq t n_{k} \alpha_{k}(x) \right) \leq 2 \exp \left( -{t^{2} n_{k} \alpha_{k}(x) \over 2 C + {2 \over 3} t} \right) \le 2 \exp \left( -C n_{k} \alpha_{k}(x) t^{2} \right).
\end{equation}
Now choosing $t = c_{1} \sqrt{\log(n_{k})/(\alpha_{k}(x) n_{k})}$ for some large enough constant $c_{1}$ so that $Cc_1^2\geq \xi$, we get 
\[
\Prob \left( \Big|d_{n,k}(x) - n_k \alpha_{k}(x) \Big| \geq c_{1} \sqrt{\alpha_{k}(x) n_{k} \log{n_{k}}} \right) \leq 2 n_{k}^{-C c_{1}^{2}}\leq 2 n_k^{-\xi},
\]
provided that $\log(n_{k}) \leq C \alpha_{k}(x) n_{k} \leq C n_{k} h^{q_{k}}$ in view of the uniform bounds~\eqref{Eqn:expectation_bound} and~\eqref{Eqn:density_condition}. But this is ensured by our bandwidth assumption~\eqref{eqn:bandwidth_condition_nonadaptive}. Thus for any fixed $x \in \cM_{k}$, we have with probability at least $1 - c_{2} n_{k}^{-c_{3}}$, 
\[
\left| {d_{n,k}(x)  \over n_{k}} - \alpha_{k}(x) \right| \leq c_{1} \sqrt{\alpha_{k}(x) \log{n_{k}} \over n_{k}}.
\]
Then it follows that with probability at least $1 - c_{2} n_{k}^{-\xi}$, 
\[
\left| {d_{n,k}(x)  \over n_{k} (\sqrt{2\pi} h)^{q_{k}}} - p_{k}(x) \right| \leq {C \over (\sqrt{2\pi})^{q_{k}}} h  +  {c_{1} \over (\sqrt{2\pi} h)^{q_{k}}} \sqrt{\alpha_{k}(x) \log{n_{k}} \over n_{k}} \leq C \left( h + \sqrt{\log{n_{k}} \over n_{k} h^{q_{k}}} \right).
\]
This implies that the rescaled degree function $n_k^{-1}\,\big(\sqrt{2\pi} h\big)^{-q_k}\,d_{n,k}(x)$ provides a good estimate of the density $p_k(x)$ at $x$. Since $X_{i} \in \cM_{k}$ are i.i.d. for $i \in G_{k}^{*}$, we have with probability at least $1 - c_{2} n_{k}^{-\xi}$, 
\[
\left| {\tilde{d}_{n,k}(X_{i})  \over (n_{k}-1) (\sqrt{2\pi} h)^{q_{k}}} - p_{k}(X_{i}) \right| \leq C \left( h + \sqrt{\log{n_{k}} \over n_{k} h^{q_{k}}} \right).
\]
Then union bound implies that 
\[
\max_{i \in G_{k}^{*}} \left| {d_{n,k}(X_{i})  \over n_{k} (\sqrt{2\pi} h)^{q_{k}}} - p_{k}(X_{i}) \right| \leq C \left( h + \sqrt{\log{n_{k}} \over n_{k} h^{q_{k}}} + {1 \over n_{k} h^{q_{k}}} \right) \leq C \left( h + \sqrt{\log{n_{k}} \over n_{k} h^{q_{k}}} \right)
\]
with probability at least $1-c_{2}n_{k}^{-\xi}$. 

The second part about the concentration of $d_{k}^\dagger(X_{i})$ can be analogously proved by applying the Chernoff bound for sum of i.i.d.~Bernoulli random variables. Let 
\begin{align*}
d_{k}^\dagger(x) = \sum_{j \in G_{k}^{*}} 1(\|x-X_j\|\leq h)
\end{align*}
so that $d_{k}^\dagger(X_{i}) = 1+\tilde{d}_{k}^\dagger(X_{i})$ where $\tilde{d}_{k}^\dagger(X_{i}) = \sum_{j \in G_{k}^{*}, j\neq i} 1(\|X_{i}-X_j\|\leq h)$. Note that Section 2.2 of \cite{burago2014graph} provides a uniform estimate of the expectation $\bb E_{X\sim p_k} [1(\|x-X\|\leq h)]$ in terms of the density $p(x)$ as
\begin{align}\label{Eqn:Expected-degree}
\sup_{x \in \cM_{k}} \Big|\bb E_{X\sim p_k} [1(\|x-X\|\leq h)] - \nu_{q_k} \, h^{q_k}\,p_k(x) \Big| \leq C\, h^{q_k+2},
\end{align}
where recall that $\nu_{q_k}$ denotes the volume of unit ball in $\bR^{q_k}$. 
The rest of the proof follows a similar line as the proof of the first part, and we omit the details.


\begin{lem}\label{lem:expectation_variance_bound}
Let $\mc M$ be a $q$-dimensional compact submanifold in $\bR^p$ with bounded absolute sectional curvature ${\rm Sec}$ and injective radius ${\rm Inj}$, and $\mbox{Vol}_{\mc M}$ denote its volume form. Let $f$ be a ${\rm Lip}$-Lipschitz continuous probability density function on $\cM$ such that $\eta \leq f(x) \leq \eta^{-1}$ for some constant $\eta > 0$. Let $\alpha(x) = \E_{X \sim f} [\kappa(x,\, X)]$ and $v(x) = \Var_{X \sim f}[\kappa(x,\, X)]$. Then we have 
\begin{equation}
\label{Eqn:expectation_bound}
\sup_{x\in \cM}  \Big| \alpha(x)  - \big(\sqrt{2\pi} h\big)^{q} \,f(x)\Big| \leq C \, h^{q+1}
\end{equation}
and $v(x) \leq C \alpha(x)$ for all $x \in \cM$, where the constant $C$ only depends $\mc M$ through $\{q,{\rm Sec},{\rm Inj}\}$ and $f$ through $\{\eta,{\rm Lip}\}$. Consequently, we have $\sup_{x \in \cM} \alpha(x) \leq C h^{q}$ and $\sup_{x \in \cM} v(x) \leq C h^{q}$. 
\end{lem}

\begin{proof}[Proof of Lemma~\ref{lem:expectation_variance_bound}]
Note that for each $x\in \mc M$ and any Lipschitz probability density function $f$ on $\cM$, the expectation $\bb E[\kappa(x,\, X)^2]$ takes the form
\begin{align*}
\bb E_{X \sim f} [\kappa(x,\, X)] = \int_{\mc M} \exp\{-\|x-y\|^2/(2h^2)\} \, f(y)\, \rd\mbox{Vol}_{\cM} (y).
\end{align*} 
Then~\eqref{Eqn:expectation_bound} follows from Lemma~\ref{Lem:convolution_bound}. Similarly, we can bound the variance of $\kappa(x,\, X)$ via $\mbox{Var}_{X \sim f} [\kappa(x,\, X)] \leq \bb E_{X \sim f}  [\kappa(x,\, X)^2]$, where
\begin{align*}
\bb E_{X \sim f}  [\kappa(x,\, X)^2]= \int_{\mc M} \exp\{-\|x-y\|^2/(h^2)\} \, f(y)\, \rd\mbox{Vol}_{\cM} (y).
\end{align*}
By a second application of Lemma~\ref{Lem:convolution_bound} with $h/\sqrt{2} \to h$ to obtain
\begin{align*}
\Big| \bb E_{X \sim f}   [\kappa(x,\, X)^2] - \big(\sqrt{\pi} h\big)^{q} \,f(x)\Big| \leq C' \, h^{q+1}.
\end{align*}
This together with the uniform boundedness condition on $f$  and inequality~\eqref{Eqn:expectation_bound} imply an upper bound to the variance by the expectation,
\begin{align}\label{Eqn:variance_bound}
\mbox{Var}_{X \sim f}  [\kappa(x,\, X)]  \leq C \,h^{q}\leq C \,\bb E_{X \sim f} [\kappa(x,\, X)],\quad \mbox{for some $C>0$}.
\end{align}
The bounds $\sup_{x \in \cM} \alpha(x) \leq C h^{q}$ and $\sup_{x \in \cM} v(x) \leq C h^{q}$ follow from the fact that $h \leq c$. 
\end{proof}

\begin{lem}[Convolution over manifold]
\label{Lem:convolution_bound}
Let $\mc M$ be a $q$-dimensional compact submanifold in $\bR^p$ with bounded absolute sectional curvature ${\rm Sec}$ and injective radius ${\rm Inj}$, and $\mbox{Vol}_{\mc M}$ denote its volume form., and $\mbox{Vol}_{\mc M}$ denote its volume form. Then there exists some constant $h_0>0$ only depending on $\{{\rm Sec},{\rm Inj}\}$, such that for all $h\in(0, h_0]$ and any ${\rm Lip}$-Lipschitz continuous function $f$ on $\mc M$, we have 
\begin{align*}
\bigg| \frac{1}{\big(\sqrt{2\pi} h\big)^q}\,\int_{\mc M} \exp\{-\|x-y\|^2/(2h^2)\} \, f(y)\, \rd\mbox{Vol}_{\mc M} (y) - f(x)\bigg| \leq  C\,h,
\end{align*}
where constant $C$ only depends on $\{{\rm Sec},{\rm Inj},{\rm Lip}\}$.
\end{lem}

\begin{proof}[Proof of Lemma~\ref{Lem:convolution_bound}] 
This lemma follows from Lemma 5.2 on pages 895--898 in \cite{yang2016bayesian}.
\end{proof}

\subsection{Proof of Lemma~\ref{Lemma:total_degree}}\label{Sec:Proof_total_degree}
Recall that $N_k=\sum_{i,j\in G^\ast_k} \kappa(X_i,X_j)$ is the total within-weight in $G^\ast_k$. According to Lemma~\ref{lem:node_degree} (note that $N_k=\sum_{i\in G_k^\ast} d_{n,k}(X_i)$ using the notation therein), it holds with probability at least $1-c_2 n_k^{-\xi}$ that
\begin{align*}
\bigg| \frac{N_k}{n_k} - (\sqrt{2\pi}\, h)^{q_k}\, \sum_{i\in G_k^\ast} p_k(X_i)\bigg| \leq C\, n_{k} h^{q_k}\,\bigg(h + \sqrt{\frac{\log n_k}{n_kh^{q_{k}}}} \, \bigg).
\end{align*}
According to the sandwiched bound on the density function $p_k$, $\{p_k(X_i):\,i\in G_k^\ast\}$ are independent and bounded random variables. Therefore, we may apply Hoeffding's inequality to obtain that 
\begin{align*}
\bigg|\sum_{i\in G_k^\ast} p_k(X_i) - n_k \,\beta_k\bigg| \leq C\,\sqrt{n_k\log n_k}
\end{align*}
holds with probability at least $1-c_2 n_k^{-\xi}$, where $\beta_k=\mathbb E[p_k(X_i)]$. Combining the two preceding inequalities, we obtain that 
\begin{align*}
\bigg| \frac{N_k}{n_k} - (\sqrt{2\pi}\, h)^{q_k} \,n_k \beta_k\bigg| \leq C \left(  n_{k} h^{q_{k}+1} + \sqrt{n_kh^{q_k}\log n_k} \right),
\end{align*}
which completes the proof.

\subsection{Proof of Theorem~\ref{thm:main_adaptive_h}}\label{Sec:Proof_thm:main_adaptive_h}
\paragraph{\emph{Proof of Part (1):}} Consider $X_i$, where $i\in G_k^\ast$ for some $k\in[K]$.
Fix $h_U=c_U(\log n/n_k)^{1/q_k}$ and $h_L=c_L(\log n/n_k)^{1/q_k}$ for two sufficiently large constant $c_U$ and $c_L$ with $c_U=2c_L$. We use notation $N(X_i,h)$ to denote the number of points from $\{X_i\}_{i=1}^n$ that is within $h$ distance from $X_i$.

Recall that $d_k^\dagger(X_i) := d_{k,h}^\dagger(X_i)$ in Lemma~\ref{lem:node_degree} is the number of points from $\{X_i\}_{i\in G_k^\ast}$ that is within $h$ distance to $X_i$ (we will choose $h=h_L$ and $h=h_U$ later). Here we put an subscript $h$ in $d_{k,h}^\dagger(X_i)$ to indicate the dependence of $d_k^\dagger$ on $h$. The condition on $\delta_{kk'}$ implies that any point outside the $k$th cluster $D_k$ has distance at least $C'(\log n/n_k)^{1/q_k} \geq \max\{h_L,h_U\}$. Therefore all points that are within $h_L(h_U)$ distance to $X_i$ must belong to $\mathcal M_k$, implying $N(X_i,h) = d_{k,h}^\dagger(X_i)$. 
Consequently, from the proof of Lemma~\ref{lem:node_degree} (take $h=h_U$ and $h=h_L$ respectively), we have (a concentration inequality as~\eqref{Eqn:degree_con} plus the expectation bound~\eqref{Eqn:Expected-degree}),
\begin{align*}
&\mathbb P \big( N(x,h_L) \leq c_1 (1-t)\, n_k h_L^{q_k} \big) \leq \exp\{- c_1' n_k h_L^{q_k} t^2\},\\
&\mathbb P \big( N(x,h_U) \geq c_2 (1+t)\, n_k h_U^{q_k} \big) \leq \exp\{- c_2' n_k h_U^{q_k} t^2\},\quad t>0, x \in \cM_{k}, 
\end{align*}
where constants $(c_1,c_1')$ and $(c_2,c_2')$ only depend on the constant $\eta_k$ in the two-sided bound on density $p_k(\cdot)$ on the $k$th region $\mathcal M_k$.
Let $c_U$ be sufficiently large such that $c_2 c_U^{q_k} = C$, where $C$ is the constant appearing in the neighborhood parameter $k_0=\lfloor C\log n\rfloor$. By taking $t$ such that $(1-t) c_1 c_L^{q_k} = C/2$ in the first inequality of the preceding display, and $t=1$ in the second inequality, we obtain that  
\begin{align*}
&\mathbb P \big( N(X_{i},h_L) \leq k_0/2 \big) \leq \exp\{- c_1'' \,C \log n\},\\
&\mathbb P \big( N(X_{i},h_U) \geq 2k_0 \big) \leq \exp\{- c_2''\, C\log n\},
\end{align*}
where $c_1''$ and $c_2''$ are two constants independent of $C$. For large enough $C$, we can make these two probabilities smaller than $n^{-2-\xi}$. Since the event $\{N(X_i,h_L) \leq k_0/2\}\cap\{N(X_i,h_U) \geq 2k_0\}$ implies $h_i \in[h_L,h_U]$ for each $i\in[n]$, a union bound argument over $i=1,2,\ldots,n$ leads to the claimed two sided bound for $h_i$ (with probability at least $1-n^{-\xi}$).

\smallskip

\paragraph{\emph{Proof of Part (2):}} 
Using the two sided bound in the proof of Part (1), the proof follows same steps in the proof of Theorem~\ref{thm:main}, with the only exception in proving a counterpart of Lemma~\ref{Lem:T_2} for bounding the spectral gap $\gamma(P_{n,k})$ of chain $P_{n,k}$ on $\{X_i\}_{i\in G_k^\ast}$ in cluster $G_k^\ast$ for $k\in [K]$, since $P_{n,k}$ now has different bandwidth parameter $h_i$ at each observed point $X_i$ (that is, a bandwidth parameter inhomogeneous chain). More precisely, it remains to show that with high probability, it holds that
\begin{align*}
\gamma(P_{n,k})\geq C' \lambda_1(\mathcal M_k) \,(\log n /n_k)^{2/q_k}, \quad \mbox{for some constant $C'>0$.}
\end{align*}
Here, we assume without loss of generality that the absolute spectral gap of $P_{n,k}$ is dominated by one minus its second largest eigenvalue. Otherwise, we can always consider the lazy random walk by replacing $P_n$ with $P_n/2 +I_n/2$ in the diffusion $K$-mean SDP, whose absolute spectral gap is $\gamma(P_{n,k})/2$. 

Our proof strategy is again based on the Markov chains comparison theorem (Lemma~\ref{lem:comparison_thm}) by comparing this bandwidth parameter inhomogeneous chain with a bandwidth parameter homogeneous chain with $h_i \equiv h_L$, for each $i\in G_k^\ast$, where $h_L$ is the lower bound of $h_i$ in the proof of Part (1). In particular, a lower bound on the spectral gap of the latter is already derived in Lemma~\ref{Lem:T_2} as of order $\lambda_1(\mathcal M_k) h_L^2$.

Fix the cluster index $k\in [K]$, and without loss of generality assume $G_k^\ast=\{1,2,\ldots,n_k\}$. To avoid confusion of notation, we put an superscript ``IH" 
indicate the bandwidth parameter inhomogeneous chain, and ``H" to denote the bandwidth parameter homogeneous chain with $h_i \equiv h_L$. For example, $P_{n,k}^{IH}$ and $\mathcal E^{IH}$ denote the transition probability matrix and the Dirichlet form~\eqref{Eqn:Dirichlet_form} (defined in Lemma~\ref{lem:comparison_thm}), respectively, associated with the bandwidth parameter inhomogeneous chain.

Due to the fact that $h_i \geq h_L$ for all $i\in G^\ast_k$, we immediately have 
\begin{align*}
2 N_k^H \mathcal E^{H}(f) &= \sum_{1\leq i,j\leq n_k} (f_i-f_j)^2 \kappa^{H}(X_i,X_j)  \\
&\leq \sum_{1\leq i,j\leq n_k} (f_i-f_j)^2 \kappa^{IH}(X_i,X_j) = 2 N_k^{IH}\mathcal E^{IH}(f), \quad\mbox{for all $f\in\mathbb R^{n_k}$,}
\end{align*}
where $k^{H}(X_i,X_j) = \exp\{-\|X_i-X_j\|^2/(2h_L^2)\}$ and $\kappa^{IH}(X_i,X_j) = \exp\{-\|X_i-X_j\|^2/(2h_ih_j)\}$, and $(N_k^H,N_k^{IH})$ are the respective total degrees within $G_k^\ast$.
Moreover, recall that the stationary distributions of $P_{n,k}^{IH}$ and $P_{n,k}^{H}$ are $\pi_{n,k}^{IH} = d_{n,k}^{IH}(X_i) / N_k^{IH}$ and $\pi_{n,k}^{H} = d_{n,k}^{H}(X_i) / N_k^{H}$, for $i\in G_k^\ast$, respectively, where $d_{n,k}^{IH}(X_i) =\sum_{j\in G_k^\ast} \kappa^{IH}(X_i,X_j)$, $d_{n,k}^{H}(X_i) =\sum_{j\in G_k^\ast} \kappa^{H}(X_i,X_j)$ are the node degrees, and $N_k^{IH}= \sum_{i\in G_k^\ast} d_{n,k}^{IH}(X_i)$, $N_k^{H}= \sum_{i\in G_k^\ast} d_{n,k}^{H}(X_i)$ are the total degrees.

Now we can apply Lemma~\ref{lem:comparison_thm} with $B= N_k^{IH}/ N_k^{H}$ and Lemma~\ref{Lem:T_2} (to the homogeneous chain $P_{n,k}^H$) to obtain
\begin{align*}
\gamma(P_{n,k}^{IH}) \geq \min_{i\in G_k^\ast} \bigg[\frac{d_{n,k}^{H}(X_i)}{d_{n,k}^{IH}(X_i)} \bigg]\, \gamma(P_{n,k}^{H}) \quad\mbox{and}\quad \gamma(P_{n,k}^{H})\geq C \lambda_1(\mathcal M_k) \,h_L^2,
\end{align*}
for some constant $C>0$.
Similar to the proof of Lemma~\ref{lem:node_degree}, we can apply the concentration inequality to the nodes degree with bandwidth $h=h_U$ and $h=h_L$ to obtain that with probability at least $1-c_2\, n^{-\xi}$, 
\begin{align*}
\max_{i\in G_k^\ast} \bigg|\,\frac{\sum_{j\in G_k^\ast}\exp\{-\|X_i-X_j\|^2/(2h_L^2)\} }{n_k(\sqrt{2\pi} h_L)^{q_k}}  - p_k(X_i) \bigg| \leq C'\bigg( h_L + \sqrt{\frac{\log n_k}{n_k h_{L}^{q_{k}}}}\,\bigg),\\
\max_{i\in G_k^\ast} \bigg|\,\frac{\sum_{j\in G_k^\ast}\exp\{-\|X_i-X_j\|^2/(2h_U^2)\} }{n_k(\sqrt{2\pi} h_U)^{q_k}}  - p_k(X_i) \bigg| \leq C'\bigg( h_U+ \sqrt{\frac{\log n_k}{n_k h_{U}^{q_{k}}}}\,\bigg), 
\end{align*}
for all $i\in G_k^\ast$. Combining this with the sandwiched bound for $h_i$ in Part (1), we obtain
\begin{align*}
c_1\, n_k h_L^{q_k}  &\leq  d_{n,k}^{H}(X_i)=\sum_{j\in G_k^\ast}\exp\{-\|X_i-X_j\|^2/(2h_L^2)\}  \\
&\leq d_{n,k}^{IH}(X_i)  = \sum_{j\in G_k^\ast}  \exp\{-\|X_i-X_j\|^2/(2h_ih_j)\} \\
&\leq \sum_{j\in G_k^\ast} \exp\{-\|X_i-X_j\|^2/(2h_U^2)\} \leq c_1'\, n_k h_U^{q_k},\ \ \ \mbox{for all $i\in G_k^\ast$.}
\end{align*}

Putting all pieces together, we obtain that it holds with probability at least $1-c_2\, n^{-\xi}$ that
\begin{align*}
\gamma(P_{n,k}^{IH})\geq C' \lambda_1(\mathcal M_k) \,(\log n /n_k)^{2/q_k}, \quad \mbox{for some constant $C'>0$.}
\end{align*}

\subsection{Proof of Lemma \ref{lem:feasibility_SDP_lambda_infinity}}
\emph{Proof of Part (1):}
Since $n\rho >\lambda_{\max}(A)$, the matrix $n\rho I_n - A$ is positive definite.
For any $Z \in \sC$, from the constraint $Z \vone_{n} = \vone_{n}$, we see that $(1, n^{-1/2} \vone_{n})$ is an eigen-pair of $Z$. In addition, since $Z \succeq 0$, all eigenvalues $\lambda_{1},\dots,\lambda_{n}$ of $Z$ are non-negative. Let $U_1=n^{-1/2}\vone_{n}, U_2,\ldots, U_n$ denote the corresponding eigenvectors of $Z$.
Thus the objective function
\begin{align*}
&\langle A ,Z\rangle - n\rho \tr(Z) =  - \langle n\rho I_n - A, Z\rangle \\
&= -\frac{1}{n}\,\vone_{n}^T(n\rho I_n - A)\vone_{n} - \sum_{i=2}^n \lambda_i\, U_i^T(n\rho I_n - A)U_i \leq -\frac{1}{n}\vone_{n}^T(n\rho I_n - A)\vone_{n},
\end{align*}
where the equality holds if and only if $\lambda_2=\cdots=\lambda_n=0$. Note that $Z^{\diamond} \in \sC$ is a feasible solution for~\eqref{eqn:clustering_Kmeans_sdp_unknown_K} and $Z^{\diamond}$ has a non-zero eigenvalue equal to $1$ and $(n-1)$ zero eigenvalues. Therefore, $Z^{\diamond}$ is the unique solution of the SDP~\eqref{eqn:clustering_Kmeans_sdp_unknown_K}.

\smallskip
\noindent \emph{Proof of Part (2):}
For any $Z\in \sC$, since $Z$ is a symmetric matrix satisfying $Z \geq 0$ and $Z  \vone_{n} = \vone_{n}$, $Z$ is a stochastic matrix and all its eigenvalues have absolute values less than or equal to one. Moreover, from the positive semi-definiteness of $Z$, all eigenvalues of $Z$ lie in the $[0,1]$ interval. Now since $n\rho <\lambda_{\min}(A)$, the matrix $A- n\rho I_n$ is positive definite. From matrix H\"older's inequality, the objective function satisfies
\begin{align*}
&\langle A ,Z\rangle - n\rho \tr(Z) =  \langle A- n\rho I_n, Z\rangle \\
&\leq \matnorm{A- n\rho I_n}_{\ast}\, \matnorm{Z}_{\mbox{\scriptsize op}} = \matnorm{A- n\rho I_n}_{\ast},
\end{align*}
where the equality holds if and only if all eigenvalues of $Z$ equal to one (since $A- n\rho I_n$ is strictly positive definite). Note that $Z^{\dagger} = I_{n} \in \sC$ is a feasible solution for~\eqref{eqn:clustering_Kmeans_sdp_unknown_K}. Therefore, $Z^{\dagger}$ is the unique solution of the SDP~\eqref{eqn:clustering_Kmeans_sdp_unknown_K}.

\smallskip
\noindent \emph{Proof of Part (3):}
By the optimality and feasibility of the solutions $\tilde Z_{\rho_1}$ and $\tilde Z_{\rho_2}$, we have
\begin{equation}\label{Eqn:optimality}
\begin{aligned}
\langle A ,\tilde Z_{\rho_2}\rangle - \rho_1 \tr (\tilde Z_{\rho_2}) &\leq \langle A ,\tilde Z_{\rho_1}\rangle - \rho_1 \tr (\tilde Z_{\rho_1}),\ \ \mbox{and}\\
\langle A ,\tilde Z_{\rho_1}\rangle - \rho_2 \tr (\tilde Z_{\rho_1}) &\leq \langle A ,\tilde Z_{\rho_2}\rangle - \rho_2 \tr (\tilde Z_{\rho_2}).
\end{aligned}
\end{equation}
Adding these two inequalities together yields
\begin{align*}
(\rho_1-\rho_2) \big( \tr (\tilde Z_{\rho_1}) -  \tr (\tilde Z_{\rho_2})\big) \leq 0,
\end{align*}
which implies $\tr (\tilde Z_{\rho_1}) \geq \tr (\tilde Z_{\rho_2})$ when $\rho_1> \rho_2$. Moreover, if at least one of the SDPs has a unique solution, then at least one of the two inequalities in~\eqref{Eqn:optimality} is strict, implying 
$$
(\rho_1-\rho_2) \big( \tr (\tilde Z_{\rho_1}) -  \tr (\tilde Z_{\rho_2})\big) < 0,$$
and $\tr (\tilde Z_{\rho_1}) > \tr (\tilde Z_{\rho_2})$.


\appendix

\section{Proofs on spectral decompositions}\label{app:A}

\begin{proof}[Proof of Lemma \ref{lem:spectral_decomposition_Markov_chain}]
Since $\kappa$ is symmetric and positive semidefinite, so is $\mathscr R$. Thus the corresponding operator $\mathscr R$ is self-adjoint in $L^2(\rd\mu)$ and is also compact if \eqref{eqn:kernel_integrability_condition} holds. Then $\mathscr R$ has a discrete set of nonnegative eigenvalues $\lambda_0\geq \lambda_1\geq\cdots \geq 0$, and has the following eigen-decomposition
\begin{align*}
\mathscr R(x,y) = \sum_{j=0}^\infty \lambda_j\, \phi_j(x)\,\phi_j(y),\quad\forall x,y\in S,
\end{align*}
where $\{\phi_j\}_{j=0}^\infty$ is an orthonormal basis of $L^2(\rd\mu)$. Denote $Z = \int_{S} d(y) \, \rd \mu(y)$ so that $\pi(x) = d(x) / Z$. Since
\begin{align*}
\int_{S} \mathscr R(x,y)\sqrt{\pi(y)} \, \rd \mu(y) &= \int_{S} {\kappa(x,y) \over \sqrt{d(x)}\,\sqrt{d(y)}} \sqrt{\pi(y)} \, \rd \mu(y) \\
&= {1 \over Z \sqrt{\pi(x)}} \int_{S} \kappa(x,y) \, \rd \mu(y) = {d(x) \over Z \sqrt{\pi(x)}} = \sqrt{\pi(x)},
\end{align*}
it follows that $\lambda_{0}=1$ is the largest eigenvalue of $\mathscr R$ with the eigenfunction $\phi_{0}(x) = \sqrt{\pi(x)}$ (of unit length) in $L^{2}(\rd \mu)$. In addition, observe that
\begin{align*}
\mathscr R(x,y) = \sqrt{\frac{d(x)}{d(y)}}\,p(x,y), \quad\forall x,y\in S.
\end{align*}
This implies a decomposition of the transition probability $p(x,y)$ as
\begin{align*}
p(x,y) = \sum_{j=0}^\infty \lambda_j\, \psi_j(x)\,\widetilde\psi_j(y),\quad\forall x,y\in S,
\end{align*}
where $\psi_j(x) = \phi_j(x)/\sqrt{d(x)}$ and $\widetilde\psi_j(y) = \phi_j(y)\sqrt{d(y)}$. In particular, for each $j=0,1,\ldots$,
\begin{align*}
\mathscr P\psi_j(x) &= \sum_{l=0}^\infty \lambda_l \psi_l(x)\, \int_S \widetilde\psi_l(y)\, \psi_j(y)\,\rd\mu(y)=\sum_{l=0}^\infty \lambda_l \psi_l(x)\,\delta_{lj} = \lambda_j \psi_j(x),\quad\forall x\in S,
\end{align*}
implying that $\{\psi_j\}_{j=0}^\infty$ are the corresponding (right) eigenfunctions of $\mathscr P$, with unit $L^2(d \rd\mu)$ norm, associated with the same eigenvalues $1=\lambda_0\geq \lambda_1\geq\cdots \geq 0$. Since $\mathscr P$ is the transition operator of a Markov chain, $\lambda_0=1$ and $\psi_0\equiv Z^{-1/2}$. 
\end{proof}

\begin{proof}[Proof of Lemma \ref{lem:spectral_representation_diffusion_distances}]
For $t\in\bN_{+}$ and $x,y\in S$, let $p_{t}(x,y)$ be the $t$-step transition probability from $x$ to $y$. By Lemma \ref{lem:spectral_decomposition_Markov_chain}, we have 
\begin{align*}
p_t(x,y) = \sum_{j=0}^\infty \lambda_j^t\, \psi_j(x)\,\widetilde\psi_j(y) 
\end{align*}
and $\{\widetilde\psi_j\}_{j=0}^\infty$ forms an orthonormal basis of $L^2(\rd\mu/d)$. 
Consequently, by viewing $\lambda_j^t\, \psi_j(x)$ as the coefficient associated with $\widetilde\psi_j$ in the orthogonal expansion of function $p_t(x,\cdot)$, we have
\begin{align*}
\mathscr D_{t}^2(x, y) =\|\, p_{t}(x, \cdot) - p_{t}(y, \cdot)\,\|^2_{L^{2}(\rd\mu/d)}= \sum_{j=0}^{\infty} \lambda_{j}^{2t} \, [\psi_{j}(x) - \psi_{j}(y)]^{2}. 
\end{align*}
\end{proof}

\section{Empirical diffusion affinity}\label{app:B}
Similar to the derivations in Appendix~\ref{app:A}, if we define matrix $R_n\in\bR^{n\times n}$ with 
\begin{align*}
[R_n]_{ij}=\frac{\kappa(X_i,X_j)}{\sqrt{d_n(X_i)}\sqrt{d_n(X_j)}},
\end{align*} 
then $\{\lambda_{n,j}\}_{j=0}^{n-1}$ are also the eigenvalues of $R_n$. Let $\phi_{n,j}\in\bR^n$ denote the unit Euclidean norm eigenvector associated with $\lambda_{n,j}$. Then the empirical probability transition matrix $P_n$ has the decomposition
\begin{align*}
P_{n}^t = \sum_{j=0}^{n-1} \lambda_{n,j}^t\, \psi_{n,j}\,\widetilde\psi_{n,j}^T,
\end{align*}
where $\psi_{n,j} = D_n^{-1/2}\,\phi_{n,j}\in\bR^n$ and $\widetilde\psi_{n,j} = D_n^{1/2}\phi_{n,j}\in\bR^n$, so that $\psi_{n,j}=D_n^{-1} \widetilde\psi_{n,j}$ for each $j\in\{0,1,\dots,n-1\}$.
In particular, $\psi_{n,j}$ has unit $L^2(\mbox{diag}(D_n))$ norm, and $\widetilde\psi_{n,j}$ has unit $L^2(\mbox{diag}(D_n^{-1}))$ norm, for each $j\in\{0,1,\dots,n-1\}$.

In addition, we have the following relation between the diffusion affinity and $P_n^{2t}$,
\begin{align*}
\langle X_i,\, X_j\rangle_{\mathscr D_{n,t}} = \sum_{l=0}^{n-1} \lambda_{n,l}^{2t} \,[\psi_{n,l}]_i \, [\psi_{n,l}]_j =\sum_{l=0}^{n-1} \lambda_{n,l}^{2t} \, [\psi_{n,l}]_i \, [\widetilde\psi_{n,l}]_j d_n^{-1}(X_j)
=[P_n^{2t}D_n^{-1}]_{ij}.
\end{align*}

\section{Technical lemmas}
In this appendix, we collect some technical lemmas used in the proofs of our main results.

\begin{lem}
\label{lem:some_ineq_feasible_set}
Let $Z^{*}$ be defined in (\ref{eqn:Kmeans_true_membership_matrix}). Then for any $Z \in \sC$ defined in \eqref{eqn:clustering_Kmeans_sdp_unknown_K}, we have 
\begin{align}
\label{eqn:ineq_1_feasible_set}
\|Z^{*} - Z^{*} Z Z^{*}\|_{1} = \|Z^{*} - Z^{*} Z\|_{1} =&\, 2\sum_{k=1}^K \sum_{m\neq k} \|Z_{G_{k}^{*} G_{m}^{*}}\|_{1}.
\end{align}
In addition, if $Z$ also satisfies $\tr(Z)=\tr(Z^\ast)$, or
$Z\in \sC_K$, where $\sC_K$ is defined in \eqref{eqn:clustering_Kmeans_sdp}, then
\begin{align}
\label{eqn:ineq_2_feasible_set}
\matnorm{(I-Z^{*}) Z (I-Z^{*})}_{\ast} = &\, \sum_{k=1}^K \sum_{m\neq k} \frac{1}{n_k}\, \|Z_{G_{k}^{*} G_{m}^{*}}\|_{1} \leq {\|Z^{*} - Z^{*} Z\|_{1} \over 2 \underline{n}}, \\
\label{eqn:ineq_3_feasible_set}
\|Z^{*} - Z^{*} Z\|_{1} \leq \|Z^{*} - Z\|_{1} \leq& \, n\,\matnorm{(I-Z^{*}) Z (I-Z^{*})}_{\ast} \leq {2 n \over \underline{n}} \|Z^{*} - Z^{*} Z\|_{1}.
\end{align}
\end{lem}

\begin{proof}[Proof of Lemma \ref{lem:some_ineq_feasible_set}]
Inequalities~\eqref{eqn:ineq_1_feasible_set} and~\eqref{eqn:ineq_3_feasible_set}  follow from Lemma 1 in \cite{GiraudVerzelen2018}. Inequality~\eqref{eqn:ineq_2_feasible_set} is due to inequality~(57) in \cite{BuneaGiraudRoyerVerzelen2016}. 
\end{proof}

\begin{lem}
\label{lem:some_ineq_feasible_set_adaptive}
Let $Z^{*}$ be defined in (\ref{eqn:Kmeans_true_membership_matrix}). Then for any $Z \in \sC$ defined in (\ref{eqn:clustering_Kmeans_sdp_unknown_K}), we have
\begin{align*}
\matnorm{(I-Z^{*}) Z (I-Z^{*})}_{\ast}\leq&\, \sum_{k=1}^K \sum_{m\neq k} \frac{1}{n_k}\, \|Z_{G_{k}^{*} G_{m}^{*}}\|_{1} + \tr(Z) - \tr(Z^\ast),\\
\|Z^{*} - Z\|_{1} \leq& \,4n\,\sum_{k=1}^K \sum_{m\neq k} \frac{1}{n_k}\, \|Z_{G_{k}^{*} G_{m}^{*}}\|_{1}+ n \big(\tr(Z) - \tr(Z^\ast)\big),
\end{align*}
and~\eqref{eqn:ineq_1_feasible_set} holds for $Z \in \sC$. 
\end{lem}
\begin{proof}[Proof of Lemma \ref{lem:some_ineq_feasible_set_adaptive}]
The first inequality follows from inequality~(57) in \cite{BuneaGiraudRoyerVerzelen2016}. The second one follows from the first, ~\eqref{eqn:ineq_1_feasible_set}, and the following decomposition,
\begin{align*}
Z - Z^\ast  = (I-Z^{*}) Z (I-Z^{*})  + (Z^\ast Z - Z^\ast) + (Z Z^\ast - Z^\ast) + (Z^\ast - Z^\ast ZZ^\ast),
\end{align*}
with inequality $\|U\|_1 \leq n \matnorm{U}_{\mbox{\scriptsize op}} \leq n \matnorm{U}_{\ast}$ for any matrix $U\in \bR^{n\times n}$. 
\end{proof}

\bibliographystyle{plain}
\bibliography{clustering_sdp}

\end{document}